\documentclass[final,1p]{elsarticle}

\usepackage{amssymb}
\usepackage{amsmath}
\usepackage{amsthm}
\usepackage{lineno}
\usepackage{latexsym}
\usepackage{cases}
\usepackage{bm}
\usepackage[ruled,lined]{algorithm2e}
\usepackage{algorithmic}
\usepackage{array}
\usepackage[colorlinks,linkcolor=blue, anchorcolor=blue, citecolor=blue]{hyperref}
\usepackage{color,tabularx}
\usepackage{booktabs}
\usepackage{subfigure}
\usepackage{multirow}
\usepackage{stmaryrd}
\usepackage[T1]{fontenc}
\biboptions{numbers,sort&compress}
\usepackage{graphicx} 
\usepackage{adjustbox}
\usepackage{makecell}
\allowdisplaybreaks

\newtheorem{theorem}{Theorem}[section]

\newtheorem{lemma}[theorem]{Lemma}

\theoremstyle{definition}

\newtheorem{example}[theorem]{Example}

\theoremstyle{remark}
\newtheorem{remark}[theorem]{Remark}

\newcommand{\f}{\frac}

\newcommand{\p}{\partial}

\newcommand{\mbb}{\mathbb}

\newcommand{\mbX}{\mathbb{X}}
\newcommand{\mrR}{\mathrm{R}}

\newcommand{\mrM}{\mathrm{ M }}
\numberwithin{equation}{section}

\begin{document}    
  \begin{frontmatter}

  \title{A linear, decoupled and positivity-preserving time-staggered block-centered finite difference method for the multi-species Keller–Segel chemotaxis system}

  \author[OUC]{Ao Zhang} \ead{zhangao6290@stu.ouc.edu.cn}
  \author[OUC]{Bingyin Zhang} \ead{zhangbingyin@stu.ouc.edu.cn}
  \author[OUC,LAB]{Hongfei Fu\corref{Fu}}\ead{fhf@ouc.edu.cn}

  \address[OUC]{School of Mathematical Sciences, Ocean University of China, Qingdao 266100, P.R. China}
  \address[LAB]{Laboratory of Marine Mathematics, Ocean University of China, Qingdao 266100, P.R. China}

\cortext[Fu]{Corresponding author.}

\begin{abstract}
In this paper, we present a linearly implicit, second-order block-centered finite difference (BCFD) prediction-then-projection scheme for the multi-species Keller–Segel chemotaxis system on non-uniform spatio-temporal grids. The proposed scheme integrates a standard Crank-Nicolson time-marching algorithm with an $L^2$ projection step to enforce positivity and mass conservation. The use of variable time stepsize and time-staggered discretization fully decouples the solutions of the multi-species cell density variables and the chemoattractant concentration variable while facilitating linearization, thereby greatly enhancing computational efficiency. Notably, the variable time-stepping algorithm and non-uniform grid BCFD discretization jointly enable adaptive resolution and local refinement near blow-up, thereby improving efficiency and accuracy without compromising the desired physical property-preserving in the simulation.
Furthermore, using the mathematical induction method and the energy analysis approach, the unique solvability of the proposed scheme is rigorously proved, and we show that cell densities achieve second-order convergence in both time and space in the discrete $L^2$ norm, while the chemoattractant concentration achieves second-order convergence in the discrete $H^1$ norm. Representative numerical experiments are presented to validate the theoretical findings and demonstrate the reliability of the proposed scheme in simulating the blow-up phenomenon.
\end{abstract}


\end{frontmatter}

\section{Introduction}
In the 1970s, Keller and Segel \cite{KS'70,KS'71} established a pioneering mathematical framework for chemotaxis. They formulated a system of nonlinear partial differential equations to represent the essential biological mechanism, in which cellular or organismal movement is directed by chemical cues that can be attractive or repulsive. Mathematically, the multi-species  ($d$-species) Keller--Segel chemotaxis model is to find the cell (or organism) density functions $\rho_i(\bm{x},t) ~ (i=1,\ldots,d)$ and the chemoattractant concentration function $c(\bm{x},t) $ such that 
\begin{equation}\label{model:dsKS}
\left\{
	\begin{aligned}
		&\p_t \rho_i  = \kappa_i\Delta \rho_i-\chi_i \nabla \cdot(\rho_i \nabla c),  & \qquad \text{in}~~ \Omega \times(0, T], \\
		&\p_t c  = \beta \Delta c- \alpha c+ \sum_{i=1}^d \gamma_i \rho_i,   & \qquad \text{in}~~ \Omega \times(0, T].
	\end{aligned}
    \right.
\end{equation}
Here $\Omega \subset \mrR^2$ is assumed to be a two-dimensional convex, bounded and open domain. 
The parameters $\kappa_i ~(i=1,\ldots,d)$ and $\beta$ are positive diffusion coefficients, $\chi_i>0~ (i=1,\ldots,d)$ is the chemoattractant sensitivity constant,  $\alpha \geq 0$ is the consumption rate of chemoattractant, and $\gamma_i \geq 0 ~ (i=1,\ldots, d)$ represents the production rate of chemoattractant. 

Without loss of generality, we only consider the two-species Keller--Segel chemotaxis model \eqref{model:dsKS}, which involves identifying three real functions $u = u(\bm{x}, t)$, $v = v(\bm{x}, t)$ and $c = c(\bm{x}, t)$ such that
\begin{equation}\label{model:KS}
\left\{
	\begin{aligned}
		&  \p_t u  = \kappa_1\Delta u-\chi_1 \nabla \cdot(u \nabla c),      & \qquad \text{in}~~ \Omega \times(0, T],\\
      &\p_t v  = \kappa_2\Delta v-\chi_2 \nabla \cdot(v \nabla c),      & \qquad \text{in}~~ \Omega \times(0, T],\\
    & \p_t c  = \beta \Delta c- \alpha c+  \gamma_1 u  +\gamma_2  v, & \qquad \text{in}~~ \Omega \times(0, T],
	\end{aligned}
    \right.
\end{equation}
subject to homogeneous Neumann boundary conditions
\begin{equation}\label{model:KS:bc}
\begin{aligned}
	&\p_{\bm{n}} u  :=\nabla u \cdot \bm{n}=0,\quad\p_{\bm{n}} v =0, \quad \p_{\bm{n}} c=0, & & \text { on } \p \Omega \times(0, T], 
\end{aligned}
\end{equation}
and initial conditions
\begin{equation}\label{model:KS:ic}
\begin{aligned}
	&u(\bm{x},0)  =u^0(\bm{x}), \quad v(\bm{x},0)  =v^0(\bm{x}), \quad c(\bm{x},0)=c^0(\bm{x}),  & & \text { in } \Omega,
\end{aligned}
\end{equation}
where $\bm{n}$ represents the unit outer normal vector onto the boundary.

Significantly, the Keller--Segel chemotaxis system \eqref{model:KS}--\eqref{model:KS:ic} obeys the mass conservation law, 
i.e.,
\begin{equation}\label{model:KS:MC}
\begin{aligned}
M[u](t) &:= \int_\Omega u(\bm{x}, t)  d\bm{x} = \int_\Omega u^0(\bm{x})  d\bm{x}=M[u](0),\\
M[v](t) &:= \int_\Omega v(\bm{x}, t)  d\bm{x} = \int_\Omega v^0(\bm{x})  d\bm{x}=M[v](0).
\end{aligned}
\end{equation}
Besides, for non-negative regular initial data \eqref{model:KS:ic}, i.e., $u^0(\bm{x})\ge 0$, $v^0(\bm{x})\ge 0$ and $c^0(\bm{x})\ge 0$, the Keller--Segel chemotaxis system admits unique solutions with non-negative cell density and chemoattractant concentration, i.e.,
 \begin{equation}\label{model:KS:positivity}
 	\begin{aligned}
 		u(\bm{x},t) \geq 0, \quad v(\bm{x},t)  \geq 0, \quad c(\bm{x},t)  \geq 0, \quad \text { in } \Omega.
 	\end{aligned}
 \end{equation}
Moreover, the Keller--Segel chemotaxis system can be viewed as a Wasserstein gradient flow driven by the total free energy 
\begin{equation}\label{def:Energy}
	\begin{aligned}
		E[u,v,c] := \int_\Omega \left[\f{\gamma_1 \kappa_1}{\chi_1}f(u) + \f{\gamma_2 \kappa_2}{\chi_2}f(v) -\gamma_1 u c -\gamma_2 v c +\f{\beta}{2} |\nabla c|^2 + \f{\alpha}{2} c^2 \right] d\bm{x},
	\end{aligned}
\end{equation}
where $f(\rho) := \rho \log(\rho) - \rho$ with $\rho \in (0,+\infty)$ for $\rho=u, v$. It is easy to verify that the two-species Keller--Segel model is energy dissipative, i.e.,
\begin{equation}\label{model:KS:Energy}
	\begin{aligned}
		\f{dE[u,v,c]}{dt} = -\int_\Omega \left[\f{\chi_1}{\gamma_1} u \Bigl(\nabla \f{\delta E}{\delta u}\Bigr)^2 + \f{\chi_2}{\gamma_2} v \Bigl(\nabla \f{\delta E}{\delta v} \Bigr)^2 + \Bigl(\f{\delta c}{\delta t}\Bigr)^2 \right] d\bm{x} \le 0,
	\end{aligned}
\end{equation}
where 
$\f{\delta E}{\delta u} = \f{\gamma_1 \kappa_1}{\chi_1} f'(u) - \gamma_1 c$ and $\f{\delta E}{\delta v} = \f{\gamma_2 \kappa_2}{\chi_2} f'(v) - \gamma_2 c$.

In recent years, considerable efforts have been devoted to the development of structure-preserving numerical methods for the Keller–Segel model that rigorously maintain the physical laws \eqref{model:KS:MC}, \eqref{model:KS:positivity}, and \eqref{model:KS:Energy}. For example, Chertock et al. \cite{AYHA'18} developed a fourth-order hybrid finite-volume–finite-difference scheme that can preserve positivity and mass conservation, while demonstrating high-order spatial accuracy and structure-preserving capabilities. Based on the Slotboom formulation \cite{Slotboom'73, JY'11,HZ'23}, Liu et al. \cite{LWZ'18} designed second-order central difference schemes that can also preserve both positivity and mass conservation. Although optimal-order error estimates were not established, the authors provided a stability analysis for the proposed scheme. Moreover, inspired by the KKT-based positivity-preserving limiter algorithm, Cheng and Shen \cite{CS'22} developed a positivity/bound-preserving and mass-conservative Lagrange multiplier method for nonlinear parabolic systems. This method bypasses the need for complex nonlinear constrained optimization solvers, thereby substantially improving computational efficiency. Based on the scalar auxiliary variable (SAV) approach and function transformation, Huang and Shen \cite{HS'21} also constructed a positivity/bound-preserving, mass-conservative, and unconditionally modified energy-dissipative high-order time discretization scheme. Recently, Tong and Cai \cite{TC'24} constructed a Crank-Nicolson (CN) type finite difference scheme with second-order accuracy for the Poisson-Nernst-Planck equation, where a novel projection approach is adopted to ensure both positivity and mass conservation. Most importantly, they proved the optimal-order error estimates for the proposed scheme. This approach has also inspired us to develop structure-preserving numerical schemes for the two-species Keller--Segel chemotaxis model.

The block-centered finite difference  (BCFD) method \cite{AWY'97}, also known as the cell-centered finite difference method, has been widely applied in recent years to the solution of various PDE models \cite{RP'12,RL'15,AWY'97,XXF'22,WXF'24,SXLF'21,LSR'19}. Notably, it can achieve second-order spatial accuracy on non-uniform spatial grids without sacrificing accuracy compared to standard finite difference schemes, and thus has great potential for simulating problems such as the Keller–Segel model with local blow-up solutions. As a positive first step, we recently developed a linearly implicit, fully decoupled CN-BCFD scheme on non-uniform spatial grids \cite{XF'25}. The proposed scheme not only guarantees mass conservation and second-order convergence, but also demonstrates a remarkable ability to capture blow-up phenomenon effectively and accurately. To further improve computational efficiency, it is better to adopt a time-staggered grid discretization that is able to facilitate linearization. In a recent paper \cite{ZWP'26}, Zhang et al. proposed a uniformly time-staggered numerical scheme for the Keller--Segel--Navier--Stokes model, which achieves decoupling and linearization of the cell density and chemoattractant concentration. 
However, rigorous error estimates for the fully discrete scheme are lacking. 
Motivated by these observations, we propose a non-uniform time-staggered BCFD scheme to better capture blow-up phenomenon effectively and provide a rigorous theoretical analysis.

In this paper, we primarily focus on preserving positivity \eqref{model:KS:positivity} and mass conservation \eqref{model:KS:MC} for the two-species Keller–Segel system \eqref{model:KS}--\eqref{model:KS:ic}. To this end, we employ the time-staggered BCFD discretization in which the cell densities are first computed via a prediction step, yielding solutions may not be positive; the predicted solutions are then projected onto a function space that enforces both positivity and mass conservation constraints via the standard $L^2$ projection; and finally, the chemoattractant concentration is solved using the up-to-date cell densities. This approach ensures the desired physical properties, while requiring only the solutions of two simple nonlinear single-variable algebraic equations; see Eq. \eqref{eq:Lagrange:nonlinear} of Remark \ref{Rem:KKT} for reference. Moreover, the proposed time-staggered BCFD method with $L^2$ projection is shown to be second-order accuracy in both time and space. In summary, this work presents a linearly implicit second-order time-staggered BCFD prediction-then-projection scheme with three main contributions:
\begin{itemize}
	\item The proposed scheme is almost linear (expect for the efficient $L^2$ projection step) and fully decoupled through a variable-step staggered-in-time discretization approach, which significantly enhances computational efficiency.
    
	\item The scheme unconditionally preserves positivity and mass conservation of the cell densities at the discrete level.  In addition, the non-negativity of the chemoattractant concentration is ensured under a sufficient time-step condition.
    
    \item Optimal-order error analysis is rigorously established on non-uniform temporal grids, enabling the use of an efficient adaptive time-stepping strategy to accurately capture the blow-up phenomenon.
\end{itemize}

The rest of the paper is organized as follows. In Section \ref{sec:mc-bcfd}, we propose a fully discrete time-staggered CN-BCFD scheme with an $L^2$ projection strategy, and prove the positivity-preserving and mass-conservation properties at the discrete level for the two-species Keller--Segel chemotaxis model. Optimal-order error estimates together with the unique solvability of solutions to the proposed scheme are presented in Section \ref{sec:erro-estimates}. In Section \ref{sec:num}, we present several numerical experiments to validate the accuracy, physical property-preserving properties, and reliability of the proposed scheme in simulating the blow-up phenomenon. Concluding remarks are given in Section \ref{sec:conclusion}.  Throughout this paper, we denote by $K$ with or without subscripts a generic positive constant that is independent of the grid parameters, but may have different values in different occurrences.

\section{A non-uniform time-staggered BCFD prediction-then-projection scheme}\label{sec:mc-bcfd}
This section is devoted to the construction of a time-staggered BCFD prediction-then-projection method for the two-species Keller–Segel chemotaxis system \eqref{model:KS}--\eqref{model:KS:ic}.
For simplicity, below we assume $\Omega :=(a^x, b^x)\times (a^y, b^y)$ and take the physical parameters $\gamma_i = \kappa_i = \chi_i = \alpha = \beta \equiv 1$. 

\subsection{Notations and preliminaries}
To fully decouple the concentration variable from the two-species density variables in \eqref{model:KS} while enabling linearization, we employ a staggered-in-time variable-step Crank–Nicolson time-marching algorithm coupled with a non-uniform grid BCFD spatial discretization. 

First, we introduce two distinct families of non-uniform partitions of $[0,T]$. The primal time levels are defined as $0 = t_0<t_1<\cdots<t_N = T$, with stepsizes $\tau_{n-1/2}=t_n-t_{n-1}$ for $1\leq n \leq N$, and the cell density unknowns are evaluated at these mesh nodes. The staggered time levels are chosen as $t_{n-1/2} = {(t_{n-1}+t_n)}/{2}$, the midpoint of each subinterval $[t_{n-1}, t_n]$  for $1\leq n \leq N$, and the the chemoattractant concentration variable is approximated at these staggered nodes. Additionally, we set $\tau_0 = t_{1/2}-t_0={\tau_{1/2}}/{2}$ and $\tau_{n} =t_{n+1/2}-t_{n-1/2}={(\tau_{n+1/2} + \tau_{n-1/2})}/{2}$ for $1\leq n \leq N-1$. Let the maximum stepsize be $\tau:=\max_{1\leq n\leq N}\tau_{n-1/2}$. Furthermore, we assume that 
there exist two positive constants $\sigma_* \le \sigma^*$ such that, for all $1\leq n \leq N-1$,
$$ 	
  \sigma_* \tau_{n-1/2} \le \tau_{n+1/2}  \le \sigma^* \tau_{n-1/2}.
$$ 
Given staggered-in-time grid functions $\{g^n\}_{n\ge 0}$ and $\{g^{n+1/2}\}_{n\ge 0}$,  we define
\begin{align*}
	& d_\tau g^{n+1/2}= \f{g^{n+1}-g^{n}}{\tau_{n+1/2}},   ~~
	  D_\tau g^{n}= \f{g^{n+1/2}-g^{n-1/2}}{\tau_{n}}.
\end{align*}

Next, let $N_x$ and $N_y$ be the numbers of spatial grids along the $x$- and $y$-coordinates, respectively. Similar to those used in \cite{WW'88,RP'13}, non-uniform staggered spatial grids are introduced. The primal grid points are denoted by
\begin{align*}
\Pi_x: ~ a^x=x_{1 / 2}<x_{3 / 2}< \ldots < x_{i-1 / 2}<x_{i+1 / 2} < \ldots <  x_{N_x+1 / 2}=b^x,\\
\Pi_y: ~ a^y=y_{1 / 2}<y_{3 / 2}< \ldots < y_{j-1 / 2}<y_{j+1 / 2} < \ldots <  y_{N_y+1 / 2}=b^y,
\end{align*}
with grid sizes $h_i=x_{i+1/2}-x_{i-1/2}$ for $i=1, \ldots, N_x$ and 
$k_j=y_{j+1/2}-y_{j-1/2}$ for $j=1, \ldots, N_y$. Let $h:=\max \{h_i,  k_j\}$.
The middle grid points are denoted by
\begin{align*}
	 \Pi_x^*: ~ x_i= (x_{i-1/2}+x_{i+1/2} ) / 2,  ~~ i=1, \ldots, N_x,~~
     \Pi_y^*: ~ y_j= (y_{j-1/2}+ y_{j+1/2} ) / 2,  ~~ j=1, \ldots, N_y,
\end{align*}
with grid sizes  
$h_{i+1/2}=x_{i+1}-x_i= (h_{i+1} +h_i)/2$ for $i=1,\ldots, N_x-1$ and 
$k_{j+1/2}=y_{j+1}-y_j= (k_{j+1} +k_j)/2$ for $j=1,\ldots, N_y-1$.
Given spatial grid functions $g=\{g_{i, j}\}$, $\hat{g}=\{g_{i+1/2, j}\}$ and $\check{g}=\{g_{i, j+1/2}\}$ defined on $\Pi_x^*\times \Pi_y^*$, $\Pi_x\times \Pi_y^*$ and $\Pi_x^*\times \Pi_y$, respectively,  we define
\begin{align*}
	& [d_x g]_{i+1/2, j}=d_x g_{i+1/2, j}=\f{g_{i+1, j}-g_{i, j}}{h_{i+1/2}}, ~~	
        && [d_y g]_{i, j+1/2}=d_y g_{i, j+1/2}=\f{g_{i, j+1}-g_{i, j}} {k_{j+1/2}}, \\
	& \bigl[D_x \hat{g}]_{i, j}=D_x \hat{g}_{i, j}=\f{\hat{g}_{i+1/2, j}-\hat{g}_{i-1/2, j}}{h_i}, ~~
	&&  \bigl[D_y \check{g}]_{i, j}=D_y \check{g}_{i, j}=\f{\check{g}_{i, j+1/2}-\check{g}_{i, j-1/2}}{k_j}.
\end{align*}
Besides, we introduce the discrete inner products and norms on $\Pi_x^*\times \Pi_y^*$, $\Pi_x\times \Pi_y^*$ and $\Pi_x^*\times \Pi_y$, respectively, as follows:
\begin{align*}
	& (f, g)_{\rm M}=\sum_{i=1}^{N_x} \sum_{j=1}^{N_y} h_i k_j f_{i, j} g_{i, j},  &&\|f\|_{\rm M}^2 = (f, f)_{\rm M}, \\
	& (f, g)_x=\sum_{i=1}^{N_x-1} \sum_{j=1}^{N_y} h_{i+1/2} k_j f_{i+1/2, j} g_{i+1/2, j},  &&\|f\|_x^2 = (f, f)_x, \\
	& (f, g)_y=\sum_{i=1}^{N_x} \sum_{j=1}^{N_y-1} h_i k_{j+1/2} f_{i, j+1/2} g_{i, j+1/2},  &&\|f\|_y^2 = (f, f)_y,\\
	& (\bm {f}, \bm {g})_{\rm TM}  = (f^x, g^x )_x+ (f^y, g^y)_y,    &&\|\bm {f}\|_{\rm TM}^2 = (\bm{f}, \bm{f})_{\rm TM},\\
    & \|\bm{d}g\|_\infty 
    = \max _{i, j} |[d_x g]_{i+1/2, j}|+ \max _{i, j} |[d_y g]_{i, j+1/2}|. &&
\end{align*}

The following two lemmas shall be used in the subsequent analysis.
\begin{lemma}[\cite{WW'88}] \label{lemma:Dd} 
 Let $\{q_{i, j}\}, \{v_{i+1/2, j}\}$ and $\{w_{i, j+1 / 2}\}$ be any grid functions defined on $\Pi_x^*\times \Pi_y^*$, $\Pi_x\times \Pi_y^*$ and $\Pi_x^*\times \Pi_y$, such that $v_{1/2, j}=$ $v_{N_x+1/2, j}=w_{i, 1/2}=w_{i, N_y+1 / 2}=0$. Then there holds
$$
\begin{aligned}
 (q, D_x v)_{\rm M}=-(d_x q, v)_x, ~~ (q, D_y w)_{\rm M}=-(d_y q, w)_y .
\end{aligned}
$$
\end{lemma}
\begin{lemma}[\cite{BGM'07}] \label{trun:Dp} Let $p\in H^3(\Omega)$. Then there holds
\begin{equation*}
 \begin{aligned} 
     \p_x p (x_{i}, y_j)= \bigl[D_x p ]_{i, j}      + \vartheta^x_{i,j}(p),\quad
     \p_y p (x_i, y_j) = \bigl[D_y p ]_{i, j}   +\vartheta^y_{i,j}(p),
\end{aligned}
\end{equation*}
such that
\begin{equation*}
\|\vartheta^x(p)\|_{\rm M}  \leq K  \|p\|_{H^{3}(\Omega)}h^2,
\quad \|\vartheta^y(p)\|_{\rm M}  \leq K \|p\|_{H^{3}(\Omega)} h^2.
\end{equation*}
\end{lemma}

Finally, given values $\{p_{i,j} = p(x_i,y_j)\}$, for any points $(x, y) \in [x_i, x_{i+1}]\times [y_j, y_{j+1}]$, $i=1,\ldots, N_x-1$, $j=1,\ldots, N_y-1$, we introduce the piecewise bilinear interpolation function $\ell_h p(x, y)$ by
 \begin{equation}\label{Operator:bi}
\begin{aligned}
	\ell_h p(x, y)& =    \f{(x_{i+1}-x)(y_{j+1}-y)}{(x_{i+1}-x_i)(y_{j+1}-y_j)}p_{i, j} 
           + \f{(x-x_i)(y_{j+1}-y)}{(x_{i+1}-x_i)(y_{j+1}-y_j)} p_{i+1, j} \\
	& \quad + \f{(x_{i+1}-x)(y-y_j)}{(x_{i+1}-x_i)(y_{j+1}-y_j)}  p_{i, j+1} 
         + \f{(x-x_i)(y-y_j)}{(x_{i+1}-x_i)(y_{j+1}-y_j)} p_{i+1, j+1}.
\end{aligned}
\end{equation}

\begin{lemma}[\cite{DWW'98}]\label{lem:interp} 
Assume that $p \in W^{2, \infty}(\Omega)$, then we have 
\begin{align*}
\left\|\ell_h p-p\right\|_{L^\infty} \leq K  \|p\|_{W^{2, \infty}(\Omega)} h^2.
\end{align*}
\end{lemma}

\subsection{The PP-MC-PBCFD scheme}\label{sec:pp-mc-I}
At each time step, the spatial discretization of the two-species Keller–Segel chemotaxis system \eqref{model:KS}–\eqref{model:KS:ic} is carried out using the non-uniform grid BCFD method. The primal scalar variables—namely, the cell densities and the chemoattractant concentration—are discretized on the staggered grid $\Pi_x^*\times \Pi_y^*$, whereas the flux/gradient variables are approximated on $\Pi_x\times \Pi_y^*$ (for $x$-direction) and $\Pi_x^*\times \Pi_y$ (for $y$-direction). For temporal discretization, the cell densities are computed at the staggered time points $t_{n+1/2}$ using a prediction-then-projection approach, and the chemoattractant concentration is subsequently solved at the primal time levels $t_{n+1}$. 

To be specific, let $ \mbX_h $ be the set of real-valued grid functions defined on $\Pi_x^*\times \Pi_y^*$, i.e.,
$$
\mbX_h:=\left\{v_h \mid v_h= \{v_{i,j}\}, ~~ i=1, \ldots, N_x, ~ j=1, \ldots, N_y\right\}.
$$ 
We denote the numerical solutions of the chemotaxis system \eqref{model:KS}–\eqref{model:KS:ic} by $\{u_h^{n+1/2},v_h^{n+1/2},c_h^{n+1}\}\in \mbX_h \times \mbX_h \times \mbX_h$. Moreover, define a subspace of $\mbX_h$ as 
$$
\mbX_h[q]:=\left\{q_h \in \mbX_h \mid q_h \ge 0, ~ M_h[q_h^{n+1/2}] = M_h[q_h^0]\right\},
$$
where $M_h[q]:=  (q, 1 )_{\mrM}$ represents the discrete version of mass. In addition, we denote the average density and concentration solutions at the temporal midpoints $(t_{n+1/2}+t_{n-1/2})/{2} 
$ and $t_{n+1/2}$, respectively, as 
\[
\overline{u}_h^{n} := \f{1}{2}({u}_h^{n+1/2} + u_h^{n-1/2}), \quad \overline{v}_h^{n} := \f{1}{2}({v}_h^{n+1/2} + v_h^{n-1/2}), \quad \overline{c}_h^{n+1/2} := \f{1}{2}(c_h^{n+1}+c_h^n).
\]   

Now, let the initial approximations be given by $\{u_h^{0},v_h^{0},c_h^{0}\}:=\{\ell_h u^{0}, \ell_h v^{0},\ell_h c^{0}\}
\in \mbX_h[u] \times \mbX_h[v] \times \mbX_h$. For $n \ge 0$, we propose the positivity-preserving, mass-conservative, projection-based BCFD scheme, termed the PP-MC-PBCFD scheme, as follows: 

{\bf Step 1}: For $n=0$, solve the predicted density solutions $\{\widetilde{u}_h^{1/2},\widetilde{v}_h^{1/2}\}\in \mbX_h \times \mbX_h$ at $t=t_{1/2}$ via the semi-implicit Euler-BCFD method:
	\begin{subequations}\label{model:KS:schemeI0}
		\begin{align}
			\f{\widetilde{u}_h^{1/2}-u_h^0}{\tau_0} &= D_x (d_x \widetilde{u}_h^{1/2}) 
              + D_y(d_y \widetilde{u}_h^{1/2} ) 
            - D_x([\ell_h \widetilde{u}_h^{1/2}][d_x c_h^0]) 
             - D_y([\ell_h \widetilde{u}_h^{1/2}][d_y c_h^0]),\label{model:KS:sheme11}\\
			\f{\widetilde{v}_h^{1/2}-v_h^0}{\tau_0} &=  D_x(d_x \widetilde{v}_h^{1/2})
               +  D_y(d_y \widetilde{v}_h^{1/2}) 
            - D_x([\ell_h \widetilde{v}_h^{1/2}][d_x c_h^0]) - D_y([\ell_h \widetilde{v}_h^{1/2}][d_y c_h^0]).\label{model:KS:sheme22}
		\end{align}
	\end{subequations}
For $n\ge 1$, given $\{u_h^{n-1/2},v_h^{n-1/2}, c_h^{n-1}, c_h^{n}\}\in \mbX_h[u] \times \mbX_h[v] \times \mbX_h \times \mbX_h$, solve the predicted density solutions $\{\widetilde{u}_h^{n+1/2},\widetilde{v}_h^{n+1/2}\}\in \mbX_h \times \mbX_h$ at the staggered time point $t=t_{n+1/2}$ via the semi-implicit CN-BCFD method: 
\begin{subequations}\label{model:KS:shemeI}
	\begin{align}
		\f{\widetilde{u}_h^{n+1/2}-u_h^{n-1/2}}{\tau_{n}} & =  D_x(d_x \overline{\widetilde{u}}_h^{n}) 
           +  D_y(d_y \overline{\widetilde{u}}_h^{n})   
           - D_x \bigl([\ell_h \overline{\widetilde{u}}_h^{n}][d_x c_h^{*,n}]\bigr) 
           - D_y \bigl([\ell_h \overline{\widetilde{u}}_h^{n}][d_y c_h^{*,n}]\bigr),\label{model:KS:shemeIa}\\
		\f{\widetilde{v}_h^{n+1/2}-v_h^{n-1/2}}{\tau_{n}} &= D_x(d_x \overline{\widetilde{v}}_h^{n})
           + D_y(d_y \overline{\widetilde{v}}_h^{n})
           -  D_x \bigl([\ell_h \overline{\widetilde{v}}_h^{n}][d_x c_h^{*,n}]\bigr)  
           - D_y \bigl([\ell_h \overline{\widetilde{v}}_h^{n}][d_y c_h^{*,n}]\bigr),\label{model:KS:shemeIb}
	\end{align}
\end{subequations}
where $\overline{\widetilde{u}}_h^{n} := \f{1}{2}(\widetilde{u}_h^{n+1/2} + u_h^{n-1/2})$ and $\overline{\widetilde{v}}_h^{n} := \f{1}{2}(\widetilde{v}_h^{n+1/2} + v_h^{n-1/2})$ represent the average values at the midpoint $t_*:=\frac{t_{n+1/2}+t_{n-1/2}}{2}
$, and the concentration variable in the nonlinear part is approximated by the linear extrapolation formula, i.e., $c_h^{*,n} := \frac{t_n-t_*}{\tau_{n-1/2}}c_h^{n-1}+\frac{t_*-t_{n-1}}{\tau_{n-1/2}}c_h^n$ for $n \ge 1$.

{\bf Step 2}: Project the intermediate density solutions $\{\widetilde{u}_h^{n+1/2},\widetilde{v}_h^{n+1/2}\}$ from $\mbX_h \times \mbX_h$ to $\mbX_h[u] \times \mbX_h[v]$ via the standard discrete $L^2$ projection, and obtain the corrected density solutions $\{u_h^{n+1/2},v_h^{n+1/2}\}\in \mbX_h[u] \times \mbX_h[v]$ such that
\begin{equation}\label{model:KS:shemeI2}
	\begin{aligned}
		& \min_{\{u_h^{n+1/2}, v_h^{n+1/2 }\}\in \mbX_h[u] \times \mbX_h[v]}  \f{1}{2} \left( \| u_h^{n+1/2} - \widetilde{u}_h^{n+1/2} \|^2 + \| v_h^{n+1/2} - \widetilde{v}_h^{n+1/2} \|^2 \right).
	\end{aligned}
\end{equation}

{\bf Step 3}: Solve the concentration solution $c_h^{n+1}\in \mbX_h$ at the primal time level $t=t_{n+1}$ via the CN-BCFD scheme:
\begin{equation}\label{model:KS:shemeI3}
	\begin{aligned}
        d_\tau c_h^{n+1/2}= D_x(d_x \overline{c}_h^{n+1/2}) + D_y(d_y \overline{c}_h^{n+1/2})
        - \overline{c}_h^{n+1/2} +	u_h^{n+1/2} + v_h^{n+1/2}.
	\end{aligned}
\end{equation}

\begin{remark} 
The staggered-in-time discretization approach fully decouples the computations of the multi-species Keller–Segel chemotaxis system while simultaneously facilitating linearization, consequently yielding a substantial improvement in computational efficiency. Moreover, the variable-step time-marching algorithm enables the development of adaptive time-stepping strategies for long-term energy dissipative simulations with preservation of positivity, and the non-uniform grid BCFD method provides more accurate and efficient simulations of chemotactic dynamics, particularly in the presence of rapid blow-up phenomenon. Numerical experiments presented in Section \ref{sec:num} support the efficiency and accuracy of the proposed scheme.
\end{remark}

\begin{remark}\label{Rem:KKT} Note that {\bf Step 1} cannot preserve positivity of the cell density solutions. Therefore, in {\bf Step 2} we adopt an $L^2$ projection to enforce both positivity and mass conservation. In fact, \eqref{model:KS:shemeI2} is a convex minimization problem and can be represented by the following Karush--Kuhn-Tucker (KKT) conditions:
\begin{subequations} 
	\begin{align}
	u_h^{n+1/2} = \widetilde{u}_h^{n+1/2} + \lambda_h^{n+1/2} - \xi^{n+1/2}, \quad v_h^{n+1/2} = \widetilde{v}_h^{n+1/2} + \eta_h^{n+1/2} - \theta^{n+1/2}, \label{model:KS:shemeI2a}
\\
	\lambda_h^{n+1/2} u_h^{n+1/2} = 0, \quad \eta_h^{n+1/2} v_h^{n+1/2} = 0, \quad \lambda_h^{n+1/2} \geq 0, \quad \eta_h^{n+1/2} \geq 0, \label{model:KS:shemeI2b}
\\
    M_h[u_h^{n+1/2}]= M_h[u_h^{0}], \quad M_h[v_h^{n+1/2}]= M_h[v_h^{0}], \label{model:KS:shemeI2c}
	\end{align}
\end{subequations}
where the time-dependent Lagrange multipliers $\{\xi^{n+1/2},\theta^{n+1/2}\}$ are introduced for the mass conservation constraint \eqref{model:KS:shemeI2c}, and the space-time-dependent Lagrange multipliers $\{\lambda_h^{n+1/2},\eta_h^{n+1/2}\}\in \mbX_h \times \mbX_h $ are introduced to enforce positivity of the numerical density solutions $\{u_h^{n+1/2},v_h^{n+1/2}\}$. 
By the complementary condition in \eqref{model:KS:shemeI2b}, the computations of $\{u_h^{n+1/2}, v_h^{n+1/2}\}$ can be expressed as
\begin{equation*} 
   \begin{aligned}
	(u_h^{n+1/2},\lambda_h^{n+1/2}) = 
	\begin{cases}
		\bigl(\widetilde{u}_h^{n+1/2} - \xi^{n+1/2}, 0 \bigr),    &\quad~~ \text{if}~ \widetilde{u}_h^{n+1/2} - \xi^{n+1/2} \geq 0,\\
		\bigl(0,-(\widetilde{u}_h^{n+1/2} - \xi^{n+1/2}) \bigr),  &\quad~~ \text{otherwise},
	\end{cases} \\
	(v_h^{n+1/2},\eta_h^{n+1/2}) = 
	\begin{cases}
		\bigl(\widetilde{v}_h^{n+1/2} - \theta^{n+1/2}, 0\bigr),    &\quad~~ \text{if}~  \widetilde{v}_h^{n+1/2} - \theta^{n+1/2} \geq 0,\\
		\bigl(0,-(\widetilde{v}_h^{n+1/2} - \theta^{n+1/2}) \bigr).  &\quad~~ \text{otherwise},
	\end{cases}
    \end{aligned}
\end{equation*}
Meanwhile, by \eqref{model:KS:shemeI2c}, the Lagrange multipliers $\xi^{n+1/2}, \theta^{n+1/2} \in \mathbb{R}$ are determined by the mass conservation constraint, which reduce to solving the following nonlinear single-variable algebraic equations:
\begin{equation}\label{eq:Lagrange:nonlinear}
    \begin{aligned}
	F(\xi^{n+1/2}) :=  \sum_{i=1}^{N_x}\sum_{j = 1}^{N_y}h_ik_j (\widetilde{u}_{i,j}^{n+1/2} - \xi^{n+1/2})^{+} - (u_h^0, 1)_{\mrM}=0, \\
    G(\theta^{n+1/2}):=\sum_{i=1}^{N_x}\sum_{j = 1}^{N_y} h_ik_j(\widetilde{v}_{i,j}^{n+1/2} - \theta^{n+1/2})^{+} - (v_h^0, 1)_{\mrM}=0, 
    \end{aligned}
\end{equation}
where $^+$ means the positive part. As suggested in Refs. \cite{TC'24,CS'22}, the above nonlinear algebraic equations can be solved efficiently by the semismooth Newton method or the secant method in only a few iterations.
\end{remark}

\begin{lemma}\label{lem:positivity:LagM} For the PP-MC-PBCFD scheme, there holds
	$$
    \xi^{n+1/2}\geq 0,~~\theta^{n+1/2}\geq 0,  \quad n \ge 0.
    $$
\end{lemma}
\begin{proof} 
We only prove the first conclusion, as the second one can be derived in a very similar way. By adding \eqref{model:KS:shemeIa} and the first equation of \eqref{model:KS:shemeI2a} together, we have
\begin{equation}\label{eq:pp-xi-lambda}
	\begin{aligned}
		\f{{u}_h^{n+1/2}-u_h^{n-1/2}}{\tau_{n}} - \f{\lambda_h^{n+1/2} - \xi^{n+1/2}}{\tau_{n}} &= D_x(d_x \overline{\widetilde{u}}_h^{n}) + D_y(d_y \overline{\widetilde{u}}_h^{n}) \\
        &\quad- D_x \bigl([\ell_h \overline{\widetilde{u}}_h^{n}][d_x c_h^{*,n}]\bigr) 
        - D_y \bigl([\ell_h \overline{\widetilde{u}}_h^{n}][d_y c_h^{*,n}]\bigr).
	\end{aligned}
\end{equation}
Then, according to Lemma \ref{lemma:Dd} and noting the fact that $u_h^{n+1/2}\in \mbX_h[u]$, we take the inner product of \eqref{eq:pp-xi-lambda} with 1 on both sides to derive
	$$
    (\lambda_h^{n+1/2} - \xi^{n+1/2},1)_{\mrM}=0 \Longrightarrow \xi^{n+1/2} = \f{(\lambda_h^{n+1/2},1)_{\mrM}}{|\Omega|} \geq 0,
    $$
 where \eqref{model:KS:shemeI2b} is applied in the last step. Thus, the conclusion is proved.
\end{proof}

Below, we present a sufficient condition for the characterization of a non-singular M-matrix, which is the main tool for establishing positivity of the chemoattractant concentration variable. 
\begin{lemma}[\cite{HZ'23}]\label{lem:M-matrix}
For a real square matrix $\mbb A$ with positive diagonal entries and non-positive off-diagonal entries,  it is a non-singular M-matrix if all the row sums of  $\mbb A$ are non-negative and at least one row sum is positive.
\end{lemma}
\begin{lemma}[\cite{Ple'77}]\label{lem:M-matrix:inverse}
Let $\mbb A\in \mathbb{R}^{n\times n}$ be a non-singular M-matrix, then $\mbb A$ is inverse-positive. That is, $\mbb A^{-1}$ exists and $\mbb A^{-1} \geq 0.$
\end{lemma}

We present the main structure-preserving conclusion for the PP-MC-PBCFD scheme as follows.
\begin{theorem}\label{thm:MassConserve} Let the initial values $\{u_h^{0},v_h^{0},c_h^{0}\}
\in \mbX_h[u] \times \mbX_h[v] \times \mbX_h$. Then, the solutions to the PP-MC-PBCFD scheme \eqref{model:KS:schemeI0}--\eqref{model:KS:shemeI3} unconditionally satisfy $\{u_h^{n+1/2},v_h^{n+1/2},c_h^{n+1}\}\in \mbX_h[u] \times \mbX_h[v] \times \mbX_h$. Furthermore, if $c_h^{0} \geq 0$ and $\tau \le \min \bigl\{1, \f{2 h_{min}^2 k_{min}^2}{h_{min}^2+ k_{min}^2}\bigr\}$, where $h_{min}:=\min_{i}h_i$ and $ k_{min}:=\min_{j}k_j$, then the chemoattractant concentration solution is also non-negative, i.e., $c_h^{n+1} \ge 0$.
\end{theorem}
\begin{proof} First, the unconditional positivity-preserving and mass conservation properties of the density solutions $\{u_h^{n+1/2}, v_h^{n+1/2}\}$ are directly implied from {\bf Step 2}. Next, we demonstrate the non-negativity of the concentration $c_h^{n+1}$. 

Let $\mbb{C}^n$, $\mbb{U}^{n+1/2}$ and $\mbb{V}^{n+1/2}$ denote the column vectors corresponding to the solutions $c_h^{n+1}$, $u_h^{n+1/2}$ and $v_h^{n+1/2}$, respectively. Moreover, let $\mbb{I}_x$ and $\mbb{I}_y$ be the identity matrices of orders $N_x$ and $N_y$, respectively. Let $\mbb{B}$ be the discrete Laplace operator matrix of order $N_s:=N_xN_y$ such that 
\begin{equation*}
    \mbb{B} = \mbb{B}_y \otimes \mbb{I}_x + \mbb{I}_y  \otimes \mbb{B}_x,
\end{equation*}
where the matrix $\mbb{B}_x$ of order $N_x$ has nonzero entries defined by
\[
\mbb{B}_x(i,j) = 
\begin{cases}
-\f{1}{h_1 h_{3/2}}, & j=i= 1, \\
~~\f{1}{h_i h_{i-1/2}}, &j=i-1,~2 \le i \le N_x,\\
-\f{1}{h_i h_{i-1/2}} - \f{1}{h_i h_{i+1/2}} , & j=i,~2 \le i \le N_x-1, \\
~~\f{1}{h_i h_{i+1/2}}, &j=i+1,~1 \le i \le N_x-1,\\
-\f{1}{h_{N_x} h_{N_x-1/2}}, & j=i= N_x,
\end{cases}
\]
and the matrix $\mbb{B}_y$ is defined analogously, with $h$ and $N_x$ replaced by $k$ and $N_y$. 
Then, we can rewrite \eqref{model:KS:shemeI3} in a compact matrix form:
\begin{equation*}
    \mbb{P} \mbb{C}^{n+1} = \mbb{Q}\mbb{C}^{n} + \mbb{U}^{n+1/2} + \mbb{V}^{n+1/2},
\end{equation*}
where $\mbb{P}:=(\f{1}{\tau_{n+1/2}} + \f{1}{2}) \mbb{I}_y \otimes \mbb{I}_x - \f{\mbb{B}}{2}$ and $\mbb{Q}= (\f{1}{\tau_{n+1/2}} - \f{1}{2}) \mbb{I}_y \otimes \mbb{I}_x + \f{\mbb{B}}{2}$.  
From Lemma \ref{lem:M-matrix}, we can readily verify that $\mbb{P}$ is a non-singular M-matrix; consequently, Lemma \ref{lem:M-matrix:inverse} implies $\mbb{P}^{-1} \geq 0$. Note that $\mbb{C}^n,\mbb{U}^{n+1/2},\mbb{V}^{n+1/2} \ge 0 $. Therefore, to preserve positivity for $\mbb{C}^{n+1}$, a sufficient condition is that all entries of $\mbb{Q}$ be non-negative. This leads to the following time-step condition:
\[
\f{1}{\tau_{n+1/2}}-\f{1}{2}-\f{1}{2}\left(\f{1}{h_{i}h_{i-1/2}} + \f{1}{h_{i}h_{i+1/2}} + \f{1}{k_{j}k_{j-1/2}} + \f{1}{k_{j}k_{j+1/2}}\right)\geq 0, 
\]
which can be satisfied by taking  $\tau$ sufficiently small such that
 \[ 
   \tau \le \min \bigl\{1, \f{2 h_{min}^2 k_{min}^2}{h_{min}^2+ k_{min}^2}\bigr\}.
 \] 
 Thus, the desired result is proved.
\end{proof}

\section{Unique solvability and error estimates}\label{sec:erro-estimates}
In this section, we show that the PP-MC-PBCFD scheme \eqref{model:KS:schemeI0}--\eqref{model:KS:shemeI3} is second-order accurate in both time and space under uniform spatial grids and the following regularity assumptions:
\begin{equation}\label{eq:regular}
	c(\bm x,t),u(\bm x,t),v(\bm x,t) \in W^{3,\infty}(0,T;W^{2,\infty}) \cap L^{ \infty}(0,T;W^{4,\infty}) \cap W^{2,\infty}(0,T;W^{3,\infty}).
\end{equation}
In particular, we assume that for some positive constant $K_*$, there holds
\begin{equation}\label{eq:regular:1}
\|\nabla c\|_{L^{\infty}(0,T;L^{\infty})} \leq K_*.
\end{equation}

Let $c^{n+1} = c(\cdot, t_{n+1})$, and $u^{n+1/2} = u(\cdot, t_{n+1/2})$ and $v^{n+1/2} = v(\cdot, t_{n+1/2})$ be the exact solutions of system \eqref{model:KS}--\eqref{model:KS:ic} at time $t=t_{n+1}$ and $t=t_{n+1/2}$, respectively. For the purpose of numerical analysis, we introduce two "biased" auxiliary solutions 
\begin{equation*}
{\theta}_{\epsilon}^{n+1/2} := \bigl(1 + \epsilon_{\theta}^{n+1/2}\bigr) \ell_h\theta^{n+1/2},
~~
	\epsilon_{\theta}^{n+1/2} := \f{M_h[\theta^0] - M_h[\theta^{n+1/2}] }{M_h[ \theta^{n+1/2}]}, \quad \text{for}~~ \theta=u, v.
\end{equation*}
Note that under assumption \eqref{eq:regular}, for sufficiently small $h$, we have ${u}_{\epsilon}^{n+1/2}, ~{v}_{\epsilon}^{n+1/2} \geq 0$ and 
\begin{equation}\label{eq:epsilon}
|\epsilon_{u}^{n+1/2}| + |\epsilon_{v}^{n+1/2}| \leq Kh^2,~~ |\epsilon_{u}^{n+1/2}-\epsilon_{u}^{n-1/2}|+|\epsilon_{v}^{n+1/2}-\epsilon_{v}^{n-1/2}| \leq K\tau_n h^2,
\end{equation}
for $1\leq n \leq N-1$.
In fact, the first part in \eqref{eq:epsilon} is a direct consequence of the second-order midpoint rule, triangle inequality and the mass conservation law \eqref{model:KS:MC}, i.e.,
\begin{align*}
    |\epsilon_{u}^{n+1/2}| = \Big|\f{M_h[u^0] - M_h[u^{n+1/2}] }{M_h[u^{n+1/2}]}\Big| 
    &\leq K \big( | M_h[u^0]-M[u^0]| +  | M[u^{n+1/2}]  - M_h[u^{n+1/2}]| \big)\\
    &\leq K \|u\|_{L^{\infty}(0,T;W^{2,1})} h^2.
\end{align*}
For the second part in \eqref{eq:epsilon}, again by \eqref{model:KS:MC} and the standard midpoint rule error estimate, we get
\begin{equation*}
	\begin{aligned}
		|\epsilon_{u}^{n+1/2}-\epsilon_{u}^{n-1/2}| &= \Big|M_h[u^0] \f{M_h[u^{n-1/2}]-M_h[ u^{n+1/2}]}{M_h[u^{n-1/2}] M_h[u^{n+1/2}]}\Big|   \leq K\big|M_h[u^{n-1/2}]-M_h[ u^{n+1/2}]\big|\\
        &\leq K\big| \big(M[u^{n+1/2}] - M_h[u^{n+1/2}]\big)-\big(M[u^{n-1/2}] - M_h[u^{n-1/2}]\big)\big|\\
        &\leq K \|u\|_{W^{1,\infty}(0,T;W^{2,1})} \tau_n h^2. 
	\end{aligned}
\end{equation*}

Combining Taylor expansion, the Cauchy--Schwarz inequality, and Lemmas \ref{trun:Dp}--\ref{lem:interp}, 
we are led to estimates for the following local truncation errors:  
	\begin{align}
		R_{u}^{1/2} &:= \frac{{u}_{\epsilon}^{1/2}-u^0}{\tau_0} - D_xd_x {u}_{\epsilon}^{1/2} - D_yd_y {u}_{\epsilon}^{1/2} 
         - D_x \bigl([\ell_h {u}_{\epsilon}^{1/2}][d_xc^0]\bigr) - D_y \bigl([\ell_h {u}_{\epsilon}^{1/2}][d_y c^0]\bigr), \label{eq:Ru_0} \\
		R_{v}^{1/2} &:= \frac{{v}_{\epsilon}^{1/2}-v^0}{\tau_0} - D_xd_x {v}_{\epsilon}^{1/2} - D_yd_y {v}_{\epsilon}^{1/2}
         - D_x \bigl([\ell_h {v}_{\epsilon}^{1/2}][d_x c^0]\bigr) - D_y \bigl([\ell_h {v}_{\epsilon}^{1/2}][d_y c^0]\bigr), \label{eq:Rv_0} 
    \end{align}
    and for $n \geq 1$,
\begin{align}         
		R_{u}^{n} &:= \frac{{u}_{\epsilon}^{n+1/2}-{u}_{\epsilon}^{n-1/2}}{\tau_{n}} - D_x d_x \overline{u}_{\epsilon}^{n} - D_y d_y \overline{u}_{\epsilon}^{n} 
        - D_x \bigl([\ell_h \overline{u}_{\epsilon}^{n}][d_x c^{*,n}]\bigr) - D_y \bigl([\ell_h \overline{u}_{\epsilon}^{n}][d_y c^{*,n}]\bigr),  \label{eq:Ru_n} \\
		R_{v}^{n} &:= \frac{{v}_{\epsilon}^{n+1/2}-{v}_{\epsilon}^{n-1/2}}{\tau_{n}} - D_x d_x \overline{v}_{\epsilon}^{n} - D_y d_y \overline{v}_{\epsilon}^{n} 
         - D_x \bigl([\ell_h \overline{v}_{\epsilon}^{n}][d_x c^{*,n}]\bigr) - D_y \bigl([\ell_h \overline{v}_{\epsilon}^{n}][d_y c^{*,n}]\bigr),  \label{eq:Rv_n}
    \end{align}
        and for $n \geq 0$,
\begin{align}              
     R_{c}^{n+1/2} := d_\tau c^{n+1/2}- D_x d_x \overline{c}^{n+1/2} - D_y d_y \overline{c}^{n+1/2} + \overline{c}^{n+1/2} 
        - {u}_{\epsilon}^{n+1/2} - {v}_{\epsilon}^{n+1/2}. \label{eq:Rc_n} 
\end{align}
    
\begin{lemma}\label{lem:truncErr} For the local truncation errors $\{R_{u}^{n+1/2},R_{v}^{n+1/2}, R_c^{n+1}\}$ defined in \eqref{eq:Ru_0}--\eqref{eq:Rv_n}, under the regularity assumption \eqref{eq:regular}, we have 
\begin{equation*}\label{eq:tuncE}
	\begin{aligned}
		&\|R_{u}^{1/2}\|_{\rm{M}} + \|R_{v}^{1/2}\|_{\rm{M}} \leq K\big(\tau_0 + h^2\big),\\
        &\|R_{u}^{n}\|_{\rm{M}} + \|R_{v}^{n}\|_{\rm{M}} \leq K\big(\tau_n^2 + h^2\big), \quad n \geq 1,\\
		&\|R_c^{n+1/2}\|_{\rm{M}}  \leq K\big(\tau_n^2 + h^2\big),	 \quad n \geq 0,	
	\end{aligned}
\end{equation*}
where the positive constant $K$ is independent of mesh sizes $\tau$ and $h$. 
\end{lemma}
Next, set ${e}_{w}^{n+1/2} := {w}_{\epsilon}^{n+1/2} - w_h^{n+1/2}$, $\widetilde{e}_{w}^{n+1/2}:={w}_{\epsilon}^{n+1/2} - \widetilde{w}_h^{n+1/2}$ for $w=u, v$, and ${e}_{c}^{n+1} := \ell_hc^{n+1} - c_h^{n+1}$. It is straightforward to verify that \begin{equation}\label{err:trunc:mass}
  (e_{w}^{n+1/2},1)_{\mrM} =(\widetilde{e}_{w}^{n+1/2},1)_{\mrM} = 0~~ \text{for} ~w =u,v.
 \end{equation}
In addition, we define ${e}_{w}^{0} := {w}^{0} - w_h^{0}$ for $w=u, v$. Note that $w^0$ and $w_h^0$ are equal at the grid points $\Pi_x^*\times \Pi_y^*$, i.e., $e_{w,i,j}^0 \equiv 0$ for $i=1, \ldots, N_x, ~ j=1, \ldots, N_y$.


\begin{lemma} \label{lem:proj:e} 
	For the errors $\{\widetilde{e}_{u}^{n+1/2},\widetilde{e}_{v}^{n+1/2}\}$ and $\{{e}_{u}^{n+1/2}, {e}_{v}^{n+1/2}\}$, it holds that
	\begin{equation*}
		\|e_{u}^{n+1/2}\|_{\rm{M}}^2 + \|e_{u}^{n+1/2} - \widetilde{e}_{u}^{n+1/2}\|_{\rm{M}}^2 \leq \|\widetilde{e}_{u}^{n+1/2}\|_{\rm{M}}^2,
		~\|e_{v}^{n+1/2}\|_{\rm{M}}^2 + \|e_{v}^{n+1/2} - \widetilde{e}_{v}^{n+1/2}\|_{\rm{M}}^2 \leq \|\widetilde{e}_{v}^{n+1/2}\|_{\rm{M}}^2,
	\end{equation*}
	for $0 \leq n \leq N-1$.
\end{lemma}
\begin{proof} The proof for $e_{v}^{n+1/2}$ is the same as that of $e_{u}^{n+1/2}$ and is omitted here for brevity. We obtain from the first equation of \eqref{model:KS:shemeI2a} that
	\begin{equation*}
		e_{u}^{n+1/2} - \widetilde{e}_{u}^{n+1/2} = \lambda_h^{n+1/2} - \xi^{n+1/2}.
	\end{equation*}
	Then, taking the discrete inner product with $e_{u}^{n+1/2}$, we have
	\begin{equation*}
			\f{1}{2}(\|e_{u}^{n+1/2}\|_{\rm{M}}^2 + \|e_{u}^{n+1/2} - \widetilde{e}_{u}^{n+1/2}\|_{\rm{M}}^2 - \|\widetilde{e}_{u}^{n+1/2}\|_{\rm{M}}^2) = -(\lambda_h^{n+1/2},e_{u}^{n+1/2})_{\mrM},
	\end{equation*}
	where we have used the fact that $(e_{u}^{n+1/2},\xi^{n+1/2})_{\mrM}= \xi^{n+1/2} (e_{u}^{n+1/2},1)_{\mrM}=0$ due to the mass conservation \eqref{err:trunc:mass}. 
	Moreover, it follows from the KKT condition \eqref{model:KS:shemeI2b} and the fact ${u}_{\epsilon}^{n+1/2} \geq 0$ that
	$$-(\lambda_h^{n+1/2},e_{u}^{n+1/2})_{\mrM} = -(\lambda_h^{n+1/2},{u}_{\epsilon}^{n+1/2})_{\mrM} \leq 0.$$
	Thus, the first conclusion is proved.
\end{proof}

\begin{lemma}\label{lem:erho} Assume that the exact solutions satisfy the regularity condition \eqref{eq:regular}. Then, there exist two positive constants $K_u$ and $K_v$, independent of $\tau$ and $h$, such that
	\begin{equation}\label{err:uv:0}
		\| \widetilde{e}_{u}^{1/2}\|_{\rm M}^2 \leq  K_{u} \big( \tau_{0}^4 + h^4\big),\quad
		\| \widetilde{e}_{v}^{1/2}\|_{\rm M}^2 \leq  K_{v} \big( \tau_{0}^4 + h^4\big),\quad \text{for}~ \tau_0 \leq \tau_{*},
	\end{equation}
 and
 \begin{equation}\label{err:u:sum} 
	\begin{aligned}
        \| \widetilde{e}_{u}^{m+1/2}\|_{\rm M}^2
		 \le K_u \Bigl( \sum_{n=1}^{m} \tau_{n}\|  \bm{d} {e}_c^{n} \|_{\rm TM}^2 
         +  \bigl(\max_{0 \le n \le m} \|\bm{d}  c_h^{n}  \|_\infty^2 +1\bigr) \sum_{n=0}^{m} \tau_{n+1/2}   \|\widetilde{e}_{u}^{n+1/2} \|_{\rm M}^2  
         + \tau^4 + h^4\Bigr),
		\end{aligned}
\end{equation}
\begin{equation}\label{err:v:sum} 
	\begin{aligned}
         \| \widetilde{e}_{v}^{m+1/2}\|_{\rm M}^2
         \leq K_v \Bigl(  \sum_{n=1}^{m} \tau_{n} \| \bm{d} {e}_c^{n} \|_{\rm TM}^2 +  \bigl(\max_{0 \le n \le m} \|\bm{d}  c_h^{n}  \|_\infty^2 +1\bigr) \sum_{n=0}^{m}  \tau_{n+1/2}  \|\widetilde{e}_{v}^{n+1/2} \|_{\rm M}^2 
         + \tau^4 + h^4\Bigr),
	\end{aligned}
\end{equation}
for $1\leq m \leq N-1$.
\end{lemma}
\begin{proof} 
First, we present estimates for $\|\widetilde{e}_{u}^{1/2}\|_{\rm M}$ and $\|\widetilde{e}_{v}^{1/2}\|_{\rm M}$.
Subtracting \eqref{model:KS:sheme11} from \eqref{eq:Ru_0} yields the following error equation
\begin{equation}\label{err:PC-BCFD:rho1} 
	\begin{aligned}
    \f{\widetilde{e}_{u}^{1/2} }{\tau_0} 
    &=  D_x ( d_x \widetilde{e}_{u}^{1/2} ) 
		+  D_y ( d_y \widetilde{e}_{u}^{1/2}  )    -  \Lambda^{1/2} +  \mrR_{u}^{1/2}, \quad \text{on} ~\Pi_x^*\times \Pi_y^*,
	\end{aligned}
\end{equation}
where  
\begin{equation*}\label{err:PC-BCFD:lbda} 
	\begin{aligned}
\Lambda^{1/2}
		&:=    D_x \big( [\ell_h {u}_{\epsilon}^{1/2}] [ d_x c^0  ]
		    - [\ell_h \widetilde{u}_h^{1/2}] [d_x c_h^0] \big)  
            +  D_y \big( [\ell_h {u}_{\epsilon}^{1/2}] [ d_y c^0 ]
              - [\ell_h \widetilde{u}_h^{1/2}] [d_y c_h^0] \big)\\
        &=    D_x \big(  [\ell_h \widetilde{e}_{u}^{1/2}] [d_x c_h^0] \big)  +  D_y \big(  [\ell_h \widetilde{e}_{u}^{1/2}] [d_y c_h^0] \big),
  \end{aligned}
\end{equation*}
as $c_h^0=c^0$ at all grid points $\Pi_x^*\times \Pi_y^*$.

Then, taking discrete inner product of \eqref{err:PC-BCFD:rho1} with $\tau_0\, \widetilde{e}_u^{1/2}$ and applying Lemma \ref{lemma:Dd} we obtain
 \begin{equation}\label{inn:PC-BCFD:rho1} 
	\begin{aligned}
	   &	\|\widetilde{e}_{u}^{1/2}\|_{\rm M}^2 + \tau_0  \|\bm{d}\widetilde{e}_{u}^{1/2}\|_{\rm TM}^2 \\
       &\quad = 
         \tau_0 \Big[\big( [\ell_h \widetilde{e}_{u}^{1/2}] [d_x c_h^0], d_x \widetilde{e}_{u}^{1/2} \big)_x
          +  \big( [\ell_h \widetilde{e}_{u}^{1/2}] [d_y c_h^0], d_y \widetilde{e}_{u}^{1/2} \big)_y \Big]
		  +  \tau_0 \big( \mrR_{u}^{1/2},\widetilde{e}_{u}^{1/2}\big)_{\rm M}   =:\sum_{i=1}^{2} I_i. 
	\end{aligned}
\end{equation}
Thus, the application of Cauchy--Schwarz inequality yields the following bounds 
\begin{equation}\label{est:pc-BCFD:rho1:rhs3}
	\begin{aligned} 
		|I_1 |
        &\leq   \f{\tau_0 \|\bm{d} c_h^0\|_\infty^2}{4}\, \|  \widetilde{e}_{u}^{1/2}|_{\rm M}^2 
                    + \tau_0  \|\bm{d} \widetilde{e}_u^{1/2}\|_{\rm TM}^2
          \leq   \f{\tau_0 K_*^2}{4}\, \|  \widetilde{e}_{u}^{1/2}|_{\rm M}^2 
                    + \tau_0  \|\bm{d} \widetilde{e}_u^{1/2}\|_{\rm TM}^2          ,
	\end{aligned}
\end{equation}
and
\begin{equation}\label{est:pc-BCFD:rho1:rhs4bf}
   | I_2 | \leq  \tau_0^2 \|R_{u}^{1/2}\|_ {\rm M}^2 + \f{1}{4}\|\widetilde{e}_{u}^{1/2}\|_ {\rm M}^2.
\end{equation}

Inserting \eqref{est:pc-BCFD:rho1:rhs3}--\eqref{est:pc-BCFD:rho1:rhs4bf} into \eqref{inn:PC-BCFD:rho1}, for a sufficiently small chosen stepsize 
$\tau_{*} :=1/K_*^2$, 
we obtain from Lemma \ref{lem:truncErr} that 
\begin{equation*}\label{err:step1}
\| \widetilde{e}_{u}^{1/2}\|_{\rm M}^2  \leq 2\tau_0^2\|R_{u}^{1/2}\|_ {\rm M}^2\leq K_{u}\big(\tau_0^4 + h^4 \big), \quad \text{for} ~\tau_0 \le \tau_{*}.
\end{equation*}
Similarly, we can obtain the estimate for $\|\widetilde{e}_{v}^{1/2}\|_{\rm M}$ that
\begin{equation*}
    \| \widetilde{e}_{v}^{1/2}\|_{\rm M}^2 
    \leq  K_{v}\big( \tau_{0}^4 + h^4\big), \quad \text{for} ~\tau_0 \le \tau_{*}.
\end{equation*}
Thus, the conclusion \eqref{err:uv:0} is proved.

Next, we proceed to bound $\|\widetilde{e}_u^{m+1/2}\|_{\rm M}$ and $\|\widetilde{e}_v^{m+1/2}\|_{\rm M}$ for $m \ge 1$. For brevity, we provide the details only for \eqref{err:u:sum}; the estimate for \eqref{err:v:sum} follows by a verbatim argument. Subtracting \eqref{model:KS:shemeIa} from \eqref{eq:Ru_n} yields the following error equation
\begin{equation}\label{err:CN-BCFD:rho} 
	\begin{aligned}
		 \f{\widetilde{e}_{u}^{n+1/2}-e_{u}^{n-1/2}}{\tau_{n}}
		& =  D_x ( d_x \overline{\widetilde{e}}_{u}^{n}  ) 
		+  D_y ( d_y \overline{\widetilde{e}}_{u}^{n} )
  -  \Lambda^{n} +  R_{u}^{n}, \quad \text{on} ~\Pi_x^*\times \Pi_y^*,
	\end{aligned}
\end{equation}
where 
 \begin{equation*}
 		\Lambda^{n} 
 		:=  \bigl[D_x ( [\ell_h \overline{u}_{\epsilon}^{n}] [ d_x c^{*,n}] - [\ell_h \overline{\widetilde{u}}_h]^{n} [d_x c_h]^{*,n}) \bigr] 
        + \bigl[D_y ( [\ell_h \overline{u}_{\epsilon}^{n}] [ d_y c^{*,n} ] - [\ell_h \overline{\widetilde{u}}_h^{n}] [d_y c_h^{*,n}]) \bigr].
 \end{equation*}
 
Then, taking the discrete inner product of \eqref{err:CN-BCFD:rho} with $2\tau_n  \overline{\widetilde{e}}_{u}^{n} =\tau_n (\widetilde{e}_u^{n+1/2} + e_u^{n-1/2})$ and using Lemma \ref{lemma:Dd}, we obtain
\begin{equation}\label{inn:CN-BCFD:rho} 
   \begin{aligned}
	   & \| \widetilde{e}_{u}^{n+1/2}\|_{\rm M}^2 -\| {e}_{u}^{n-1/2}\|_{\rm M}^2 + 2\tau_{n} \|\bm{d} \overline{\widetilde{e}}_{u}^{n}\|_{\rm TM}^2 \\
      & \quad = 2\tau_{n} \Big( \big([\ell_h \overline{u}_{\epsilon}^{n}]  [ d_x c^{*,n}]  - [\ell_h \overline{\widetilde{u}}_h^{n}] [d_x c_h^{*,n}],d_x \overline{\widetilde{e}}_{u}^{n} \big)_x   
        + \big( [\ell_h \overline{u}_{\epsilon}^{n}] [ d_y c^{*,n} ] - [\ell_h \overline{\widetilde{u}}^{n}] [d_y c_h^{*,n}] , d_y \overline{\widetilde{e}}_{u}^{n} \big)_y \Big) \\
      & \qquad +   2\tau_{n}\big( R_{u}^{n},\overline{\widetilde{e}}_{u}^{n}\big)_{\rm M} =:\sum_{i=1}^{2} J_i.
	\end{aligned} 
\end{equation}
For the first right term $J_1$, we have
\begin{equation*}\label{est:erho:rhs3}
	\begin{aligned} 
   J_1 &= 2\tau_{n} \Big(( [\ell_h \overline{u}_{\epsilon}^{n}]  [ d_x e_c^{*,n}]  + [\ell_h \overline{\widetilde{e}}_{u}^{n}] [d_x c_h^{*,n}],d_x \overline{\widetilde{e}}_{u}^{n} \big)_x +  ( [\ell_h \overline{u}_{\epsilon}^{n}]  [ d_y e_c^{*,n} ]  + [\ell_h \overline{\widetilde{e}}_{u}^{n}] [d_y c_h^{*,n}],d_y \overline{\widetilde{e}}_{u}^{n} \big)_y\Big), 
 \end{aligned}
\end{equation*}
which can further be estimated by Cauchy--Schwarz inequality and Lemma \ref{lem:proj:e} that
\begin{equation}\label{est:erho:rhs3f}
	|J_1 | \leq  K \tau_n  \big( \|  \bm{d} {e}_c^{*,n} \|_{\rm TM}^2
		+   \|\bm{d} c_h^{*,n}\|_\infty^2 \|\widetilde{e}_{u}^{n+1/2} \|_{\rm M}^2\big)  
		 + 2\tau_n \,\|\bm{d} \overline{\widetilde{e}}_{u}^{n}\|_{\rm TM}^2.
\end{equation}
Moreover, for the second right term $J_2$, we have
\begin{equation}\label{est:erho:rhs4f}
    |J_2| \leq 2\tau_n \|R_{u}^{n}\|_{\rm M}^2 + \f{\tau_{n}}{2}\|\widetilde{e}_{u}^{n}\|_{\rm M}^2.
\end{equation}
Due to the regular assumption on the temporal mesh, we see
\begin{equation}\label{est:erho:rhs5f}
\begin{aligned}
      \|  \bm{d} {e}_c^{*,n} \|_{\rm TM} \le  K_{\sigma} \big( \| \bm{d} {e}_c^{n-1}\|_{\rm TM} +  \| \bm{d} {e}_c^{n}\|_{\rm TM}\big), \quad 
       \|\bm{d} c_h^{*,n}\|_\infty \le K_{\sigma}\big( \| \bm{d} {c}_h^{n-1}\|_\infty +  \| \bm{d} {c}_h^{n}\|_\infty\big),
\end{aligned}
\end{equation}
where the positive constant $K_{\sigma}:={(\sigma^*+3)}/{4}$ depends only on $\sigma^*$.

Therefore, by inserting \eqref{est:erho:rhs3f}–\eqref{est:erho:rhs5f} and Lemma \ref{lem:truncErr} into \eqref{inn:CN-BCFD:rho} and summing over $n$ from $1$ to $m$ ($1 \leq m \leq N-1$), we arrive at the desired conclusion. The proof is thus complete.
\end{proof}

\begin{remark}  It should be noted that the estimates \eqref{err:u:sum}--\eqref{err:v:sum} in Lemma \ref{lem:erho} require the uniform boundedness of $\|\bm{d} c_h^{n}\|_\infty$ for all $n$. This boundedness will be proved by mathematical induction in Theorem \ref{thm:coverg}. Moreover, these estimates also rely on the estimate $\sum_{n=1}^{m} \tau_{n}\|  \bm{d} {e}_c^{n} \|_{\rm TM}^2$, the proof of which is deferred to the next lemma.
\end{remark}

\begin{lemma}\label{lem:ec} Assume that the exact solutions satisfy the regularity condition \eqref{eq:regular}. Then, there exists a positive constant $K_c$, independent of $\tau$ and $h$, such that
\begin{equation}\label{lem:ec:sum}
 	\begin{aligned}
		\|{e}_{c}^{m+1}\|_{\rm M}^2 + \| \bm{d} {e}_{c}^{m+1} \|_{\rm TM}^2  \leq  K_c\Bigl(\sum_{n=0}^{m}\tau_{n+1/2} \bigl(\|e_{u}^{n+1/2}\|_{\rm M}^2  +  \|e_{v}^{n+1/2}\|_{\rm M}^2\bigr) + \tau^4 + h^4\Bigr), 
	\end{aligned}
\end{equation}
for $0\leq m \leq N-1$. 
\end{lemma}
\begin{proof} 
Subtracting \eqref{model:KS:shemeI3} from \eqref{eq:Rc_n} yields the following error equation
 \begin{equation}\label{err:CN-BCFD:c} 
	\begin{aligned}
		d_{\tau}e_c^{n+1/2}
		& =  D_x (d_x \overline{e}_{c}^{n+1/2}  )  
		+  D_y ( d_y \overline{e}_{c}^{n+1/2}  ) -  \overline{e}_{c}^{n+1/2}   + e_{u}^{n+1/2} + e_{v}^{n+1/2} +  \mrR_c^{n+1/2}.
	\end{aligned}
\end{equation}

Then, taking the discrete inner product of \eqref{err:CN-BCFD:c} with $d_{\tau}e_c^{n+1/2}$ and  applying Lemma \ref{lemma:Dd}, we obtain
\begin{equation}\label{est:Dec:inn}
\begin{aligned}
	\|d_{\tau}e_c^{n+1/2}\|_{\rm M}^2 
	& =- \big[\bigl( d_x \overline{e}_{c}^{n+1/2}, d_x d_{\tau}e_c^{n+1/2}\bigr)_x - \bigl( d_y \overline{e}_{c}^{n+1/2}, d_y d_{\tau}e_c^{n+1/2}\bigr)_y \big] \\
	& \quad -(\overline{e}_c^{n+1/2},  d_{\tau}e_c^{n+1/2})_{\rm M} 
	    +(e_{u}^{n+1/2},  d_{\tau}e_c^{n+1/2})_{\rm M} +(e_{v}^{n+1/2},  d_{\tau}e_c^{n+1/2})_{\rm M} \\
     & \quad  + (\mrR_c^{n+1/2},  d_{\tau}e_c^{n+1/2})_{\rm M}
   =: \sum_{i=1}^5 S_i.
\end{aligned}
\end{equation}
The right-hand side of \eqref{est:Dec:inn} is estimated as follows. For the first two terms, the following estimates are obtained
 \begin{equation}\label{est:Dec:rhs1}
 	\begin{aligned}
 		S_1 &=  -\f{1}{2\tau_{n+1/2}} \big[\bigl(\|d_x e_c^{n+1}\|_x^2 - \|d_x e_c^{n}\|_x^2\bigr)+ \bigl(\|d_y e_c^{n+1}\|_y^2 - \|d_y e_c^{n}\|_y^2\bigr) \big]\\
        &= -\f{1}{2\tau_{n+1/2}} \big[\| \bm{d} {e}_{c}^{n+1}\|_{\rm TM}^2-\| \bm{d} {e}_{c}^{n}\|_{\rm TM}^2\big],
 	\end{aligned}
 \end{equation}
\begin{equation}\label{est:Dec:rhs2}
	\begin{aligned}
		S_2 &= -\big( \overline{e}_{c}^{n+1/2}, d_{\tau}e_c^{n+1/2}\big)_{\rm M}
        = -\f{1}{2\tau_{n+1/2}} \big(\|e_c^{n+1}\|_{\rm M}^2 - \|e_c^{n}\|_{\rm M}^2\big).
	\end{aligned}
\end{equation}
The last three terms can be bounded by the Cauchy-Schwarz inequality and Young's inequality that
\begin{equation}\label{est:Dec:rhs3} 
	|S_3|+|S_4|+|S_5|
	\leq \|e_{u}^{n+1/2}\|_{\rm M}^2 +\|e_{v}^{n+1/2}\|_{\rm M}^2 +\f{1}{2}\|\mrR_c^{n+1/2}\|_{\rm M}^2  + \|d_{\tau}e_c^{n+1/2}\|_{\rm M}^2.
\end{equation}

Therefore, by inserting \eqref{est:Dec:rhs1}--\eqref{est:Dec:rhs3} into \eqref{est:Dec:inn}, multiplying the resulting equation by $2\tau_{n+1/2}$, 
and summing over $n$ from $0$ to $m$ for $0\leq m \leq N-1$, we obtain
 \begin{equation*}
	\begin{aligned}
		\| {e}_{c}^{m+1}\|_{\rm M}^2 + \| \bm{d} {e}_{c}^{m+1} \|_{\rm TM}^2 &\leq \| {e}_{c}^{0}\|_{\rm M}^2 + \| \bm{d} {e}_{c}^{0}\|_{\rm TM}^2 +  2\sum_{n=0}^{m} \tau_{n+1/2} \big(\|e_{u}^{n+1/2}\|_{\rm M}^2 +\|e_{v}^{n+1/2}\|_{\rm M}^2\big)\\ 
           &\quad  + \sum_{n=0}^{m}\tau_{n+1/2} \|\mrR_c^{n+1/2}\|_{\rm M}^2.
	\end{aligned}
\end{equation*}
Thus, the conclusion \eqref{lem:ec:sum} is proved by collecting the estimate in Lemma \ref{lem:truncErr}.
\end{proof}


Now, by combining the results of Lemmas \ref{lem:erho} and \ref{lem:ec} together, and using the mathematical induction method, we proceed to prove the main convergence result.
\begin{theorem}\label{thm:coverg}	Let $\{u_h^{n+1/2},v_h^{n+1/2},c_h^{n+1}\}\in \mbX_h[u] \times \mbX_h[v] \times \mbX_h$ be the solutions to the PP-MC-PBCFD scheme \eqref{model:KS:shemeI}--\eqref{model:KS:shemeI3}. Then, if $\tau \leq \min\{\tau_*,\tau_{**}\}$,  there exist unique solutions to the proposed scheme. Moreover, under the regularity assumptions \eqref{eq:regular}--\eqref{eq:regular:1} and the condition $\tau \leq K_0 h$ for some $K_0>0$, if $\tau\leq \hat{\tau}:=\min\{\tau_{*},\tau_{**},\tau_{***}\}$ and $h\le  \hat{h} :=\min\{h_*,h_{**}\} $, then there exists a positive constant $K^{*}$, independent of $\tau$, $h$ and $n$, such that
\begin{equation} \label{thm:converg:result}
		\begin{aligned}
			&\| c^{n+1}-c_h^{n+1}\|_{\rm M} 
			+  \| \nabla c^{n+1} - \bm{d}c_h^{n+1}\|_{\rm TM} + \| u^{n+1/2}-u_h^{n+1/2}\|_{\rm M} + \| v^{n+1/2}-v_h^{n+1/2}\|_{\rm M}\\
			&\quad \leq K^{*} \big(\tau^2 + h^2\big),~~ 0\leq n \leq N-1.
		\end{aligned}
\end{equation}
\end{theorem}
\begin{proof} To obtain the desired estimate \eqref{thm:converg:result}, it suffices to bound the four terms $\|e_{c}^{n+1}\|_{\rm M}$, $\|\bm{d}e_{c}^{n+1}\|_{\rm TM}$, $\| {e}_{u}^{n+1/2}\|_{\rm M}$, and $\| {e}_{v}^{n+1/2}\|_{\rm M}$. Indeed, invoking the triangle inequality together with the interpolation condition \eqref{Operator:bi}--which guarantees that $\{\ell_h u^{n+1/2},\ell_h v^{n+1/2},\ell_h c^{n+1}\}$ coincides with $\{u^{n+1/2},v^{n+1/2},c^{n+1}\}$ on $\Pi_x^*\times \Pi_y^*$, we have 
\begin{equation}\label{est:thm:e1}
   \begin{aligned}
     \|c^{n+1}-c_h^{n+1}\|_{\rm M} 
     & 
     =\|e_{c}^{n+1}\|_{\rm M},\\
      \| \nabla c^{n+1} - \bm{d}c_h^{n+1}\|_{\rm TM} 
      & \le 
       \|\bm{d}e_{c}^{n+1}\|_{\rm TM} +\| \nabla c^{n+1} - \bm{d}c^{n+1}\|_{\rm TM},\\
      \|u^{n+1/2}-u_h^{n+1/2}\|_{\rm M}  
      & \le \|e_u^{n+1/2}\|_{\rm M} + \|\epsilon_u^{n+1/2}\ell_h u^{n+1/2}\|_{\rm M},\\
       \|v^{n+1/2}-v_h^{n+1/2}\|_{\rm M} 
    & \le \|e_v^{n+1/2}\|_{\rm M} + \|\epsilon_v^{n+1/2}\ell_h v^{n+1/2}\|_{\rm M}.
    \end{aligned}
\end{equation}
Moreover, using the estimates  \eqref{eq:epsilon}  for $\epsilon_u^{n+1/2}$ and $\epsilon_v^{n+1/2}$ together with the regularity assumption \eqref{eq:regular}, we obtain 
\begin{equation}\label{est:thm:e2}
    \begin{aligned}
    \|\epsilon_u^{n+1/2} \ell_h u^{n+1/2}\|_{\rm M}  &\leq   \|\epsilon_u^{n+1/2}\|_{\rm M} \|u\|_{L^{\infty}(0,T;L^{\infty})} \le K h^2,\\
     \|\epsilon_v^{n+1/2} \ell_h v^{n+1/2}\|_{\rm M} &\leq  \|\epsilon_v^{n+1/2}\|_{\rm M} \|v\|_{L^{\infty}(0,T;L^{\infty})} \le K h^2,\\        
        \|\nabla c^{n+1} - \bm d c^{n+1}\|_{\rm TM} &\leq K h^2.
    \end{aligned}
\end{equation}
Consequently, combining \eqref{est:thm:e1} and \eqref{est:thm:e2} together yields
\begin{equation}\label{est:thm:e3}
    \begin{aligned}
        &\|c^{n+1}-c_h^{n+1}\|_{\rm M} + \| \nabla c^{n+1} - \bm{d}c_h^{n+1}\|_{\rm TM}+\|u^{n+1/2}-u_h^{n+1/2}\|_{\rm M} + \|v^{n+1/2}-v_h^{n+1/2}\|_{\rm M}\\
        &\quad \leq \|e_{c}^{n+1}\|_{\rm M}   +\|\bm{d}e_{c}^{n+1}\|_{\rm TM} +\|e_u^{n+1/2}\|_{\rm M}  + \|e_v^{n+1/2}\|_{\rm M} +  K h^2.
    \end{aligned}
\end{equation}
Therefore, in what follows, we shall prove the unique solvability and establish error estimates for 
the right-hand side of \eqref{est:thm:e3} sequentially by mathematical induction, demonstrating that $\|\bm{d} c_h^{n}\|_\infty$ is uniformly bounded. More specifically, at each time step, assuming the boundedness result $\|\bm{d} c_h^{\ell}\|_\infty \leq K_*+1$ for all $\ell \le n$, we proceed in three steps: first, we prove the uniqueness of the solutions to the PP-MC-PBCFD scheme; second, we derive the optimal-order error estimates; and third, we show that $\|\bm{d} c_h^{n+1}\|_\infty \leq K_*+1$ holds, thereby closing the induction loop.

Noting that for a finite-dimensional square linear algebraic system, the uniqueness of the solution also implies the existence. Therefore, we will pay attention to the proof of uniqueness of solutions to the PP-MC-PBCFD scheme \eqref{model:KS:schemeI0}--\eqref{model:KS:shemeI3}. Assume $\{\widehat{u}_h^{n+1/2},\widehat{v}_h^{n+1/2}, \widehat{c}_h^{n+1}\}\in \mbX_h[u] \times \mbX_h[v] \times \mbX_h$ is another solution triple with the same initial values $\{u_h^{0},v_h^{0},c_h^{0}\}
\in \mbX_h[u] \times \mbX_h[v] \times \mbX_h$, and define $\varepsilon_u^{n+1/2}:=u_h^{n+1/2}-\widehat{u}_h^{n+1/2}$, $\varepsilon_v^{n+1/2}:=v_h^{n+1/2}-\widehat{v}_h^{n+1/2}$, and $\varepsilon_c^{n+1}:=c_h^{n+1}-\widehat{c}_h^{n+1}$ for $n \ge 0$. 

\paragraph{\bf Part I. Unique solvability and error estimate for $n=0$} First, for the initial time it holds that $\varepsilon_u^{0}=\varepsilon_v^{0}=\varepsilon_c^{0} \equiv 0$ and
$
	\|\bm{d} c_h^0\|_\infty = \|\bm{d} c^0\|_\infty \leq K_*.
$
It follows from \eqref{model:KS:schemeI0} that
    \begin{align}
        \f{1}{\tau_{0}}{\varepsilon_u^{1/2}} & = D_x(d_x {\varepsilon_u^{1/2}}) 
           + D_y(d_y {\varepsilon_u})  
           -  D_x \bigl([\ell_h {\varepsilon_u^{1/2}}][d_x c_h^{0}]\bigr) 
        -  D_y \bigl([\ell_h {\varepsilon_u^{1/2}}][d_y c_h^{0}]\bigr),\label{sol:eq1}\\
		\f{1}{\tau_{0}}{\varepsilon_v^{1/2}} & = D_x(d_x {\varepsilon_v^{1/2}}) 
           + D_y(d_y {\varepsilon_v^{1/2}})  
           -  D_x \bigl([\ell_h {\varepsilon_v^{1/2}}] [d_x c_h^{0}]\bigr) 
        -  D_y \bigl([\ell_h {\varepsilon_v^{1/2}}][d_y c_h^{0}]\bigr).\label{sol:eq2}
    \end{align}
    Then, taking the discrete inner product of \eqref{sol:eq1} with $\tau_0\, {\varepsilon_u^{1/2}}$, we have
    \begin{equation*}
        \begin{aligned}
        \|{\varepsilon_u^{1/2}}\|_{\rm M}^2  + \tau_0 \| \bm{d}{\varepsilon_u^{1/2}} \|_{\rm TM}^2 
        &=  \tau_0 \Big[\bigl( [\ell_h {\varepsilon_u^{1/2}}] [d_x c_h^{0}], d_x {\varepsilon_u^{1/2}} \bigr)_x 
        +   \bigl( [\ell_h {\varepsilon_u^{1/2}}] [d_y c_h^{0}], d_y {\varepsilon_u^{1/2}} \bigr)_y \Big]\\
        &\leq \f{\tau_0 \|\bm{d} c_h^{0}\|_\infty^2}{4}  \|{\varepsilon_u^{1/2}}\|_{\rm M}^2
        + \tau_0 \| \bm{d}{\varepsilon_u^{1/2}} \|_{\rm TM}^2\\
        &\le \f{\tau_0 K_*^2}{4}  \|{\varepsilon_u^{1/2}}\|_{\rm M}^2
        + \tau_0 \| \bm{d}{\varepsilon_u^{1/2}} \|_{\rm TM}^2,
        \end{aligned}
    \end{equation*}
    which implies that
\begin{equation*}\label{uniqu:Un}
    \Big(1-\f{\tau_0 K_*^2}{4}\Big) \|{\varepsilon_u^{1/2}}\|_{\rm M}^2 \leq 0
    \Longrightarrow {\varepsilon_u^{1/2}} = 0 \Longrightarrow u_h^{1/2}=\widehat{u}_h^{1/2},\quad \text{for}~\tau_0 \leq \tau_*  <4/K_*^2.
\end{equation*}
    Similarly, taking the discrete inner product of \eqref{sol:eq2} with $\tau_0 \varepsilon_v^{1/2}$, we can obtain
    \begin{equation*}\label{uniqu:Vn}
    \Big(1-\f{\tau_0 K_*^2}{4}\Big) \|\varepsilon_v^{1/2}\|_{\rm M}^2 \leq 0
    \Longrightarrow \varepsilon_v^{1/2} = 0 \Longrightarrow v_h^{1/2}=\widehat{v}_h^{1/2},\quad \text{for}~\tau_0 \leq \tau_*.
    \end{equation*}   
Furthermore, by \eqref{model:KS:shemeI3} with $n=0$  and the proved uniqueness result of $\{u_h^{1/2},v_h^{1/2}\}$, we have  
    \begin{align}
        \f{1}{\tau_{1/2}}\varepsilon_c^{1} &= \f{1}{2} D_x(d_x \varepsilon_c^{1}) + \f{1}{2} D_y(d_y \varepsilon_c^{1})
        - \f{1}{2}\varepsilon_c^{1}.\label{sol:eq3}
    \end{align}
Thus, taking the discrete inner product of \eqref{sol:eq3} with $\tau_{1/2}\varepsilon_c^1$, we obtain
\begin{equation*}\label{uniqu:cn}
    (1+\f{\tau_{1/2}}{2}) \|\varepsilon_c^{1}\|_{\rm M}^2  + \f{\tau_{1/2}}{2} \| \bm{d}\varepsilon_c^1 \|_{\rm TM}^2 = 0 
        \Longrightarrow \varepsilon_c^1 = 0 \Longrightarrow c_h^1=\widehat{c}_h^1.
\end{equation*}

Next, Lemma \ref{lem:ec} with $m=0$, and Lemmas \ref{lem:proj:e}--\ref{lem:erho} directly implies that
 \begin{equation}\label{reslt:thm:t1c}
	\| {e}_{c}^{1}\|_{\rm M}^2 	+ \| \bm{d} {e}_{c}^{1} \|_{\rm TM}^2  
       \leq  K_c\Bigl(\tau_{1/2} \bigl(\|\widetilde{e}_{u}^{1/2}\|_{\rm M}^2  +  \|\widetilde{e}_{v}^{1/2}\|_{\rm M}^2\bigr) + \tau^4 + h^4\Bigr)
       \leq K_1\big(\tau^4 + h^4\big),\quad \text{for}~\tau_{1/2} \le 1,
\end{equation}
where $K_1:=K_cK_u + K_cK_v + K_c$ is a constant  that is independent of $\tau$ and $h$, depending only on $K_c$, $K_u$ and $K_v$. Thus, it is straightforward to derive the following estimate 
  \begin{equation}\label{reslt:thm:t1}
		\| {e}_{c}^{1}\|_{\rm M}^2 
		+ \| \bm{d} {e}_{c}^{1} \|_{\rm TM}^2 + \| {e}_{u}^{1/2}\|_{\rm M}^2 + \| {e}_{v}^{1/2}\|_{\rm M}^2 \leq K_2 \big(\tau^4 + h^4 \big), \quad \text{for}~ \tau\le \tau_{*}, 
\end{equation}
by combining \eqref{err:uv:0} of Lemma \ref{lem:erho}, Lemma \ref{lem:proj:e}, and \eqref{reslt:thm:t1c} together,
where $K_2:= K_1 + K_u + K_v$ is a constant independent of $\tau$ and $h$, depending only on $K_c$, $K_u$ and $K_v$.

Finally, note that \eqref{reslt:thm:t1c} directly implies that
\begin{equation*}
		 \| \bm{d} {e}_{c}^{1} \|_{\rm TM}  
     \leq \sqrt{K_1} \big( \tau^2 + h^2\big).
\end{equation*}
Then, applying the triangle inequality, the inverse estimate with constant $K_{inv}$, and assumption \eqref{eq:regular:1}, we obtain
\begin{equation}\label{inf:dc}
	\begin{aligned}
	\|\bm{d} c_h^1\|_\infty &\leq  \| \bm{d} c^{1} -\bm{d}e_c^1 \|_\infty
	\leq K_* + \|\bm{d}e_c^1\|_\infty
 	\leq K_*+ K_{inv} \sqrt{K_1}\big( h^{-1}\tau^2 + h \big)
    \leq K_* + 1,
 	\end{aligned}
\end{equation}
provided that $\tau \leq K_0 h$ for some $K_0>0$ and $0< h \le h_*$, where the positive constant $h_*$ is chosen such that 
$
 K_{inv}\sqrt{K_1} \big(1 + K_0^2\big) h_* \leq 1.
$

\paragraph{\bf Part II. Unique solvability and error estimate for $n \ge 1$}
Now, suppose that $\|\bm{d} c_h^{\ell}\|_{\infty} \leq K_* + 1$ for all $\ell\le n$ with some $n \ge 1$ and $\varepsilon_u^{n-1/2}=\varepsilon_v^{n-1/2}=\varepsilon_c^n=\varepsilon_c^{n-1} = 0$ hold. 

In the following, we first prove that $\varepsilon_u^{n+1/2}=\varepsilon_v^{n+1/2}=\varepsilon_c^{n+1} = 0$ to conclude the unique solvability.
It follows from \eqref{model:KS:shemeI} and \eqref{model:KS:shemeI3} that
\begin{align}
    \f{1}{\tau_{n}}{\varepsilon_u^{n+1/2}} 
    & = \f{1}{2} D_x(d_x {\varepsilon_u^{n+1/2}}) 
           + \f{1}{2} D_y(d_y {\varepsilon_u^{n+1/2} }) 
           - \f{1}{2} D_x \bigl([\ell_h {\varepsilon_u^{n+1/2}}][d_x c_h^{*,n}]\bigr) 
        \label{sol:equ}\\
        &\quad - \f{1}{2} D_y \bigl([\ell_h {\varepsilon_u^{n+1/2}}][d_y c_h^{*,n}]\bigr),\nonumber \\
		\f{1}{\tau_{n}}{\varepsilon_v^{n+1/2}} & = \f{1}{2} D_x(d_x {\varepsilon_v^{n+1/2}}) 
           + \f{1}{2} D_y(d_y {\varepsilon_v^{n+1/2}})  
           - \f{1}{2} D_x \bigl([\ell_h {\varepsilon_v^{n+1/2}}][d_x c_h^{*,n}]\bigr) \label{sol:eqv}\\
        &\quad-  \f{1}{2} D_y \bigl([\ell_h {\varepsilon_v^{n+1/2}}][d_y c_h^{*,n}]\bigr),\nonumber\\
        \f{1}{\tau_{n+1/2}}\varepsilon_c^{n+1} &= \f{1}{2} D_x(d_x \varepsilon_c^{n+1}) + \f{1}{2} D_y(d_y \varepsilon_c^{n+1})
        - \f{1}{2}\varepsilon_c^{n+1} + \varepsilon_u^{n+1/2} + \varepsilon_v^{n+1/2}, \label{sol:eqc}
    \end{align}
where we have used the fact that $\bm{d} \widehat{c}_h^{*,n}= \bm{d} c_h^{*,n}= \frac{t_n-t_*}{\tau_{n-1/2}} \bm{d} c_h^{n-1}+\frac{t_*-t_{n-1}}{\tau_{n-1/2}} \bm{d} c_h^n$ due to the unique solvability at the former time levels.

Then, by taking the discrete inner product of \eqref{sol:equ} with $\tau_n {\varepsilon_u^{n+1/2}}$, and applying the Cauchy-Schwarz inequality, Young's inequality, and Lemma \ref{lemma:Dd} yield  
\begin{equation}\label{uniqu:Un+1:est}
        \begin{aligned}
        &\|{\varepsilon_u^{n+1/2}}\|_{\rm M}^2  + \f{\tau_n}{2} \| \bm{d}{\varepsilon_u^{n+1/2}} \|_{\rm TM}^2\\ 
        &= \f{\tau_n}{2} \left( [\ell_h \varepsilon_u^{n+1/2}] [d_x {c}_h^{*,n}], d_x {\varepsilon_u^{n+1/2}} \right)_x 
        + \f{\tau_n}{2}\left( [\ell_h \varepsilon_u^{n+1/2}] [d_y {c}_h^{*,n}], d_y {\varepsilon_u^{n+1/2}}\right)_y\\
        &\leq \f{\tau_n \|\bm{d} {c}_h^{*,n}\|_\infty^2}{8}  \|{\varepsilon_u^{n+1/2}}\|_{\rm M}^2
        + \f{\tau_n}{2} \| \bm{d}{\varepsilon_u^{n+1/2}} \|_{\rm TM}^2\\
        &\leq \f{\tau_n (K_*+1)^2}{4} \left[\Big( \f{1}{4}-\f{\tau_{n+1/2}}{4\tau_{n-1/2}}\Big)^2 + \Big( \f{3}{4}+\f{\tau_{n+1/2}}{4\tau_{n-1/2}}\Big)^2\right] \|{\varepsilon_u^{n+1/2}}\|_{\rm M}^2
        + \f{\tau_n}{2} \| \bm{d}{\varepsilon_u^{n+1/2}} \|_{\rm TM}^2,
        \\
        &\leq \f{\tau_n (K_*+1)^2}{32} \left( (\sigma^*)^2+ 2\sigma^*+ 5 \right) \|{\varepsilon_u^{n+1/2}}\|_{\rm M}^2
        + \f{\tau_n}{2} \| \bm{d}{\varepsilon_u^{n+1/2}} \|_{\rm TM}^2,
        \end{aligned}
    \end{equation}
 where the assumption $\tau_{n+1/2}/\tau_{n-1/2}\leq\sigma^*$ is used. This further implies that
\begin{equation*}\label{uniqu:Un+1}
    \Big(1-\f{\tau_n (({\sigma^*})^2+2\sigma^*+ 5)(K_*+1)^2}{32}\Big) \|{\varepsilon_u^{n+1/2}}\|_{\rm M}^2 \leq 0
    \Longrightarrow {\varepsilon_u^{n+1/2}} = 0 \Longrightarrow u_h^{n+1/2}=\widehat{u}_h^{n+1/2},
\end{equation*}
provided that $\tau \leq \tau_{**} < \f{32}{(({\sigma^*})^2+2\sigma^*+ 5)(K_*+1)^2}$.

Similarly, taking the discrete inner product of \eqref{sol:eqv} with $\tau_n {\varepsilon_v^{n+1/2}}$, and applying the Cauchy-Schwarz inequality, Young's inequality, and Lemma \ref{lemma:Dd} yield a similar result to \eqref{uniqu:Un+1:est} that
\begin{equation*}\label{uniqu:Vn+1}
    \Big(1-\f{\tau_n (({\sigma^*})^2+2\sigma^*+ 5)(K_*+1)^2}{32}\Big) \|{\varepsilon_v^{n+1/2}}\|_{\rm M}^2 \leq 0
    \Longrightarrow {\varepsilon_v^{n+1/2}} = 0 \Longrightarrow v_h^{n+1/2}=\widehat{v}_h^{n+1/2},
\end{equation*}
provided that $\tau \leq \tau_{**}$. 

Furthermore, taking the discrete inner product of \eqref{sol:eqc} with $\tau_{n+1/2}\varepsilon_c^{n+1}$, and using the proved uniqueness results of $\{u_h^{n+1/2},v_h^{n+1/2}\}$, we have
\begin{equation*}\label{uniqu:cn+1}
        \Big(1+\f{\tau_{n+1/2}}{2}\Big) \|\varepsilon_c^{n+1}\|_{\rm M}^2  + \f{\tau_{n+1/2}}{2} \| \bm{d}\varepsilon_c^{n+1} \|_{\rm TM}^2 
        =  0 
        \Longrightarrow \varepsilon_c^{n+1} = 0 \Longrightarrow c_h^{n+1}=\widehat{c}_h^{n+1}.
\end{equation*}
Consequently, the existence and uniqueness of the solutions $\{u_h^{n+1/2},v_h^{n+1/2},c_h^{n+1}\}$ are proved.

Next, under the bounded assumption on $\|\bm{d} c_h^{\ell}\|_{\infty}$ for all $\ell\le n$, by adding \eqref{err:u:sum}--\eqref{err:v:sum} from Lemma \ref{lem:erho} and \eqref{lem:ec:sum} from Lemma \ref{lem:ec} with $m=n$, and applying Lemma \ref{lem:proj:e}, we obtain
  \begin{equation} \label{est:ec:l-0}
	\begin{aligned}
		&\| {e}_c^{n+1}\|_{\rm M}^2 	
             +\| \bm{d} {e}_{c}^{n+1} \|_{\rm TM}^2 + \| \widetilde{e}_{u}^{n+1/2}\|_{\rm M}^2 +\| \widetilde{e}_{v}^{n+1/2}\|_{\rm M}^2 \\
		&\quad\leq K_3 \Big(\sum_{m=0}^{n}\tau_{m} \| \bm{d} {e}_{c}^{m} \|_{\rm TM}^2 
		+   \sum_{m=0}^{n}\tau_{m+1/2} \left(\|\widetilde{e}_{u}^{m+1/2}\|_{\rm M}^2 + \|\widetilde{e}_{v}^{m+1/2}\|_{\rm M}^2 \right)+ \tau^4 + h^4\Big),
	\end{aligned}
\end{equation}
where $K_3:=2(K_u+K_v) ((K_*+1)^2+1) + K_c$ is a constant independent of $\tau$ and $h$, depending only on $K_c$, $K_u$, $K_v$ and $K_*$. Thus, an application of the discrete Gr\"{o}nwall's inequality and Lemma \ref{lem:proj:e} to \eqref{est:ec:l-0} directly yields   
  \begin{equation} \label{est:ec:l-m+1}
	\begin{aligned}
		&\|e_{c}^{n+1}\|_{\rm M}^2+\|\bm{d}e_{c}^{n+1}\|_{\rm TM}^2+\| {e}_{u}^{n+1/2}\|_{\rm M}^2 + \| {e}_{v}^{n+1/2}\|_{\rm M}^2\\
        &\quad \le \| {e}_c^{n+1}\|_{\rm M}^2 	
             +\| \bm{d} {e}_{c}^{n+1} \|_{\rm TM}^2 + \| \widetilde{e}_{u}^{n+1/2}\|_{\rm M}^2 +\| \widetilde{e}_{v}^{n+1/2}\|_{\rm M}^2 \leq  K_4 \big( \tau^4 + h^4\big),
	\end{aligned}
\end{equation}
for $\tau\le \tau_{***}:=1/(2K_3)$, where $K_4:=2K_3e^{2K_3T}$ is a constant independent of $\tau$ and $h$, depending only on $K_c$, $K_u$, $K_v$, $K_*$ and  $T$. 

Finally, we show that  $\|\bm{d} c_h^{n+1}\|_{\infty} \leq K_*+1$.
Following the same approach as in \eqref{inf:dc} and using the estimate \eqref{est:ec:l-m+1}, we obtain
\begin{equation}\label{inf:dzm+1}
		\|\bm{d} c_h^{n+1}\|_\infty =\|\bm{d} c^{n+1}-\bm{d} e_c^{n+1}\|_\infty
    \leq K_* + K_{inv} \sqrt{K_4} \big(h^{-1}\tau^2 + h \big)
    \leq K_* + 1,
\end{equation}
provided that $\tau \leq K_0 h$ and $0<h<h_{**}$, where $h_{**}$ is chosen such that $K_{inv}\sqrt{K_4}\big(1 + K_0^2 \big) h_{**} \leq 1$. This completes the induction process for $\ell = n+ 1$. 
Consequently, the theorem is proved by choosing $\tau\leq \hat{\tau}:=\min\{\tau_{*},\tau_{**},\tau_{***}\}$ and $h\le  \hat{h} :=\min\{h_*,h_{**}\} $ sufficiently small, and then inserting \eqref{reslt:thm:t1} and \eqref{est:ec:l-m+1} into \eqref{est:thm:e3}. 
\end{proof} 

\begin{remark}
	From the proof of Theorem \ref{thm:coverg}, we have the following observations.
	\begin{itemize}
		\item[(i)] The discrete $L^2$ projection \eqref{model:KS:shemeI2} in {\bf Step 2} of the proposed scheme serves a dual purpose: it preserves positivity and enforces mass conservation. However, the correction step is not limited to this projection; similar results can be obtained using, for instance, the discrete $H^1$ projection \cite{TC'24}. 
        Moreover, the estimate established in Lemma \ref{lem:proj:e}, which follows directly from this projection, plays a crucial role in the error analysis.
        
		\item[(ii)] For the intermediate solutions $\widetilde{u}_h^{n+1/2}$ and $\widetilde{v}_h^{n+1/2}$, following the similar decomposition estimates \eqref{est:thm:e1}--\eqref{est:thm:e2} and using \eqref{est:ec:l-m+1}, we have       
      \begin{equation*} 
         \begin{aligned}
        &\|u^{n+1/2}-\widetilde{u}_h^{n+1/2}\|_{\mrM} + \|v^{n+1/2}-\widetilde{v}_h^{n+1/2}\|_{\mrM}\\
        &\quad \le \|\widetilde{e}_{u}^{n+1/2}\|_{\rm M} + 
        \|\widetilde{e}_{v}^{n+1/2}\|_{\rm M} + \|\epsilon_u^{n+1/2}\|_{\rm M} \|u\|_{L^{\infty}(0,T;L^{\infty})}
        + \|\epsilon_v^{n+1/2}\|_{\rm M} \|v\|_{L^{\infty}(0,T;L^{\infty})}\\
        &\quad \leq K \big(\tau^2 + h^2\big).
       \end{aligned}
     \end{equation*}

 	\item[(iii)] The proposed second-order PP-MC-PBCFD scheme \eqref{model:KS:schemeI0}–\eqref{model:KS:shemeI3} is also applicable to the three-dimensional (3D) Keller–Segel chemotaxis system, and the error analysis remains valid with a minor modification: the inverse inequality \eqref{inf:dzm+1} is replaced by its 3D counterpart
         \begin{equation*}\label{inf:3d:dzm+1}
		    \|\bm{d} c_h^{n+1}\|_\infty 
            \leq K_* + K_{inv} \sqrt{K_4} \big(h^{-3/2}\tau^2 + h^{1/2} \big) \leq K_* + 1,
         \end{equation*}
provided that $\tau \leq K_0 h$ and $h \le h_{***}$, with $h_{***}$ chosen such that $K_{inv}^2 K_4 \big(1 + K_0^2 \big)^2 h_{***} \leq 1$. 
	\end{itemize}
\end{remark}

\section{Numerical results}\label{sec:num}

In this section, we conduct several numerical experiments using the proposed PP-MC-PBCFD scheme (Eqs. \eqref{model:KS:schemeI0}–\eqref{model:KS:shemeI3}) to demonstrate its accuracy, verify the preservation of three important physical laws: positivity, mass conservation, and energy dissipation. Meanwhile, the blow‑up dynamics of the multi‑species Keller–Segel chemotaxis system on staggered non‑uniform spatio-temporal grids are also simulated. 

In the following, we introduce the non-uniform temporal and spatial grid partitions as follows:
\begin{equation}\label{eq:adap}
    \begin{aligned}
	t_{n} &= t_{{\rm fix},n} + \mu \, \Delta t  \, (-1 + 2\,rand),\quad n=1,\ldots,N-1,\\
    x_{i+1/2} &= x_{{\rm fix},i+1/2} + \nu \, \Delta x\, (-1+2\, rand),\quad i=1,\ldots,N_x-1,\\
    y_{j+1/2} &= y_{{\rm fix},j+1/2} + \omega \, \Delta y\, (-1+2\, rand),\quad j=1,\ldots,N_y-1,
    \end{aligned}
\end{equation}
where $t_{{\rm fix},n}= n\Delta t \,(n=0,\ldots,N)$, $x_{{\rm fix},i+1/2}=a^x+i\Delta x\,(i=0,\ldots,N_x)$, $y_{{\rm fix},j+1/2}=a^y+j\Delta y\,(j=0,\ldots,N_y)$, with uniform grid sizes $\Delta t = {T}/{N}$, $\Delta x= (b^x-a^x)/N_x$, and $\Delta y= (b^y-a^y)/N_y$.  Here, $\mu$, $\nu$ and $\omega$ are small mesh parameters that control the extent of the random mesh perturbations within a specific range.  In particular, when $\mu=\nu=\omega=0$, \eqref{eq:adap} defines a uniform spatio-temporal partition. The symbol $rand$ denotes a uniformly distributed random number in $[0,1]$. 
In what follows, we assume $N_x=N_y (=N_z) \equiv M$ and $\mu=\nu=\omega$.

\begin{example}[Accuracy test]\label{exam:s1}
In this example, we consider the two-species Keller--Segel chemotaxis model \eqref{model:KS} with source terms $(f_1,f_2,g)$ as follows:
\begin{equation*}
	\begin{aligned}
		  \p_t u  &= \Delta u- \nabla \cdot(u \nabla c) + f_1,  & \qquad \text{in}~~ \Omega \times(0, T],\\
		\p_t v  &= \Delta v- \nabla \cdot(v \nabla c) + f_2,  & \qquad \text{in}~~ \Omega \times(0, T],\\
		\p_t c  &=  \Delta c-  c+ u  +  v + g,  & \qquad \text{in}~~ \Omega \times(0, T],
	\end{aligned}
\end{equation*}
where the computational domain $\Omega=(-0.5,0.5)^2$ and $T=1$, and the manufactured exact solutions are taken as
\begin{equation*}
\begin{aligned}
	u(\bm x, t) &= \sin(\pi x)\sin(\pi y)\sin(t)+1.1,\\
	v(\bm x, t) &= \f{1}{2\pi^2+1} \sin(\pi x)\sin(\pi y)\sin(t)+1.1, \\
	c(\bm x,t) &= \sin(\pi x)\sin(\pi y)\sin(t) + 1.1.
\end{aligned}
\end{equation*}
\end{example}

This example is mainly used to test the accuracy of the PP-MC-PBCFD scheme on staggered non‑uniform spatio-temporal grids. To numerically evaluate both the spatial and temporal accuracy, we set the grid sizes in \eqref{eq:adap} such that $\Delta x=\Delta t=1/M$, and measure the discrete $L^2$ errors for the approximations of $u$, $v$, $c$ and $\nabla c$. As shown in Table \ref{tab:accuracy}, the scheme clearly exhibits second‑order convergence in both time and space on uniform grids (i.e., $\mu=0$), which is fully consistent with the theoretical conclusion in Theorem \ref{thm:coverg}.  Furthermore, for small mesh perturbations with $\mu=0.1$ and $0.3$, although a rigorous convergence proof for non‑uniform spatial partitions is not yet available, the numerical results in Table \ref{tab:accuracy} nonetheless indicate that the scheme maintains second‑order accuracy. This observation is accordance with the theoretical proof established for a linearized non-uniform grid mass-conservative BCFD scheme developed in \cite{XF'25} for the Keller–Segel chemotaxis model.
\begin{table}[!t]
\vskip -5pt
	\centering 
    \caption{$L^2$ errors of $u$,$v$, $c$ and $\nabla c$ for the PP-MC-PBCFD scheme for Example \ref{exam:s1}. } \label{tab:accuracy}
	{\footnotesize
        \begin{tabular*}{\columnwidth}{@{\extracolsep\fill}c|ccccccccc@{\extracolsep\fill}}
		\toprule
			$\mu$  & $M$ &  $\|u-u_h\|_{\rm M}$&  Order&  $\|v-v_h\|_{\rm M}$& Order &  $\|c-c_h\|_{\rm M}$& Order  &$\|\nabla c-\bm{d} c_h\|_{\rm TM}$& Order\\
		\midrule
		\multirow{5}*{\makecell{ $0$}}
	&10 & 4.48e-03 & ---    & 6.30e-04 & ---    & 3.95e-03 &--&9.82e-03 & ---\\
	&20 & 1.14e-03 & 1.97  & 1.62e-04 & 1.96  & 9.85e-04 & 2.01 &2.45e-03 & 2.00\\
	&40 & 2.88e-04 & 1.99  & 4.11e-05 & 1.98  & 2.46e-04 & 2.00 &6.12e-04 & 2.00\\
	&80 & 7.24e-05 & 1.99  & 1.04e-05 & 1.99  & 6.15e-05 & 2.00 &1.53e-04 & 2.00\\
		\midrule
		\multirow{5}*{\makecell{ $0.1$}}
	&10 & 4.67e-03 & ---    & 6.52e-04 & ---    & 4.09e-03 &--&1.02e-02 & ---\\
	&20 & 1.18e-03 & 1.98  & 1.65e-04 & 1.98  & 1.02e-03 & 2.00 &2.58e-03 & 1.98\\
	&40 & 3.02e-04 & 1.97  & 4.11e-05 & 2.00  & 2.65e-04 & 1.95 &6.75e-04 & 1.94\\
	&80 & 7.69e-05 & 1.97  & 1.06e-05 & 1.96  & 6.46e-05 & 2.03 &1.65e-04 & 2.03\\
		\midrule
		\multirow{5}*{\makecell{ $0.3$}}
	&10 & 6.66e-03 & ---    & 9.13e-04 & ---    & 5.44e-03 &---&1.52e-02 & ---\\
	&20 & 1.65e-03 & 2.01  & 1.92e-04 & 2.25  & 1.43e-03 & 1.92 &3.96e-03 & 1.94\\
	&40 & 4.38e-04 & 1.92  & 5.36e-05 & 1.84  & 3.50e-04 & 2.03 &9.73e-04 & 2.03\\
	&80 & 1.17e-04 & 1.91  & 1.35e-05 & 1.98  & 8.79e-05 & 1.99 &2.45e-04 & 1.99\\
		\bottomrule
	\end{tabular*}
 }
\end{table}

\begin{example}[Physical property-preserving in 2D]\label{exam:less8pi_1}
As shown in Refs. \cite{HS'24,EVC'12,L'24},  if the total initial mass $M[u^0]+M[v^0]<8\pi$, the global solutions for the two-species Keller–Segel chemotaxis system exist. In this example, we take $\Omega=(0,1) ^2$ and the initial conditions 
	\begin{align*}
		u^0(\bm{x})&=10 \exp \left(-2\big((x-0.5)^2+(y-0.5)^2\big)\right),\\
		v^0(\bm{x})&=2 \exp \left(-\big((x-0.5)^2+(y-0.5)^2\big)\right),\\
		c^0(\bm{x})&= \exp \left(-0.5\big((x-0.5)^2+(y-0.5)^2\big) \right),
	\end{align*}  
such that $M[u^0]+M[v^0] \approx 9.02 < 8\pi$ and the global solutions exist. 
\end{example}

\begin{figure}[!ht]
	\centering
	\includegraphics[width=0.32\linewidth]{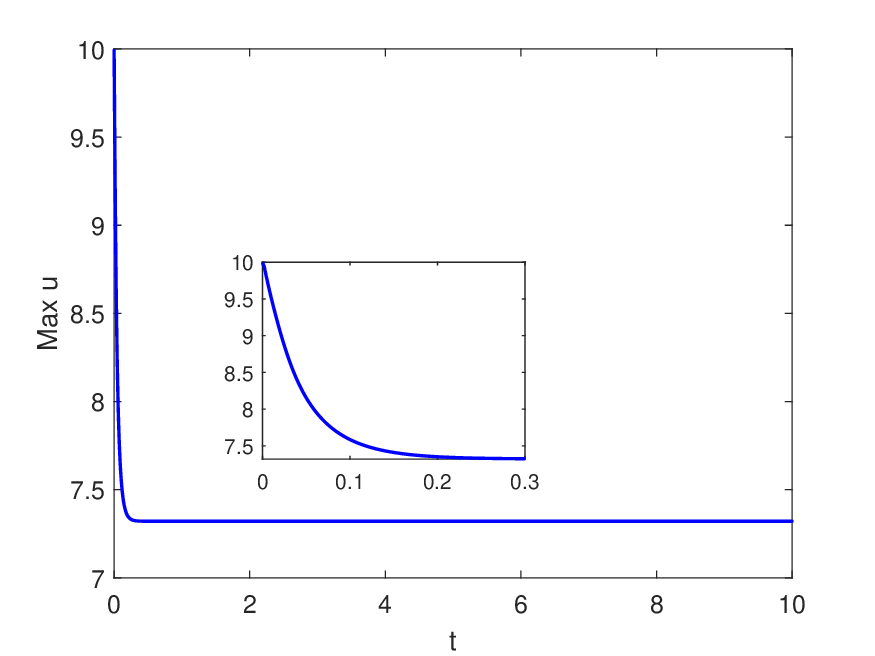}
	\includegraphics[width=0.32\linewidth]{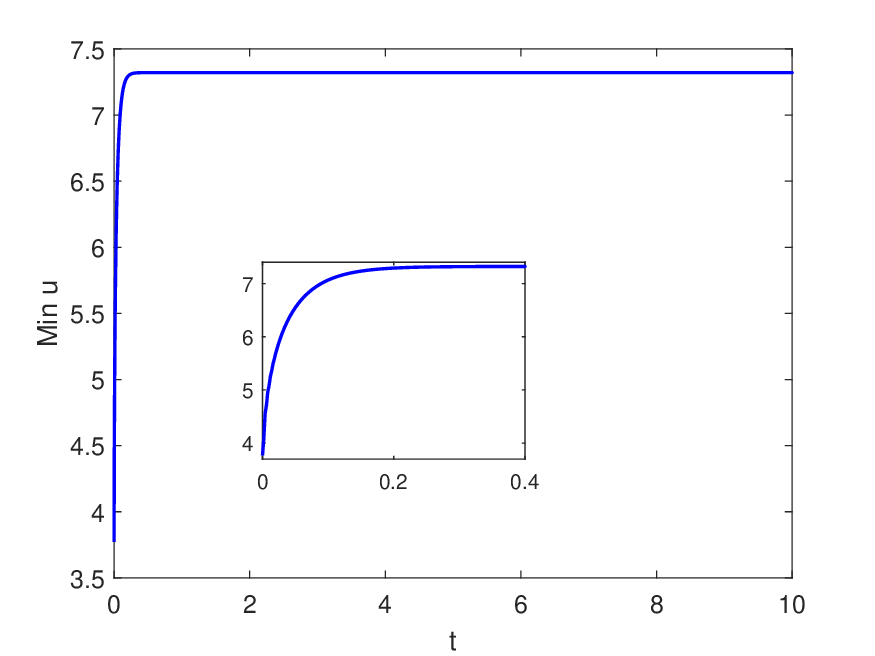}
	\includegraphics[width=0.32\linewidth]{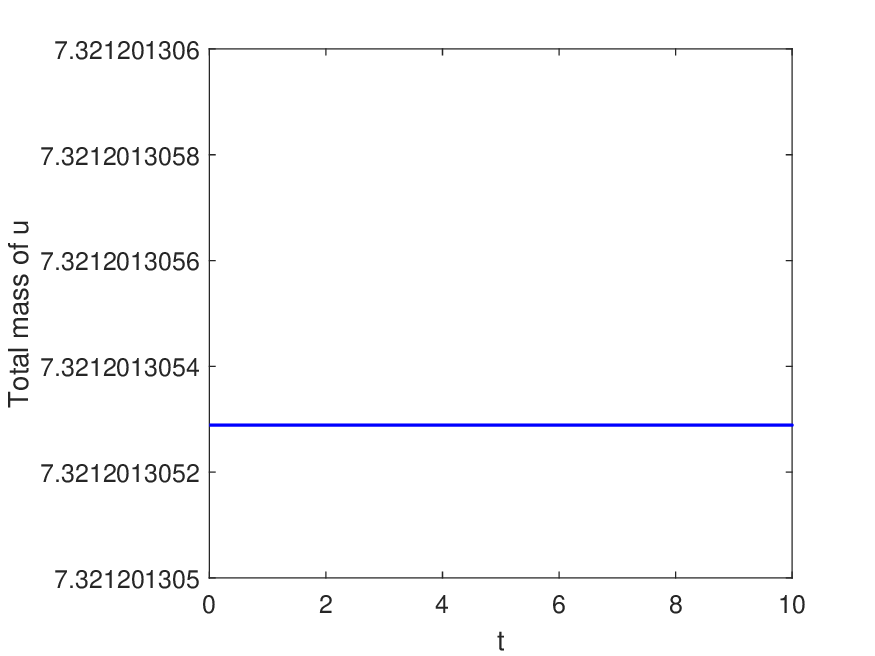}
	\includegraphics[width=0.32\linewidth]{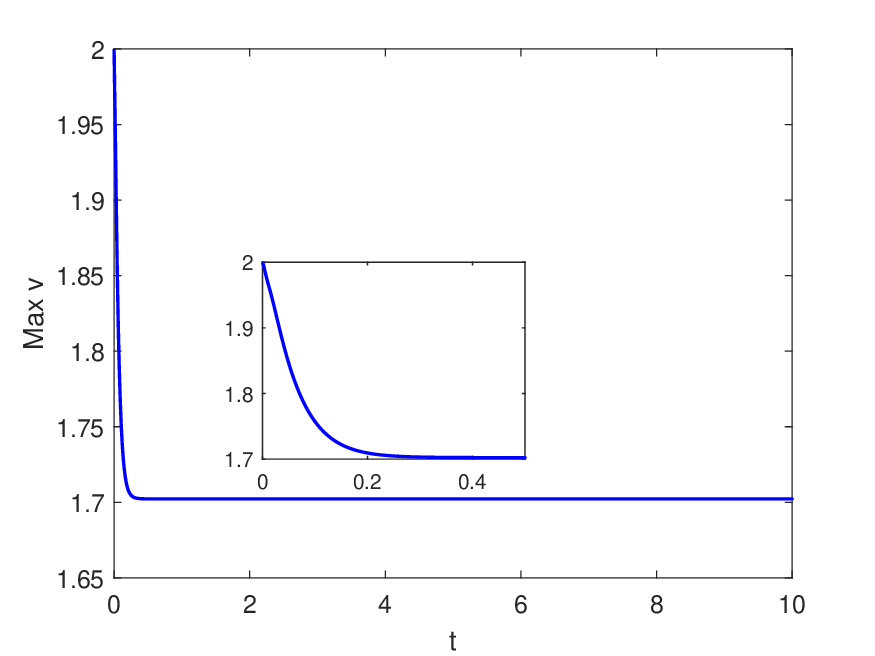}
	\includegraphics[width=0.32\linewidth]{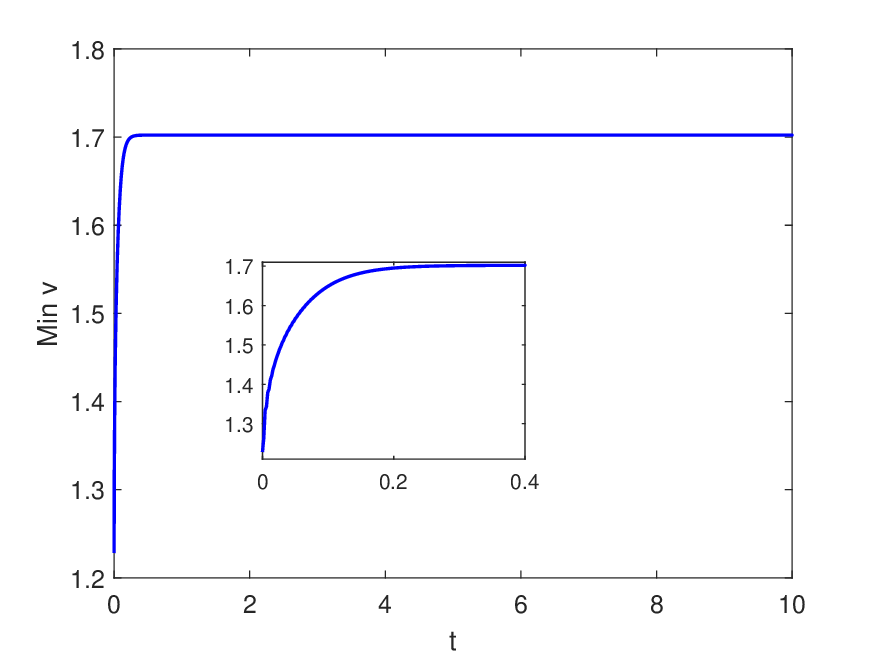}
	\includegraphics[width=0.32\linewidth]{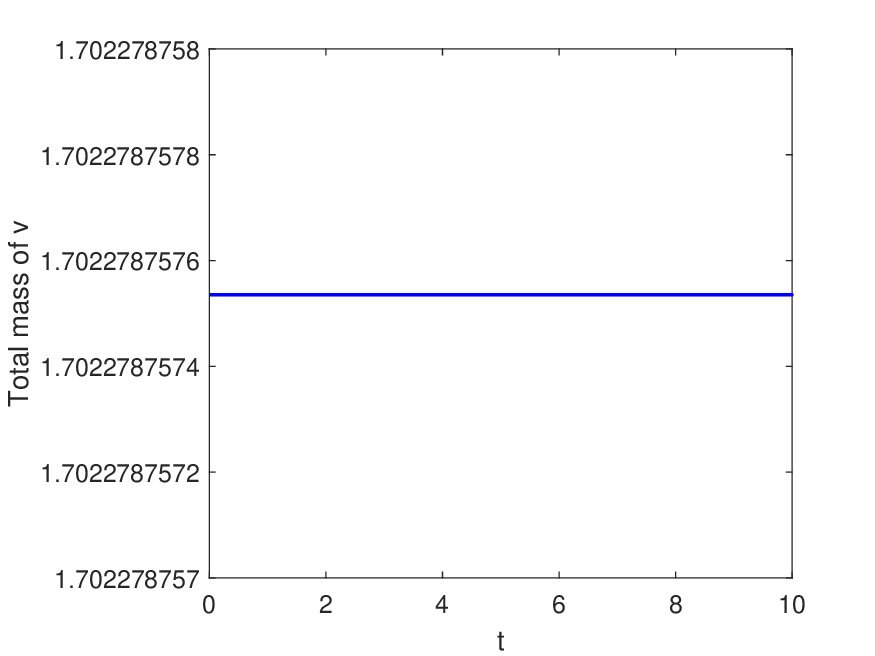}
	\includegraphics[width=0.32\linewidth]{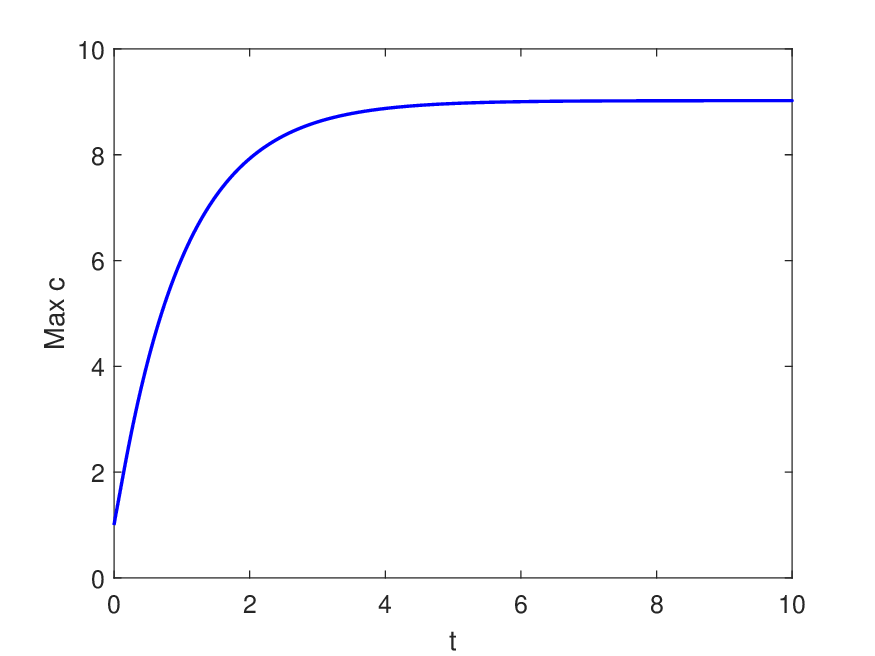}
	\includegraphics[width=0.32\linewidth]{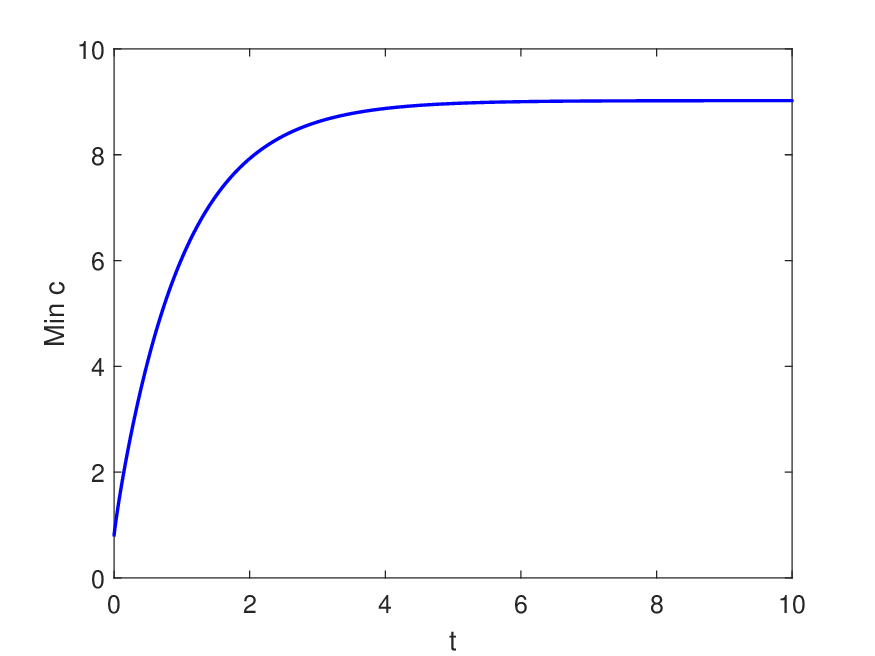}
	\includegraphics[width=0.32\linewidth]{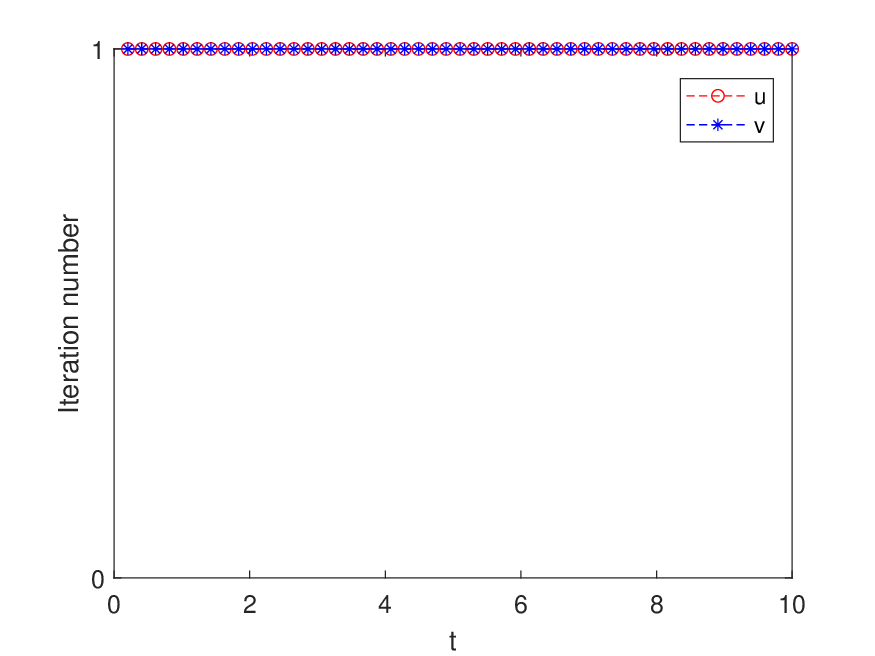}
	\caption{Evolution of the extrema of $u$, $v$, $c$, and the total mass and iteration numbers for $u$ and $v$ ($\mu=0$) for Example \ref{exam:less8pi_1}. }
	\label{fig:eg2:mass}
\end{figure}
\begin{figure}[!t]
	\centering
	\includegraphics[width=0.32\linewidth]{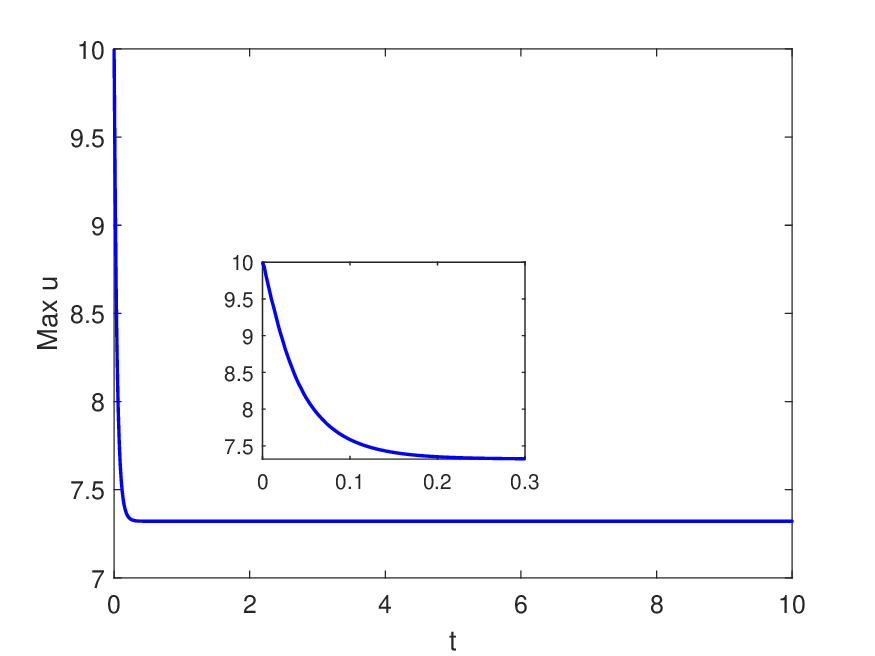}
	\includegraphics[width=0.32\linewidth]{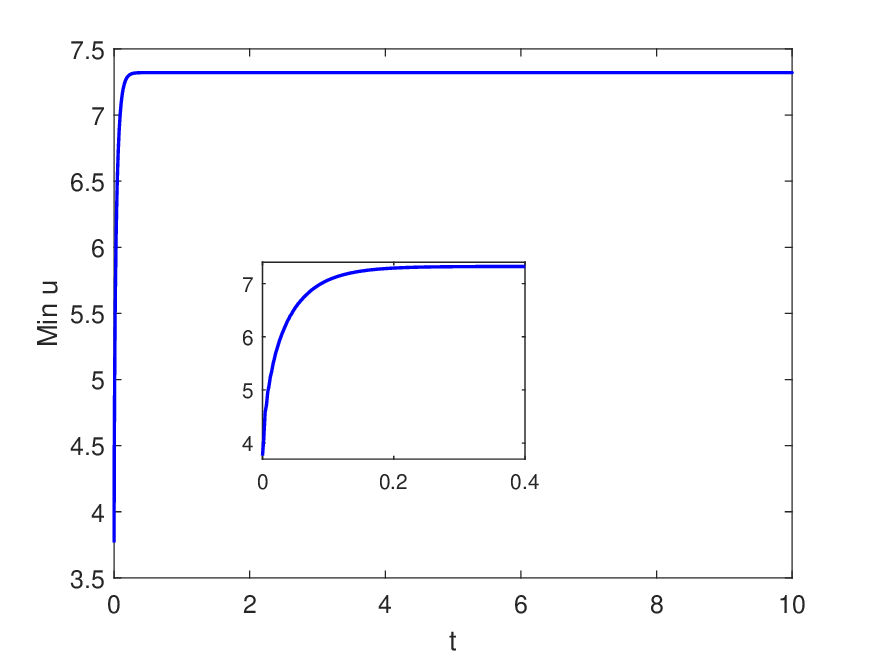}
	\includegraphics[width=0.32\linewidth]{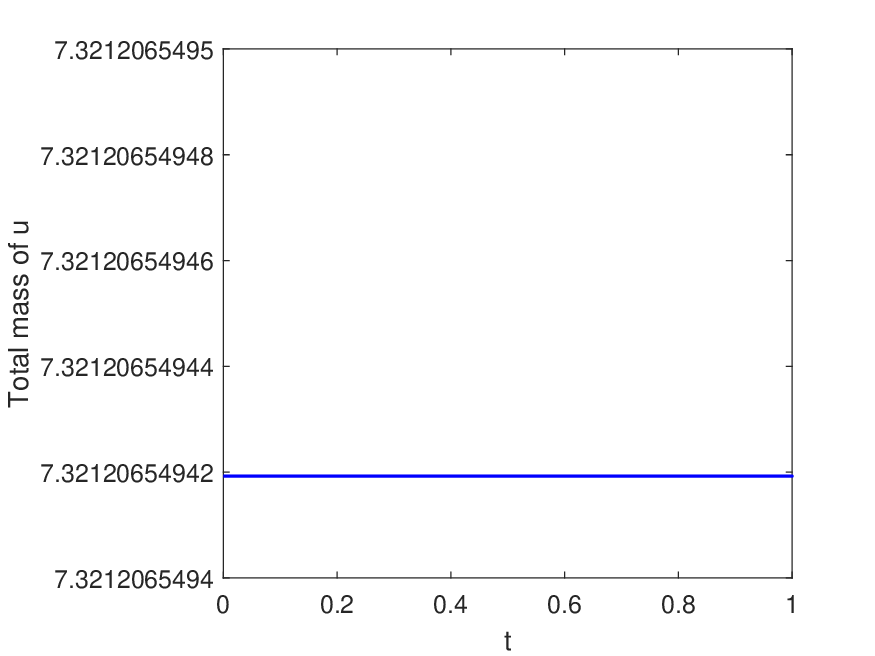}
	\includegraphics[width=0.32\linewidth]{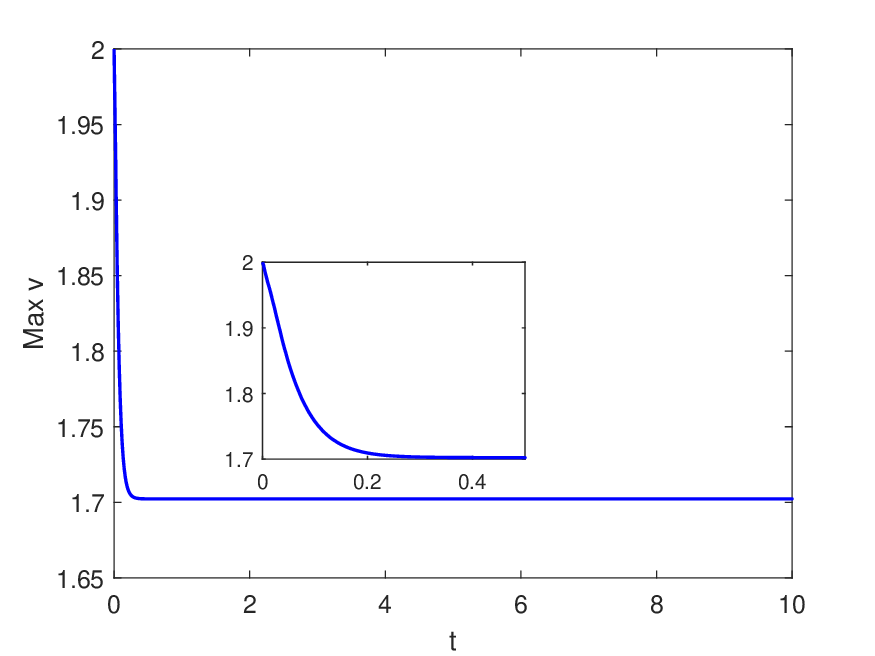}
	\includegraphics[width=0.32\linewidth]{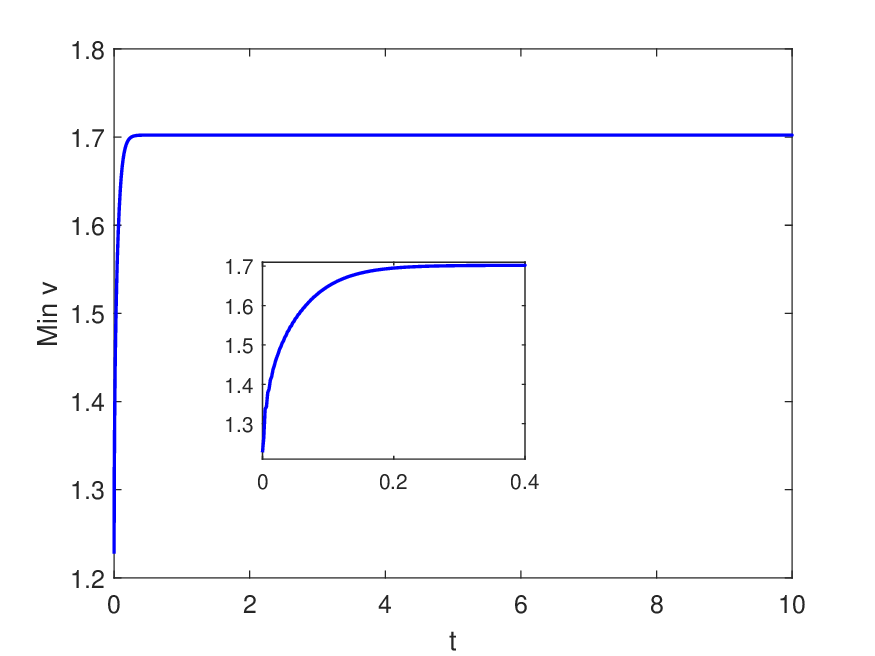}
	\includegraphics[width=0.32\linewidth]{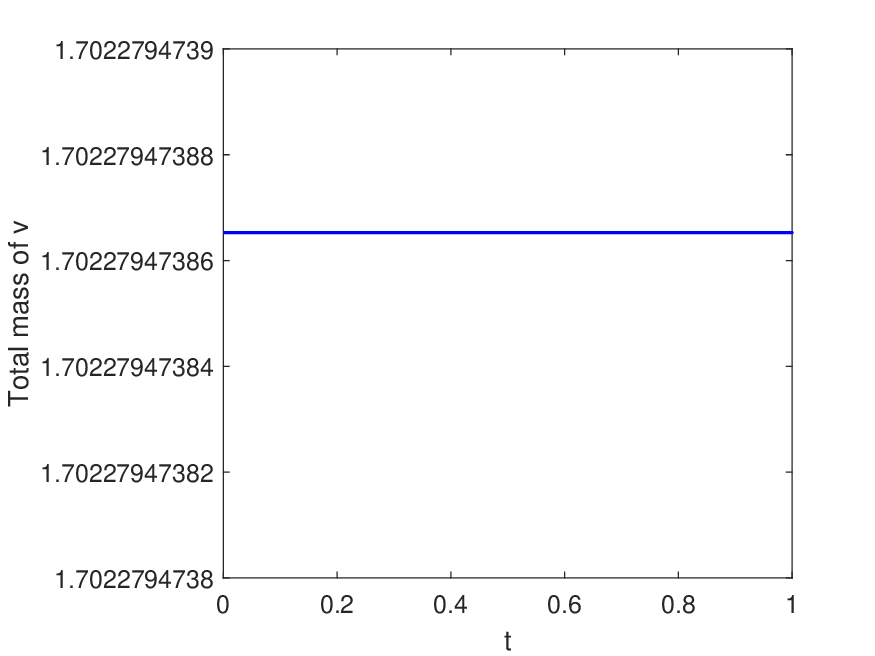}
	\includegraphics[width=0.32\linewidth]{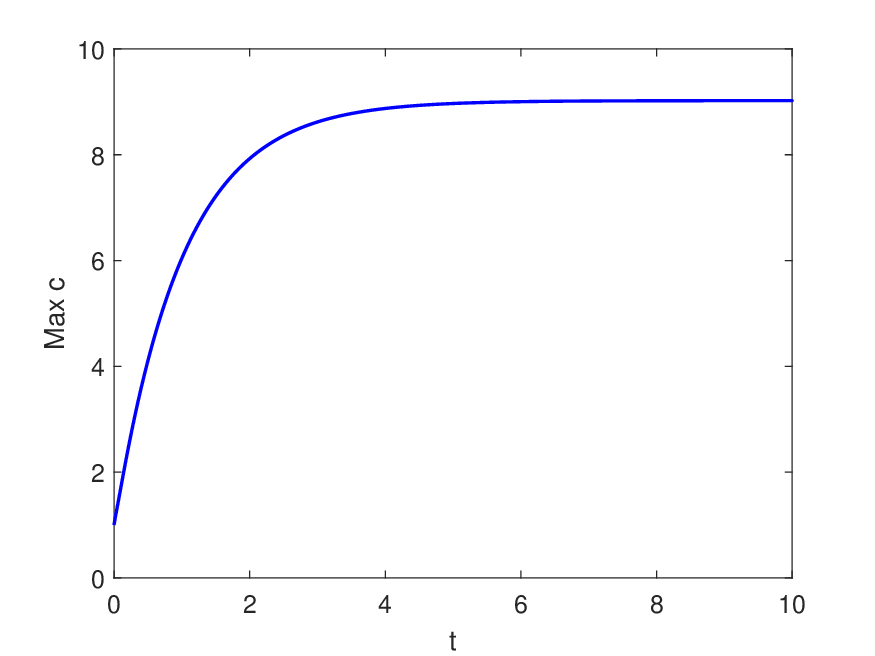}
	\includegraphics[width=0.32\linewidth]{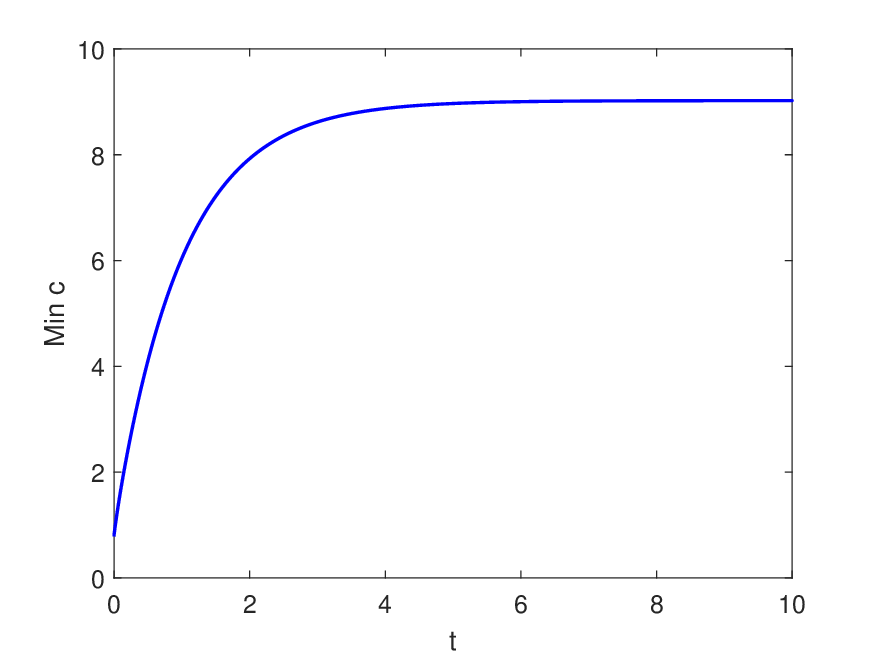}
	\includegraphics[width=0.32\linewidth]{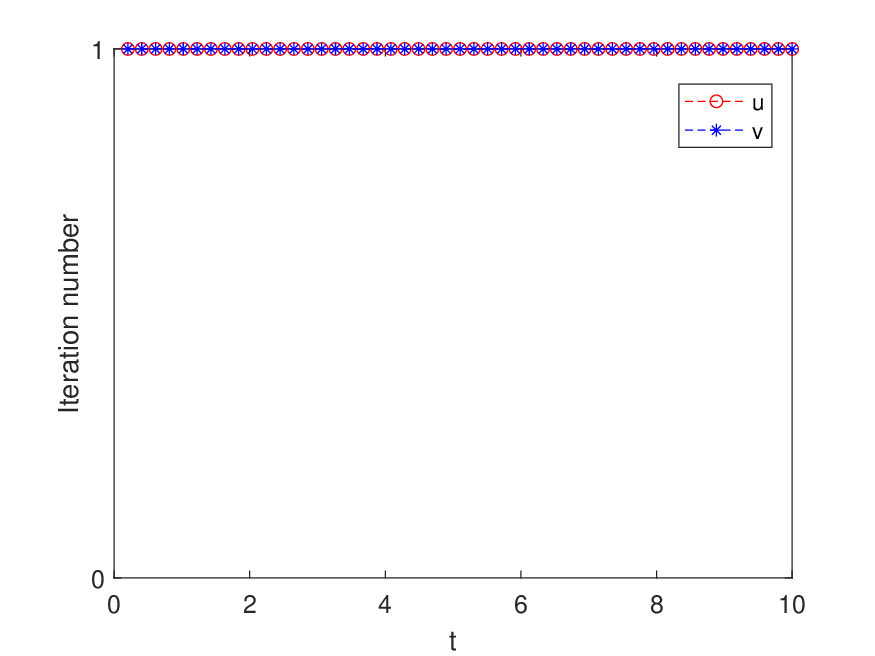}
	\caption{Evolution of the extrema of $u$, $v$, $c$, and the total mass and iteration numbers for $u$ and $v$ ($\mu=0.1$) for Example \ref{exam:less8pi_1}. }
	\label{fig:eg2:mass:non}
\end{figure}
We use this example to verify the preservation of positivity and mass conservation properties of the PP-MC-PBCFD scheme \eqref{model:KS:schemeI0}--\eqref{model:KS:shemeI3} under both uniform and non‑uniform spatio-temporal grids. Meanwhile, we also test the original energy dissipativity law of the proposed scheme, using the discrete version of energy \eqref{def:Energy} defined as 
\begin{equation}\label{def:disEnergy}
\begin{aligned}
  &E_h[u_h^n,v_h^n,c_h^n]\\
  &\quad :=\big(u_h^n\log u_h^n + v_h^n\log v_h^n - u_h^n - v_h^n - u_h^n c_h^n - v_h^n c_h^n + \f{1}{2}(c_h^n)^2,\bm 1\big)_{\rm M} + \f{1}{2}\big( (\bm d c_h^n)^2,\bm 1\big)_{\rm TM}.
\end{aligned}
\end{equation}
In the following simulation, the computational domain is discretized using $M=80$ grid points in both the $x$- and $y$-directions, and the time stepsize in \eqref{eq:adap} is set to $\Delta t=2.0\times10^{-3}$. The simulation results on both uniform grids (i.e., $\mu=0$) and non-uniform grids (i.e., $\mu=0.1$) are summarized in Figs. \ref{fig:eg2:mass}--\ref{fig:eg2:mass:non}, which show the time evolution of the maximum and minimum values of $u$, $v$ and $c$, along with the total mass of $u$ and $v$, and the iteration numbers of semismooth Newton solver for computing the $L^2$ projection \eqref{model:KS:shemeI2}. We have the following observations: (i) the maximum values of both $u$ and $v$ initially decrease while the minimum values initially increase; thereafter, both quantities gradually approach their steady states; (ii) the densities $u$ and $v$ remain positivity, and their masses are always conserved to at least 10 significant figures; (iii) the positivity of the concentration is also preserved throughout the simulation; (iv) the semi-smooth Newton method for solving the mass-conservative multiplier $\xi$ and $\theta$ in \eqref{eq:Lagrange:nonlinear} converges in just one iteration per time step, demonstrating the  efficiency of the proposed $L^2$ projection process; and (v) as shown in Fig. \ref{fig:eg2:energy}, the original discrete energy \eqref{def:disEnergy} of the PP-MC-PBCFD scheme exhibits favorable dissipative behavior on both uniform (i.e., $\mu=0$) and non-uniform (i.e., $\mu=0.1$) grids, although a rigorous proof of such dissipative property is not established herein. To the best of our knowledge, comparable second-order \textit{linear} schemes
for the multi-species Keller--Segel model, together with rigorous positivity-preserving and energy-dissipation analysis, remain scarce in the literature. 
\begin{figure}[!t]
	\centering
	\includegraphics[width=0.45\linewidth]{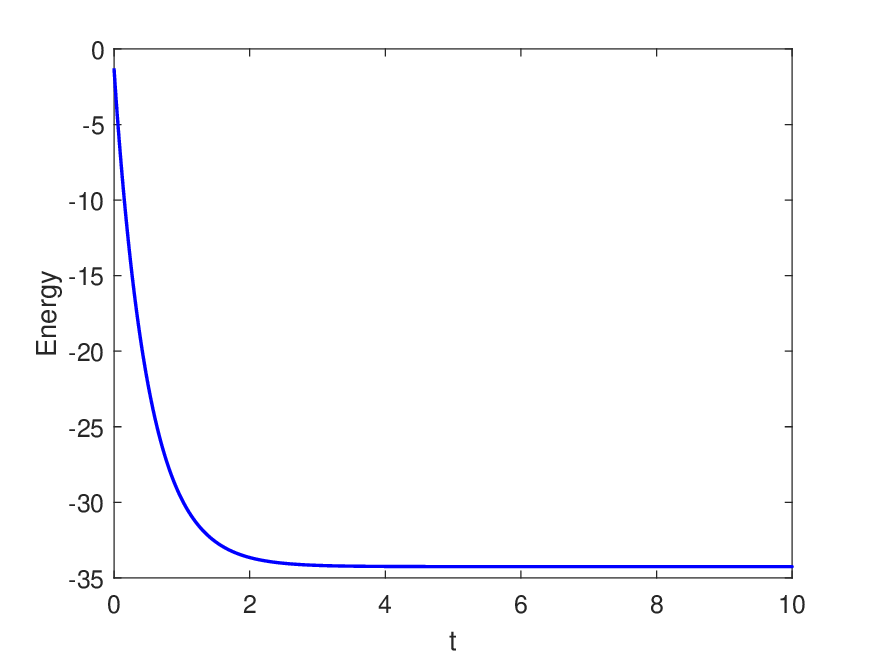}
    \includegraphics[width=0.45\linewidth]{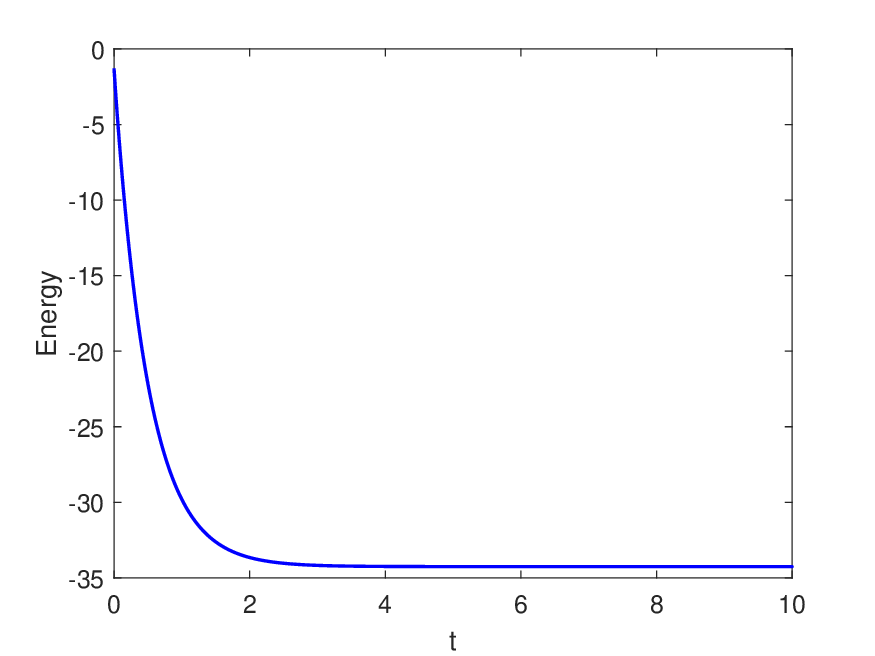}
	\caption{Evolution of the energy $\mu=0$ (left) and $\mu=0.1$ (right) for Example \ref{exam:less8pi_1}.}	\label{fig:eg2:energy}
    \vspace{-5pt}
\end{figure}

\begin{example}[Blow up of 2D two-species model]\label{exam:blow_up} In this example, we modify the initial conditions of Example \ref{exam:less8pi_1} to
\begin{equation*}
	\begin{aligned}
		u^0(\bm{x})&=1000 \exp\left( -100 ((x-0.5)^2+(y-0.5)^2 )\right) , \\
		v^0(\bm{x})&=500 \exp\left( -100 ((x-0.5)^2+(y-0.5)^2)\right) , \\
		c^0(\bm{x})&=0.
    \end{aligned}
\end{equation*}
Under this scenario, the solutions of the 2D two-species Keller–Segel chemotaxis model \eqref{model:KS} are expected to blow up in a finite time, as the total initial mass satisfies $M[u^0]+M[v^0]\approx 47.07>8\pi$. Nevertheless, as long as the solution exists prior to the blow-up time, both positivity and mass conservation properties shall be preserved.  
\end{example}

To better simulate the blow-up phenomenon, we adopt the following time-adaptive strategy:
\begin{equation}\label{tau_adap}
	\tau_n := \min\bigg\{\max \bigg\{\frac{\tau_{max}}{\sqrt{1+\zeta\max\{(\|D_{\tau}u_h^{n-1}\|_{\infty},\|D_{\tau}v_h^{n-1}\|_{\infty})\}}},\tau_{min}\bigg\},\varsigma\tau_{n-1}\bigg\},
\end{equation}
where $\tau_{max}$ and $\tau_{min}$ denote the maximum and minimum  allowable time stepsizes,  respectively, and $\zeta$ and $\varsigma$ are two positive tunable parameters used to adjust the stepsize of the next time level. In addition, for this simulation, we adopt the following  specially designed non-uniform spatial grids, which concentrate significantly more grid nodes near the blow-up point $(0.5,0.5)$ \cite{XF'25}:
\begin{equation}\label{grid:mid}
\begin{cases}
	x_{N_{x}/2+i+1/2}=\f{1}{2} + \frac{i^2}{2\left(N_x / 2+1\right)^2}, & i=0,1, \ldots, N_x / 2+1, \\ 
    x_{N_{x}/2-i+1/2}=\f{1}{2} -\frac{i^2}{2\left(N_x / 2+1\right)^2}, & i=1,2, \ldots, N_x / 2+1,
\end{cases}
\end{equation}
with the $y_{j+1/2}~(j=0,\ldots, N_y)$ grid points defined in an analogous manner.

\begin{figure}[!ht]
	\centering
	\includegraphics[width=0.32\linewidth]{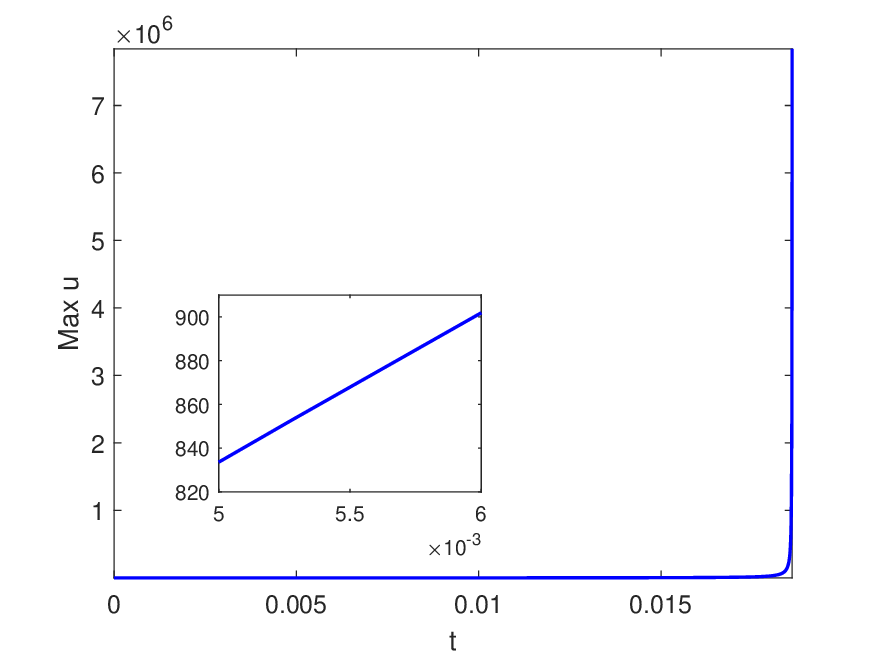}
	\includegraphics[width=0.32\linewidth]{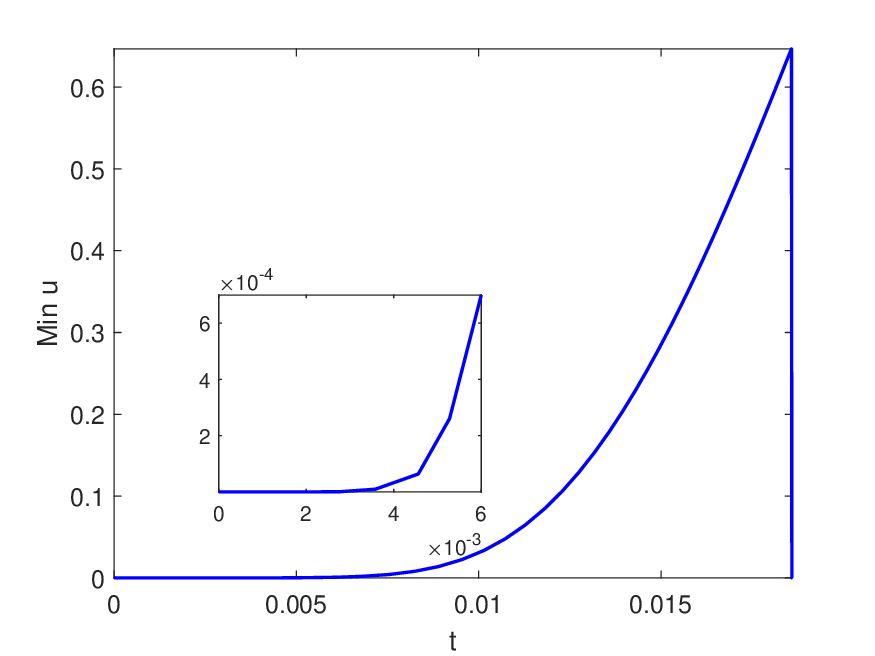}
	\includegraphics[width=0.32\linewidth]{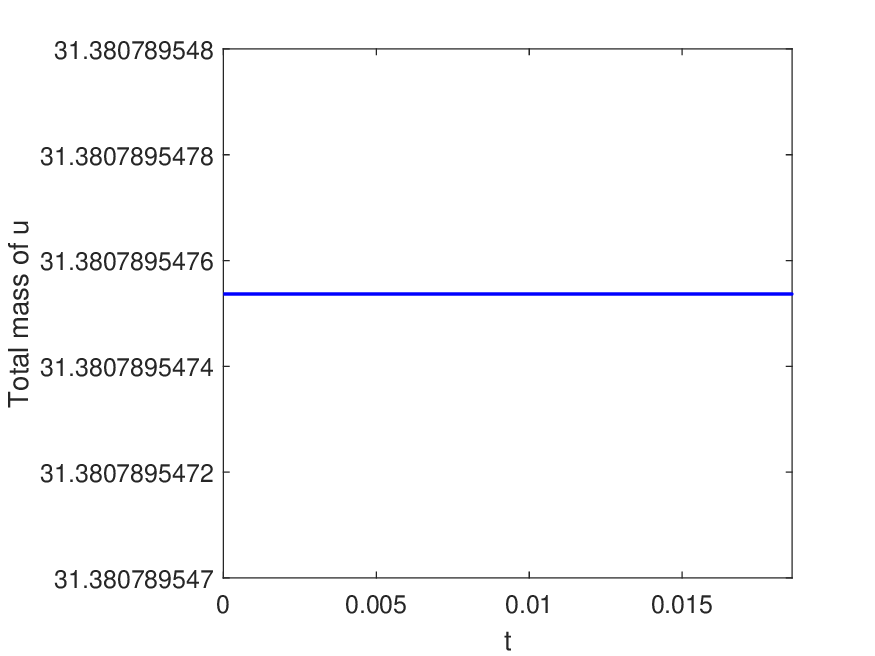}
	\includegraphics[width=0.32\linewidth]{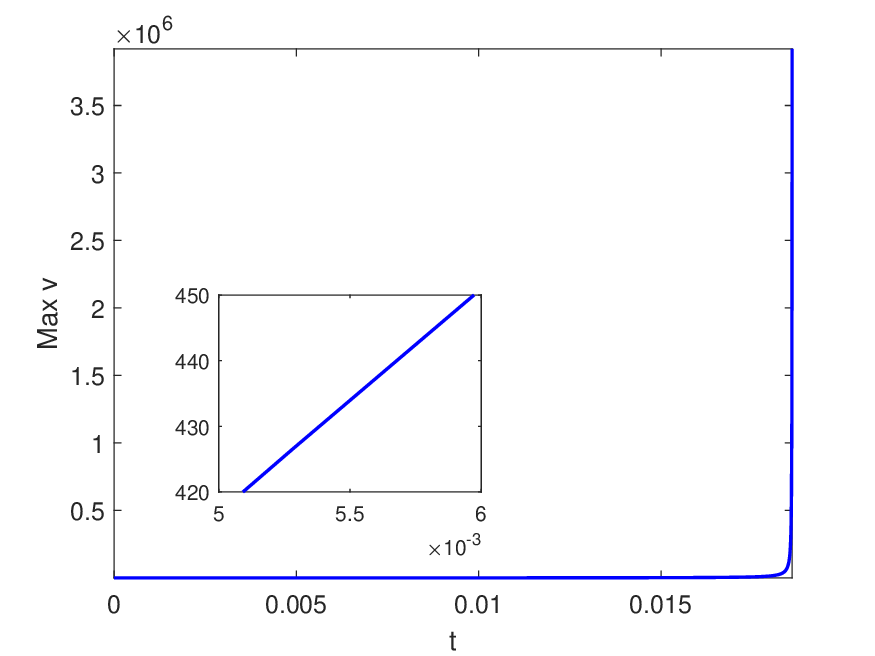}
	\includegraphics[width=0.32\linewidth]{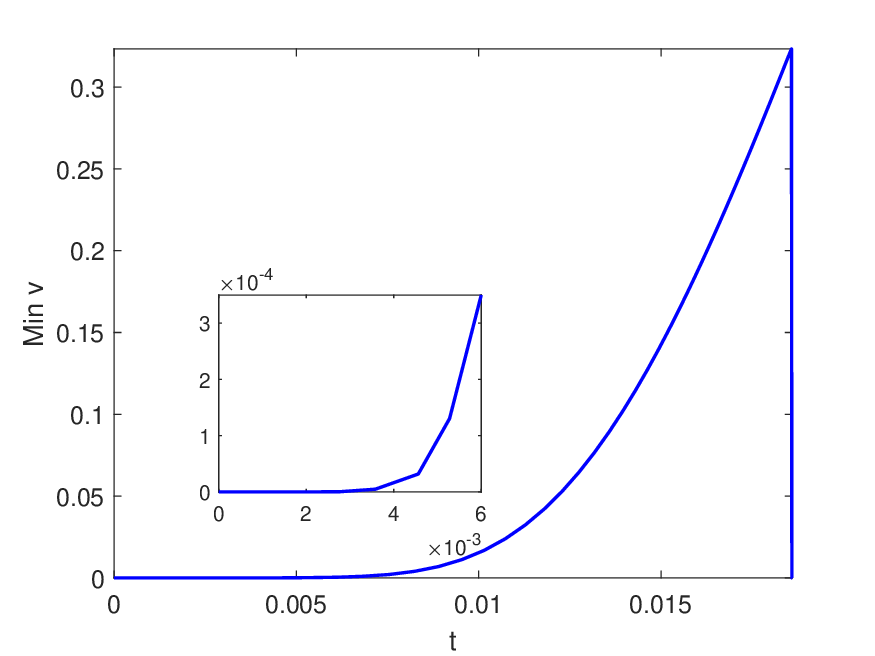}
	\includegraphics[width=0.32\linewidth]{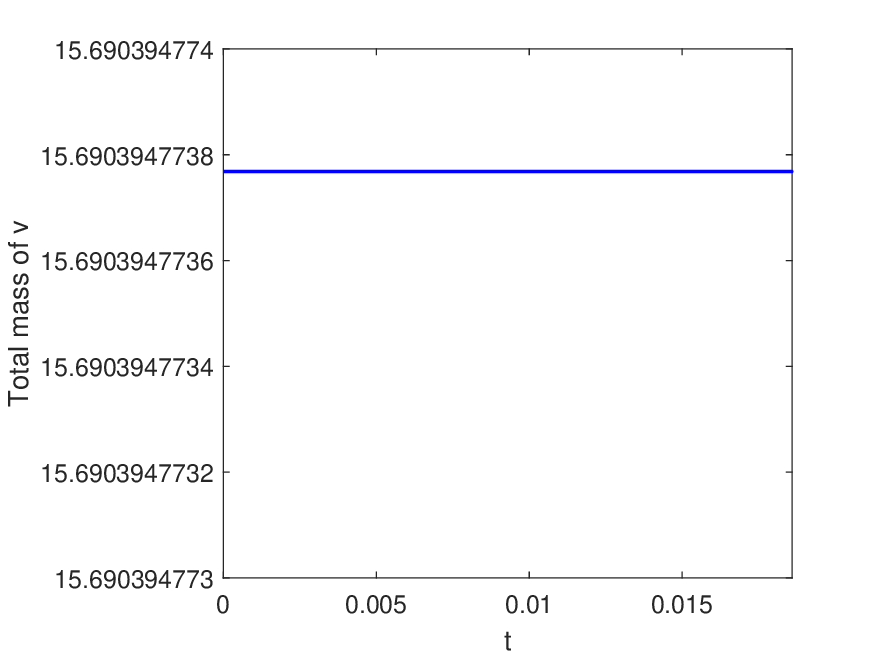}
	\includegraphics[width=0.32\linewidth]{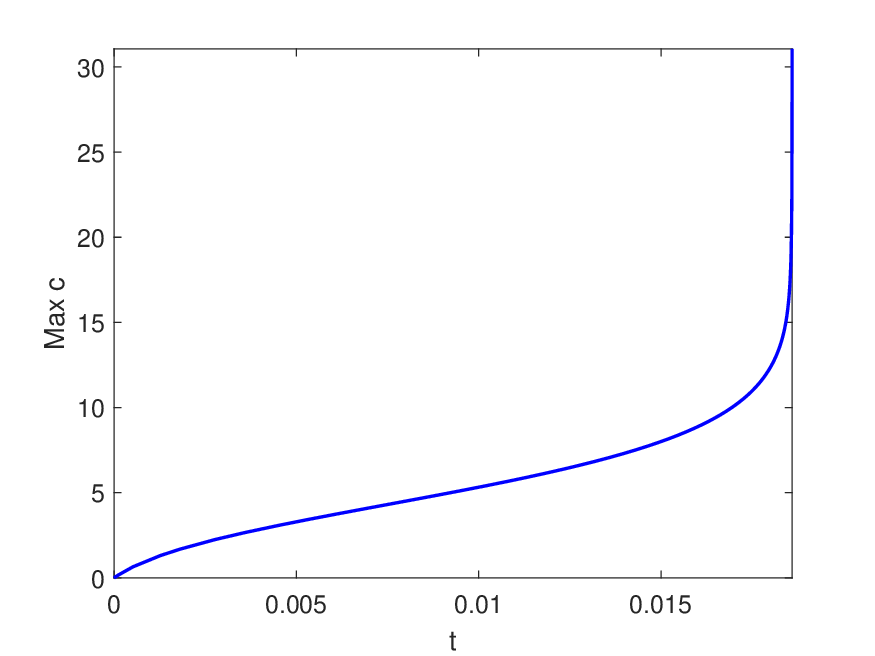}
	\includegraphics[width=0.32\linewidth]{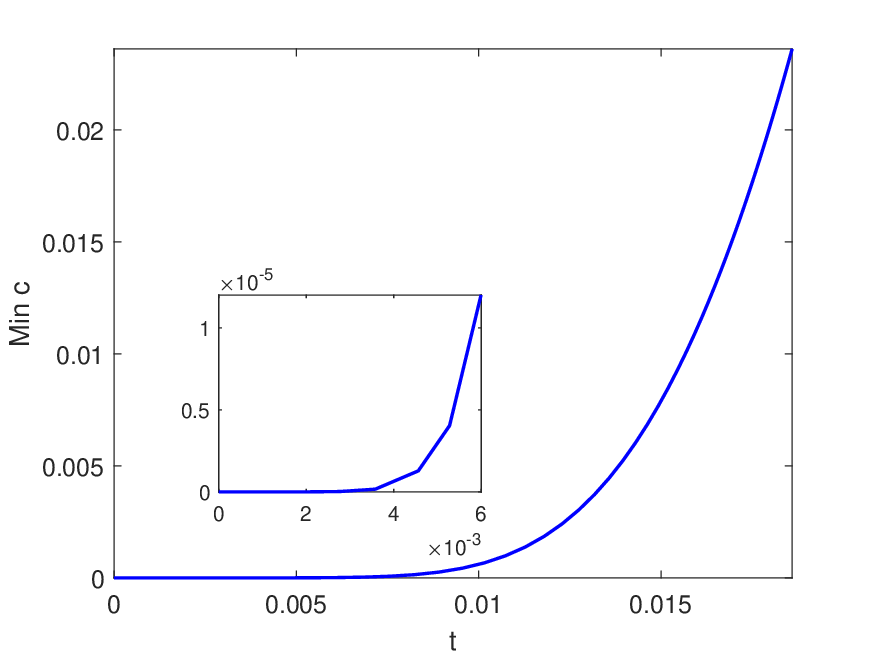}
	\includegraphics[width=0.32\linewidth]{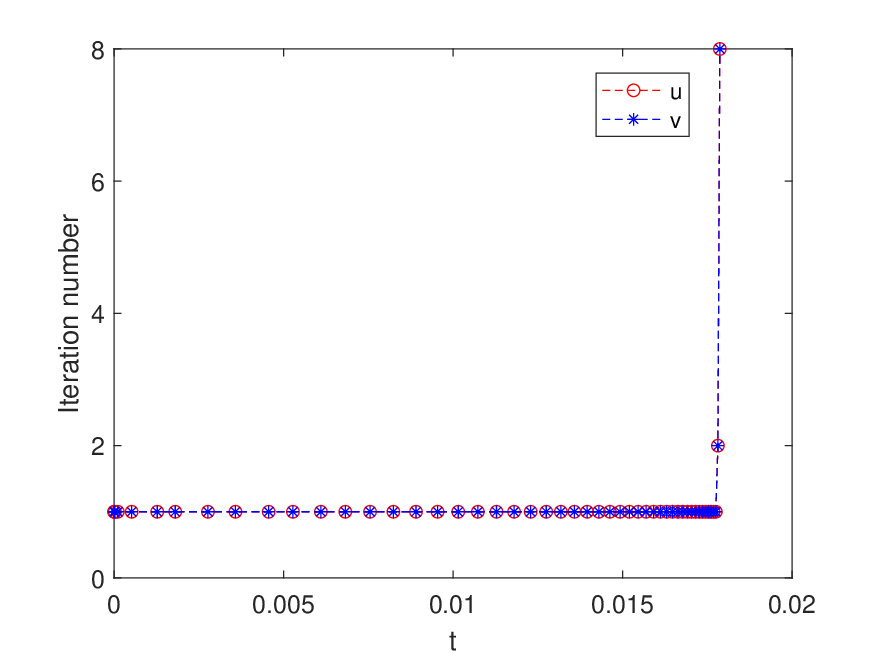}
	\caption{Evolution of the extrema of $u$, $v$, $c$, and the total mass and iteration numbers for $u$ and $v$ for Example \ref{exam:blow_up}.}
	\label{fig:eg3:mass}
\end{figure}
\begin{figure}[!ht]
	\centering
    \subfigure[maximum of $u$ and $v$]
    {
    \includegraphics[width=0.313\linewidth]{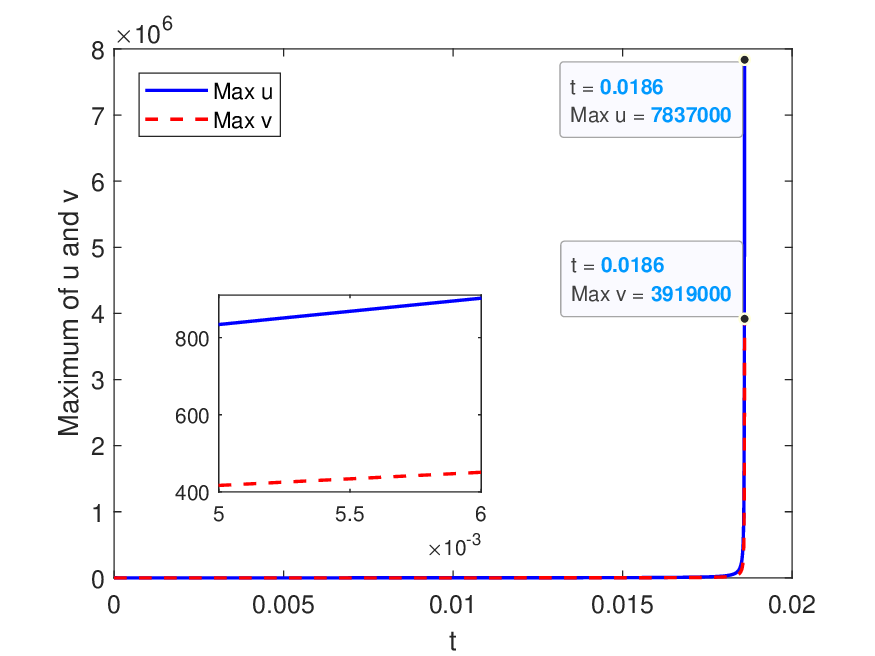}
    \label{fig:max_uv}
    }
    \subfigure[time stepsizes]
    {
    \includegraphics[width=0.313\linewidth]{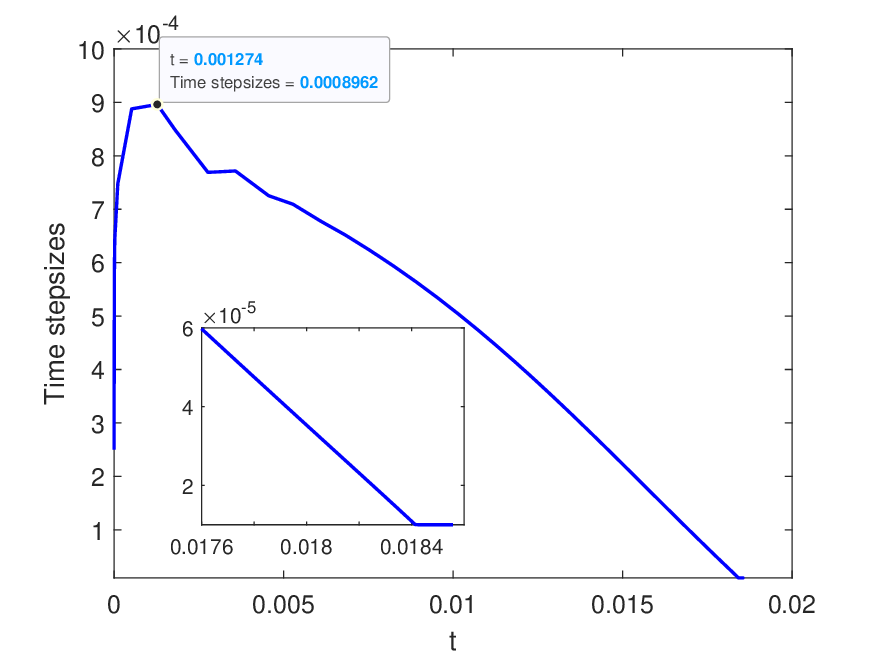}
    \label{fig:tau}
    }
    \subfigure[energy]
    {
    \includegraphics[width=0.313\linewidth]{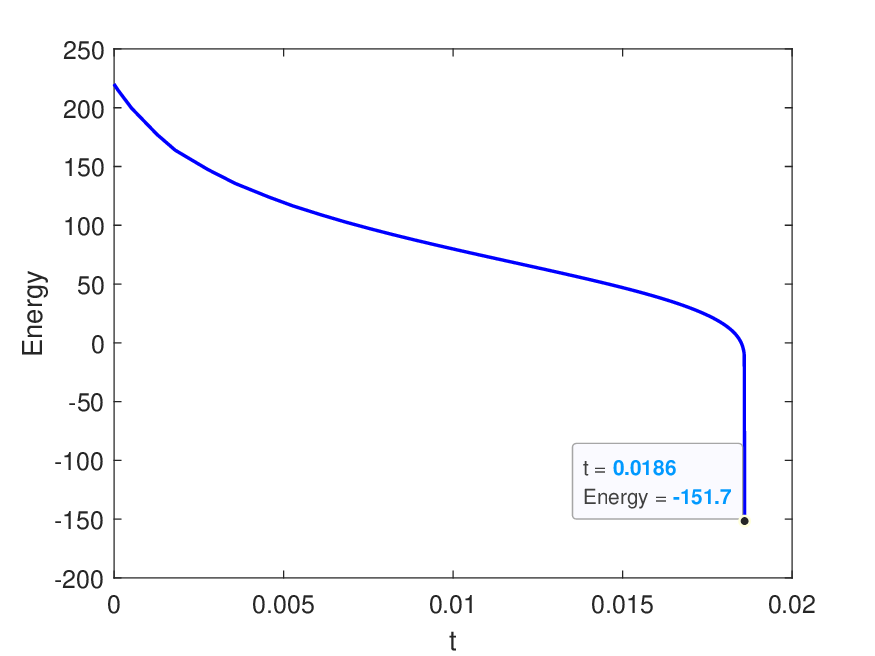}
    \label{fig:energy_blowup}
    }
	\caption{Evolution of the maximum values of $u$, $v$, the time stepsizes, and energy for Example \ref{exam:blow_up}.}
\end{figure}
In this test, we set $M=80$ grid points in each spatial direction and adopt the grid partitions \eqref{grid:mid} along both the $x$- and $y$-directions. The adaptive time-stepping strategy \eqref{tau_adap} is employed with parameters $\tau_{max}=1.0\times10^{-3}$, $\tau_{min}=1.0\times10^{-5}$, $\zeta = 1.0\times10^{-5}$ and $\varsigma=5$. The simulation results including the time evolution of the extrema of $u$, $v$ and $c$, the total mass of $u$ and $v$, and the iteration number of the semismooth Newton solver used in the $L^2$ projection \eqref{model:KS:shemeI2}, as well as the time evolution of the discrete energy are displayed in Figs. \ref{fig:eg3:mass} and \ref{fig:energy_blowup}, respectively. These results lead to conclusions similar to those drawn in Example \ref{exam:less8pi_1}. In addition, the evolution of the adaptive time stepsize is also illustrated in Fig. \ref{fig:tau}. As can be observed, the time stepsize initially increases and then gradually decreases until it reaches the minimum value. In fact, it is precisely at this moment that the blow-up phenomenon occurs (see Fig. \ref{fig:max_uv}), and the minimum time stepsize $\tau_{min}=1.0\times10^{-5}$ is thus required to adequately  capture such blow-up phenomenon. This behavior is fully consistent with the expected physical scenario. Furthermore, the time-adaptive strategy \eqref{tau_adap} simulates the blow-up phenomenon in just $86$ time steps. In contrast, a uniformly small stepsize $\tau_{min} = 1.0\times10^{-5}$ requires $1,860$ time steps. As a result, the adaptive approach is approximately 22 times more efficient.  Finally, we also present in Fig. \ref{fig:blow_up} the time evolution of the cell densities $u$, $v$, and the chemoattractant concentration $c$. The results clearly show that both $u$ and $v$ undergo blow-up at $t = 1.86\times10^{-2}$, which is in full agreement with the underlying physical mechanism.
\begin{figure}[!htbp]
	\centering
	\includegraphics[width=0.24\linewidth]{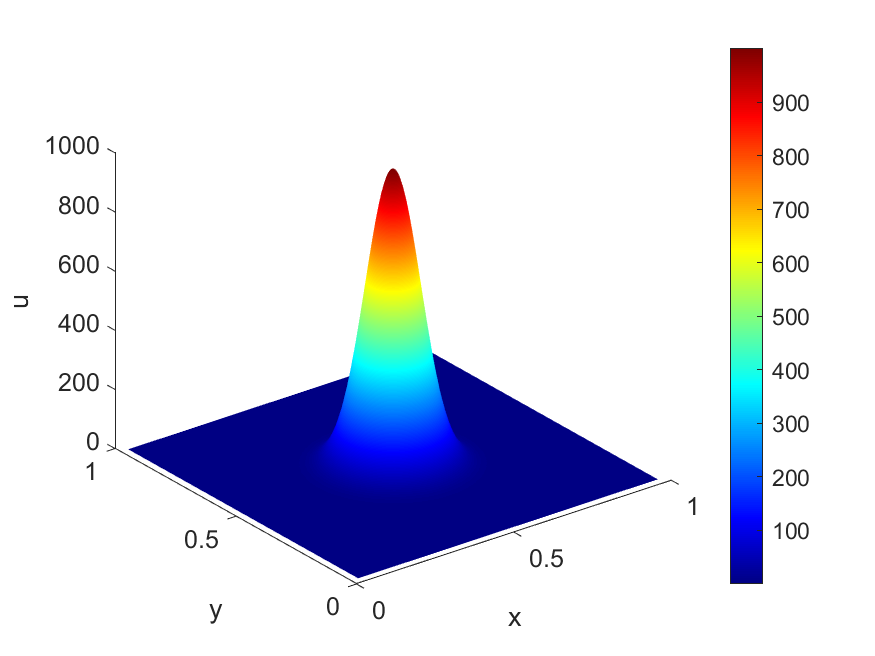}
	\includegraphics[width=0.24\linewidth]{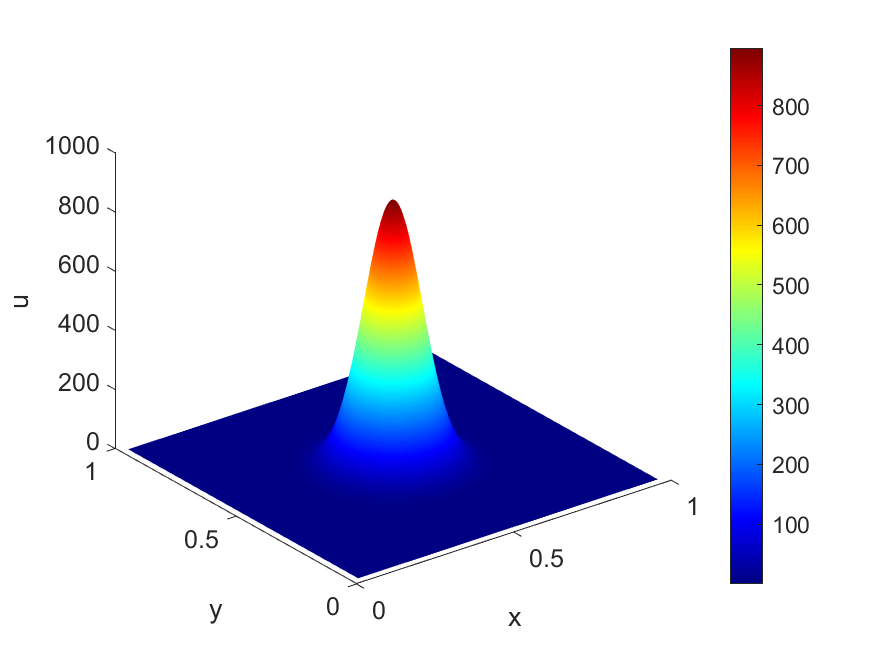}
	\includegraphics[width=0.24\linewidth]{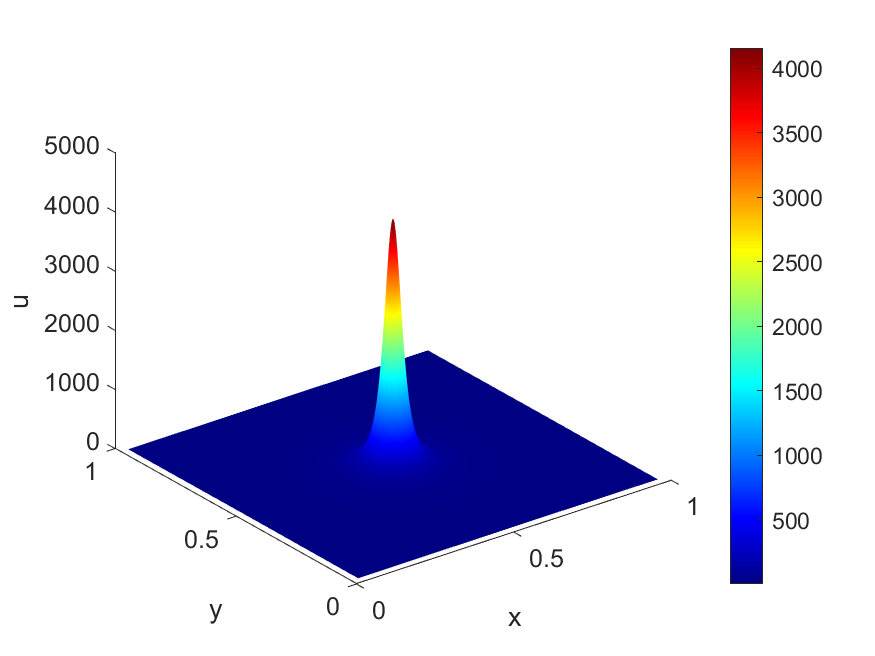}
	\includegraphics[width=0.24\linewidth]{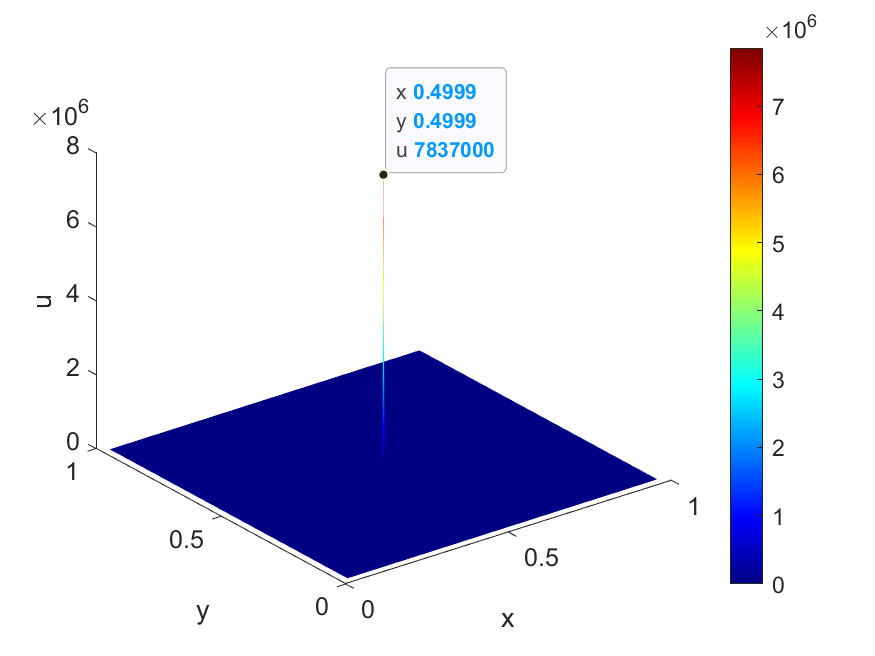}
	\includegraphics[width=0.24\linewidth]{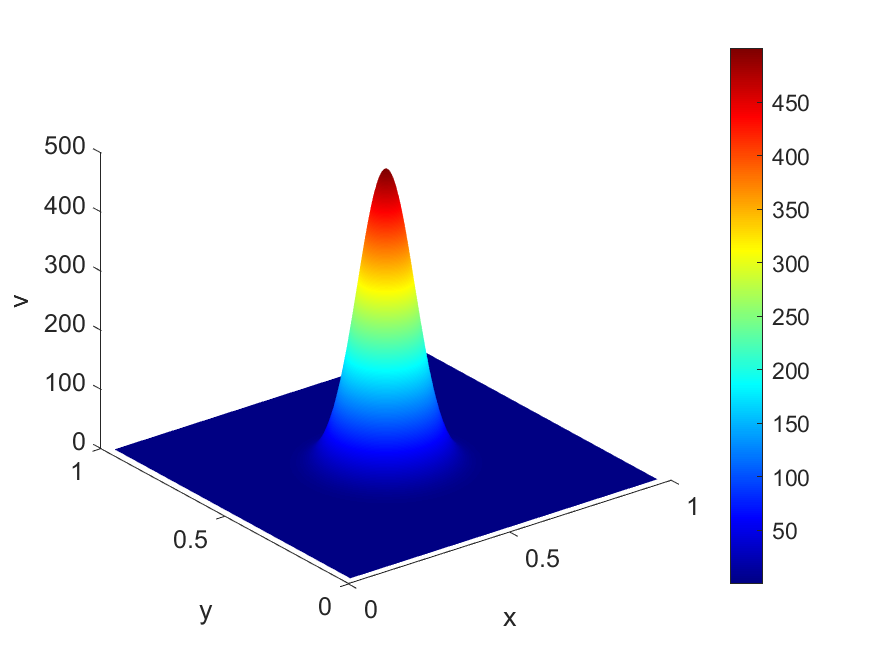}
	\includegraphics[width=0.24\linewidth]{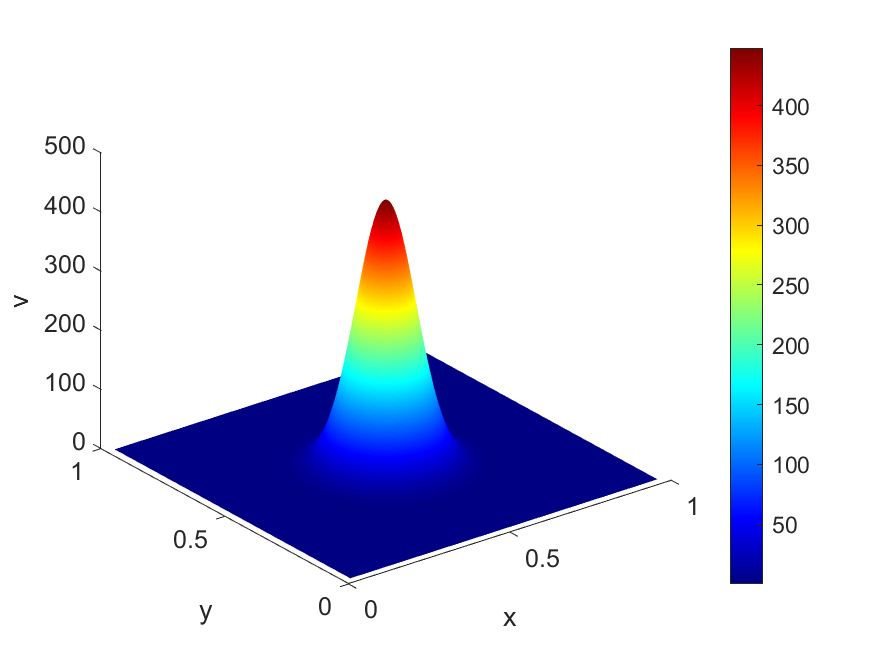}
	\includegraphics[width=0.24\linewidth]{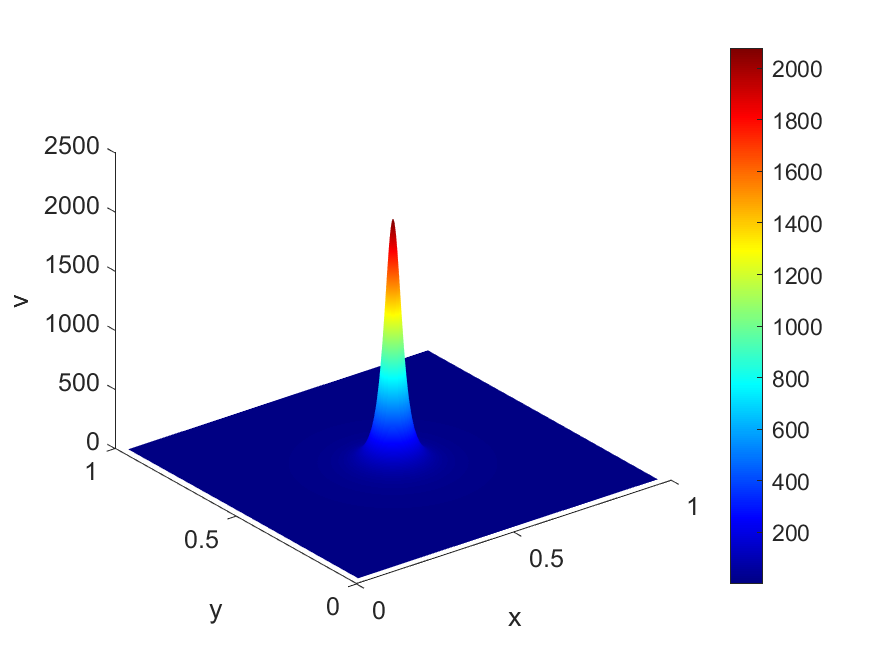}
	\includegraphics[width=0.24\linewidth]{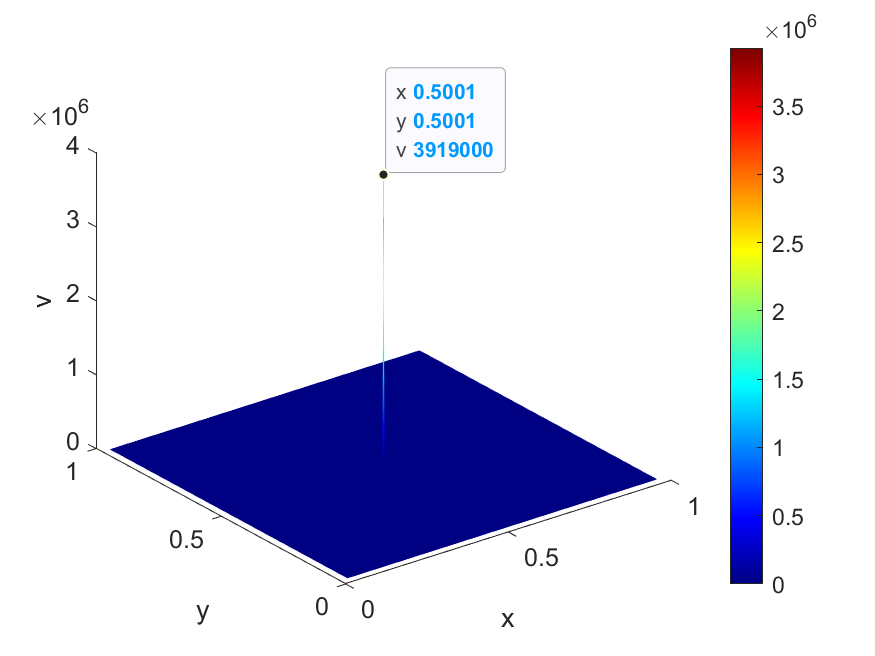}
	\includegraphics[width=0.24\linewidth]{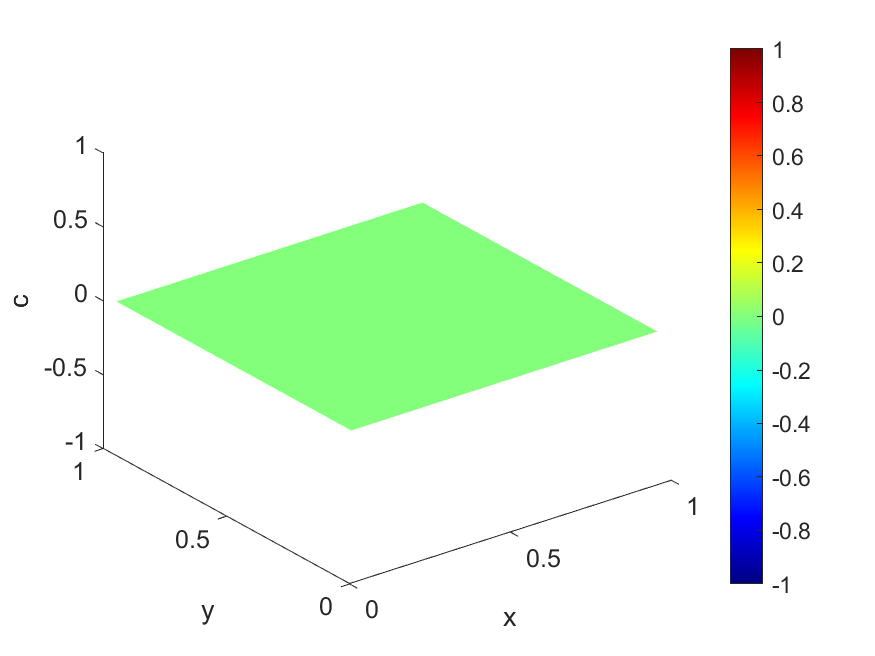}
	\includegraphics[width=0.24\linewidth]{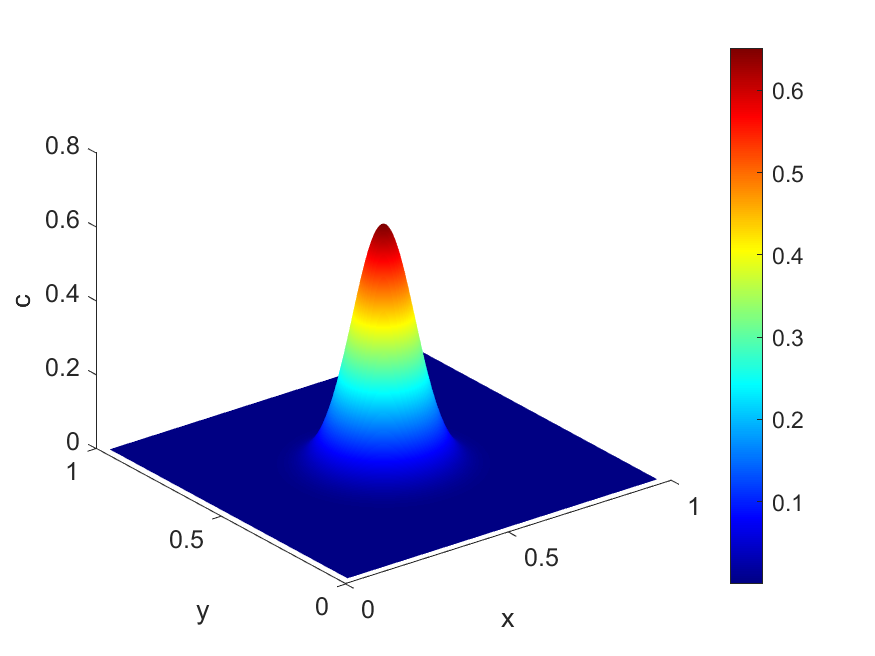}
	\includegraphics[width=0.24\linewidth]{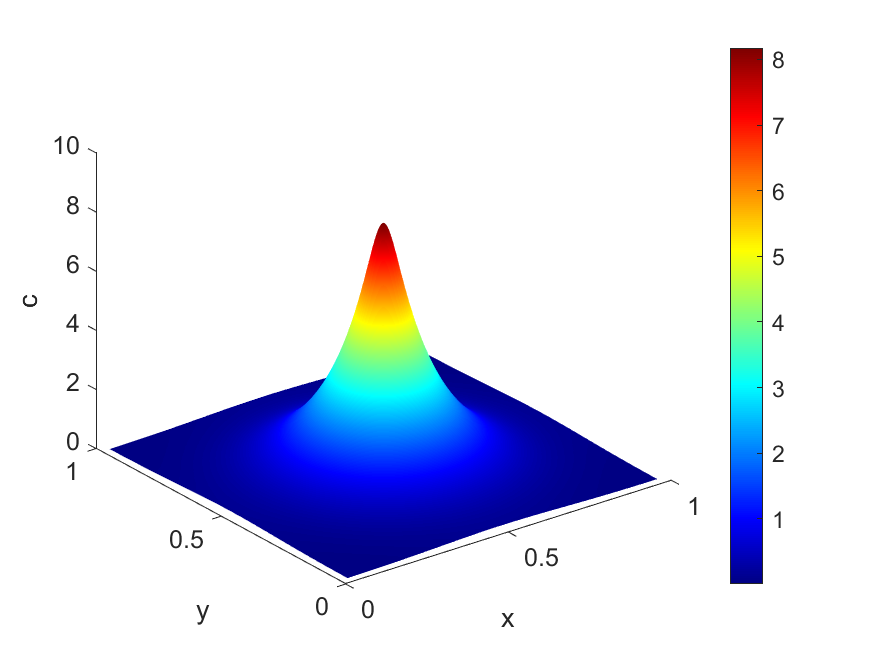}
	\includegraphics[width=0.24\linewidth]{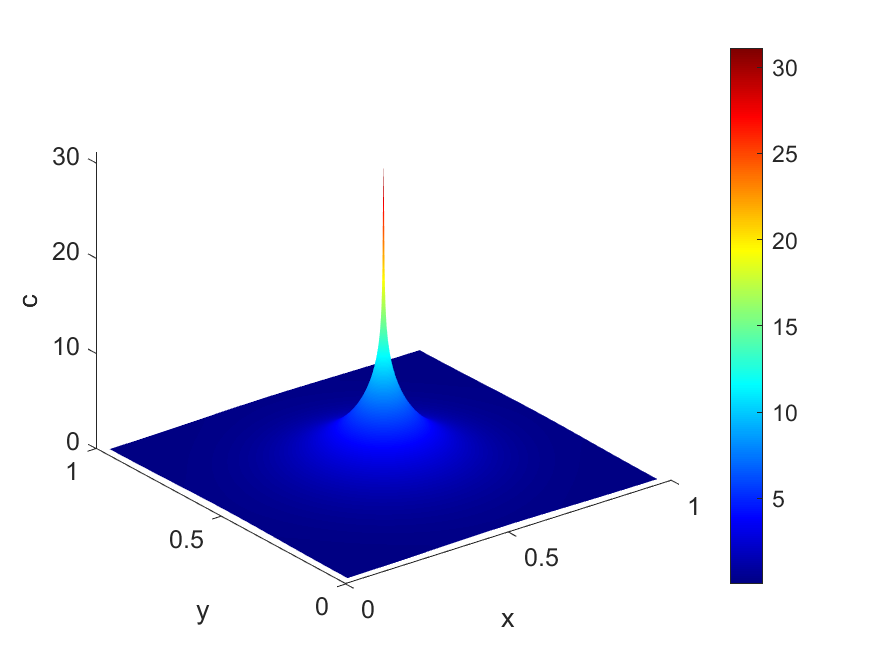}
	\caption{Contour plots of $u$, $v$, $c$ (from top to bottom) at time instants $t = 0,\ 1.27\times10^{-3},\ 1.57\times10^{-2},\ 1.86\times10^{-2}$ (from left to right)  for Example \ref{exam:blow_up}.}
	\label{fig:blow_up}
\end{figure}

\begin{example}[3D simulation]\label{exam:bu:3D} In the last example, we consider the 3D two-species Keller--Segel chemotaxis model \eqref{model:KS} in a cubic domain $\Omega=(0,1) ^3$. The initial conditions are prescribed as follows:
		\begin{equation*}
			\begin{aligned}
				 u^0(\bm{x})&=4 \exp\left( -100 ((x-0.5)^2+(y-0.5)^2+(z-0.5)^2 )\right) , \\
				 v^0(\bm{x})&=6 \exp\left( -50 ((x-0.5)^2+(y-0.5)^2+(z-0.5)^2 )\right) , \\
				 c^0(\bm{x})&=4 \exp\left( -20 ((x-0.5)^2+(y-0.5)^2+(z-0.5)^2 )\right)  .
			\end{aligned}
		\end{equation*}
\end{example}

In this test, we set $M=40$ grid points in each spatial direction and choose the time stepsize $\Delta t = 2.0\times10^{-3}$ as specified in \eqref{eq:adap}. Fig. \ref{fig:eg4:mass} presents the simulation results on uniform grids (i.e., $\mu=0$), including the time evolution of the extrema of $u$, $v$ and $c$, the total mass of $u$ and $v$, and the iteration number of the semismooth Newton solver used in the $L^2$ projection \eqref{model:KS:shemeI2}. Fig. \ref{fig:eg4:mass:non} displays the corresponding results on non-uniform grids with $\mu=0.1$. In addition, the time evolution of the original discrete energy on both uniform (i.e., $\mu=0$) and non-uniform (i.e., $\mu=0.1$) grids is shown in Fig. \ref{fig:eg4:energy}, which clearly illustrates its strict dissipative behavior. Furthermore, we also present slices at $x = 0.5$, $y = 0.5$ and $z = 0.5$ of the cell densities $u,v$ and the chemoattractant concentration $c$ at four different time instants $t = 0,\, 8.2 \times 10^{-3},\, 4.0 \times 10^{-2},\, 0.2$ in Fig. \ref{fig:3D}. Overall, the conclusions drawn here are largely consistent with those from the 2D Example \ref{exam:less8pi_1}.
\begin{figure}[!htbp]
 	\centering
 	\includegraphics[width=0.32\linewidth]{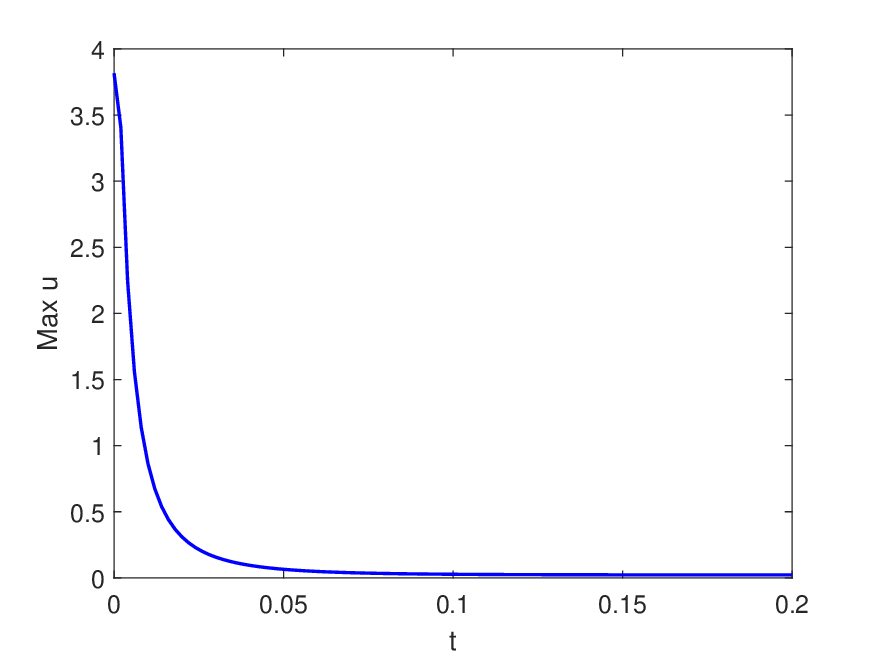}
 	\includegraphics[width=0.32\linewidth]{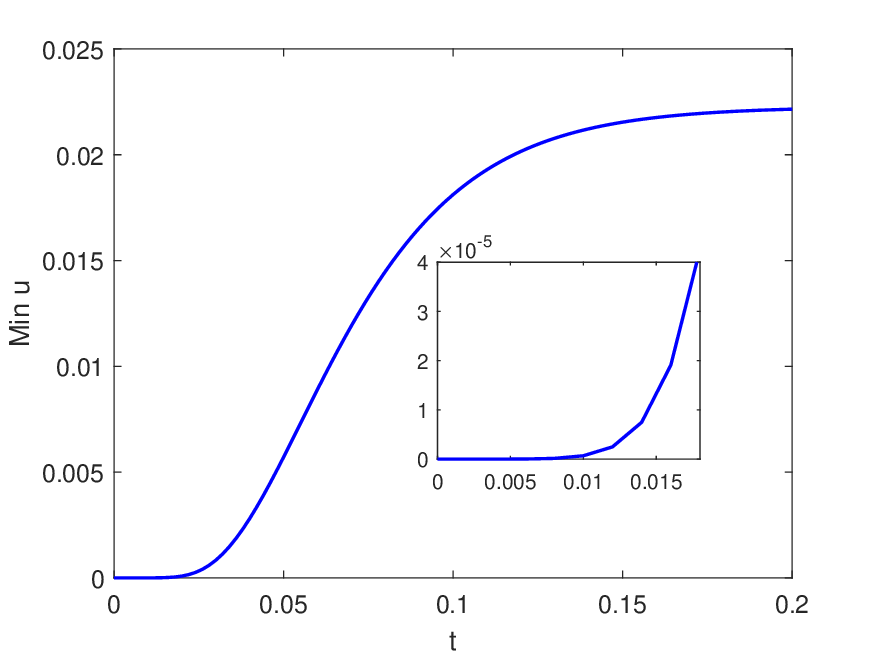}
 	\includegraphics[width=0.32\linewidth]{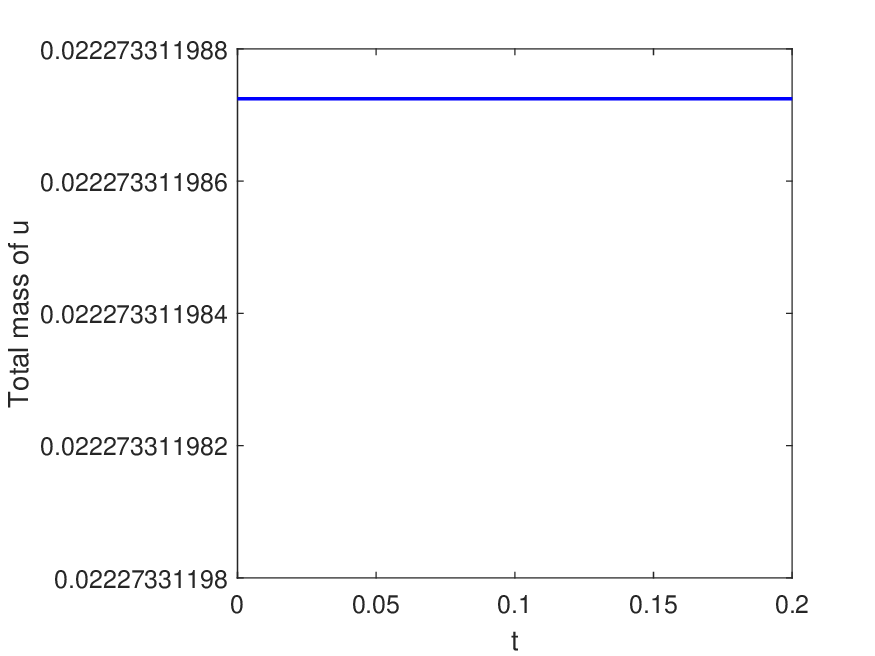}
 	\includegraphics[width=0.32\linewidth]{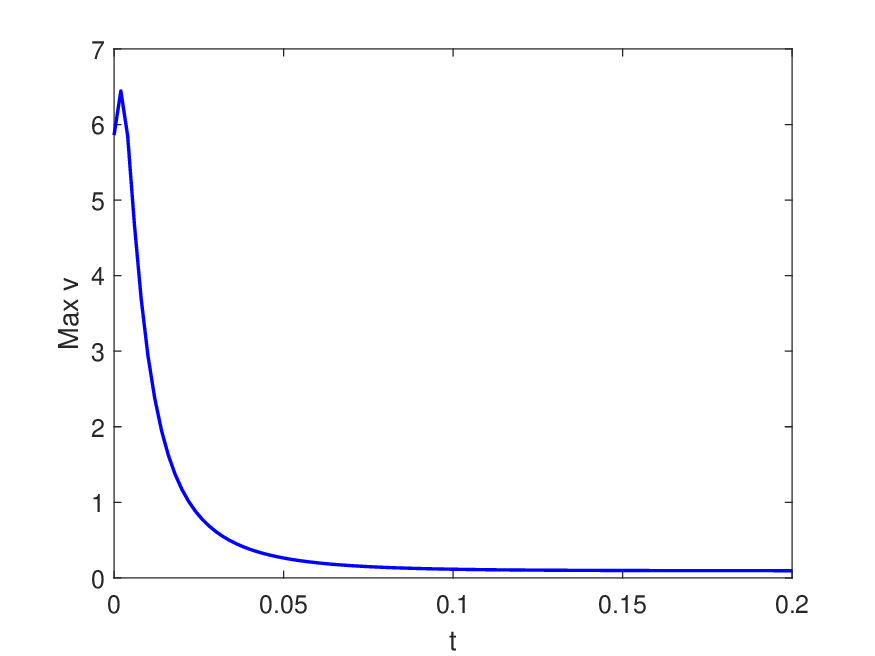}
 	\includegraphics[width=0.32\linewidth]{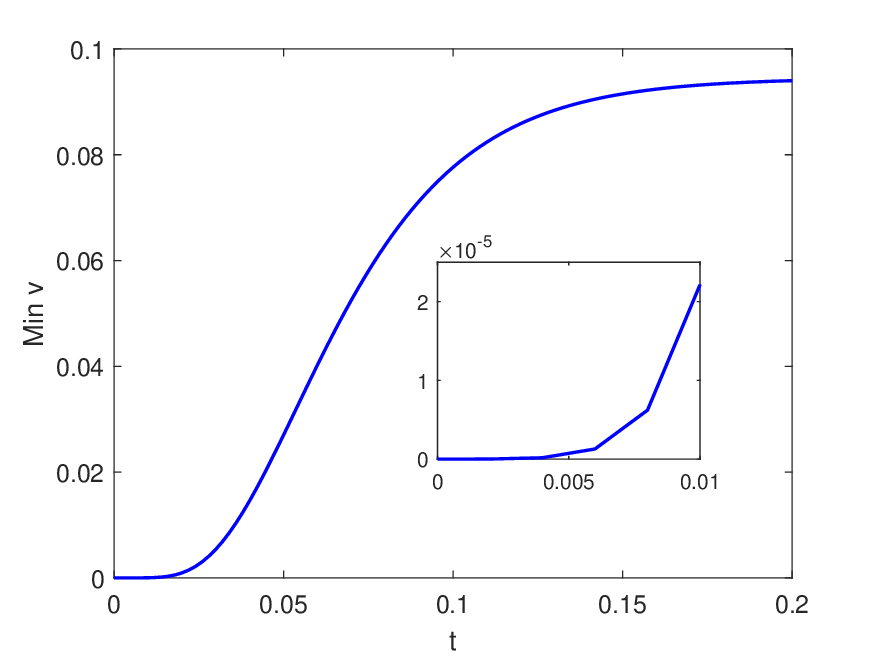}
 	\includegraphics[width=0.32\linewidth]{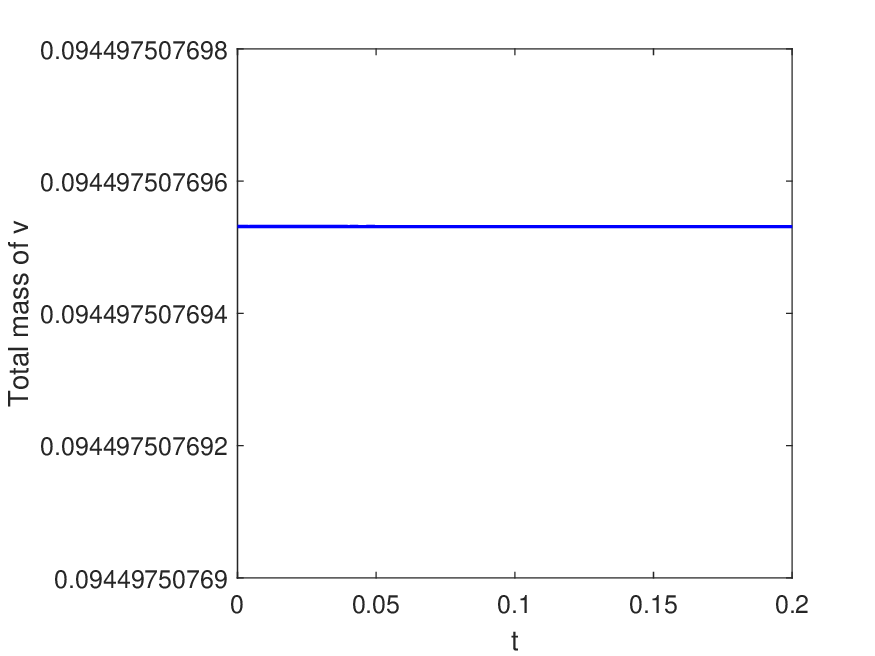}
 	\includegraphics[width=0.32\linewidth]{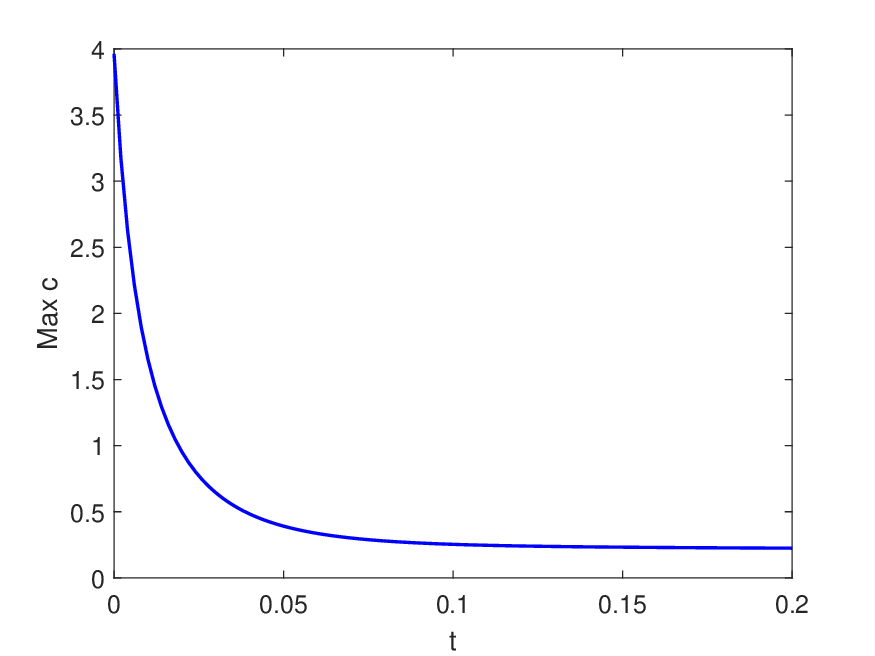}
 	\includegraphics[width=0.32\linewidth]{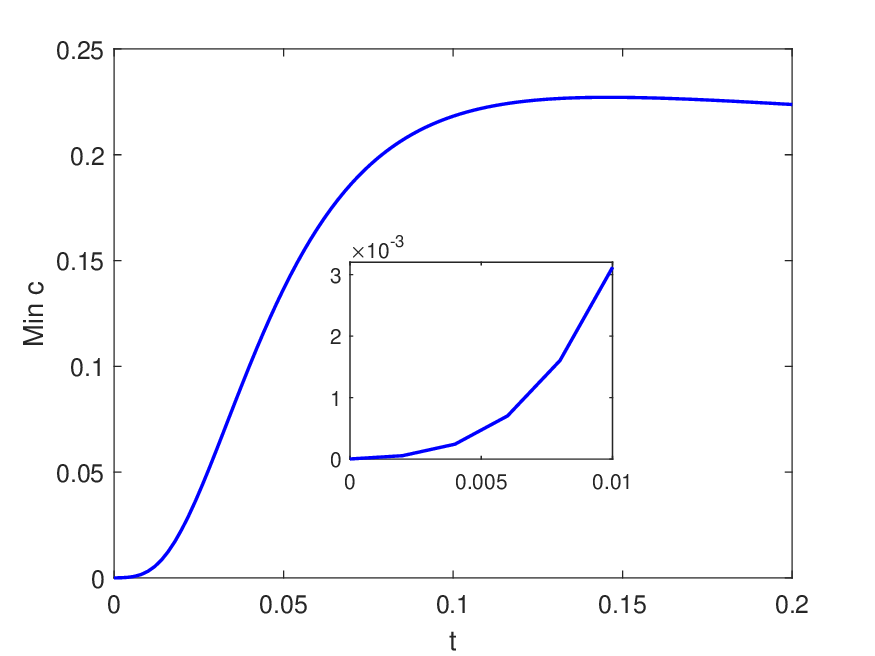}
 	\includegraphics[width=0.32\linewidth]{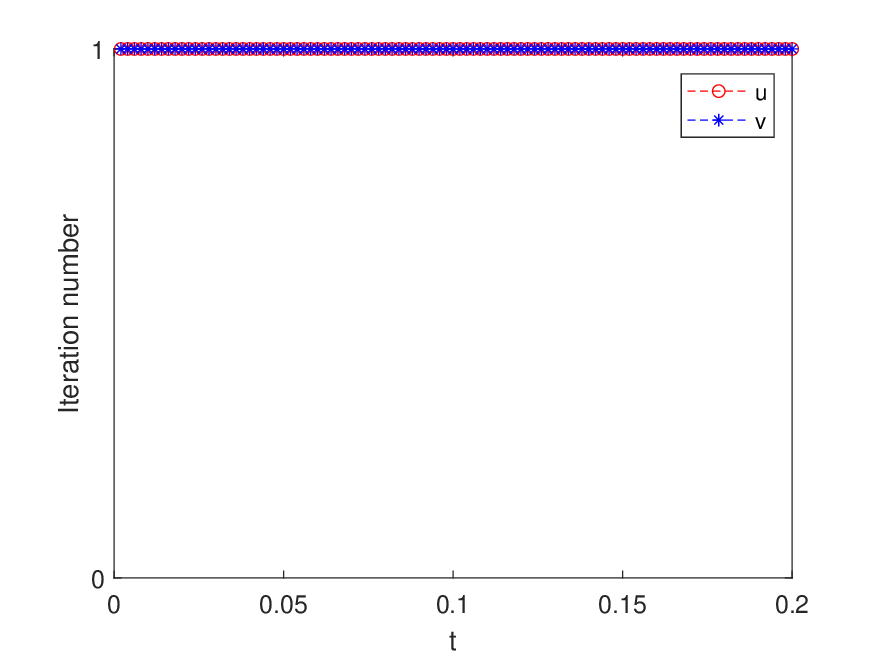}
 	\caption{Evolution of the extrema of $u$, $v$, $c$, and the total mass and iteration numbers for $u$ and $v$ ($\mu=0$) for Example \ref{exam:bu:3D}.}
 	\label{fig:eg4:mass}
 \end{figure}
 \begin{figure}[!h]
 	\centering
 	\includegraphics[width=0.32\linewidth]{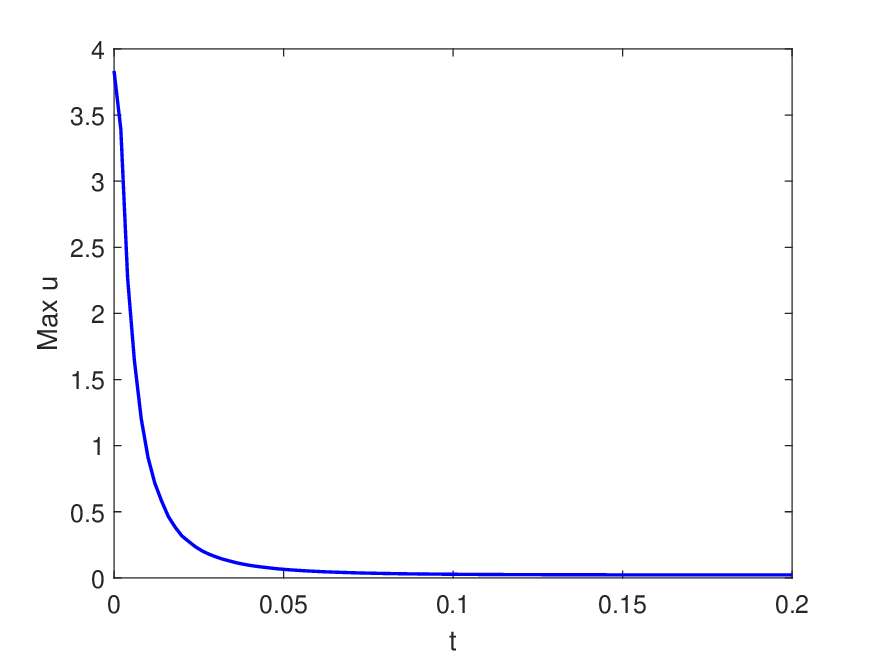}
 	\includegraphics[width=0.32\linewidth]{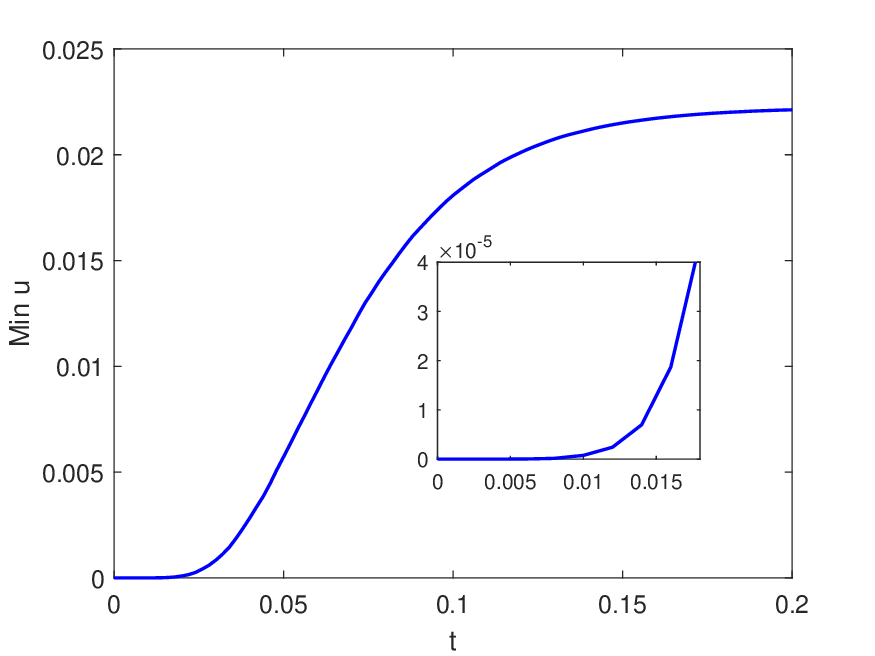}
 	\includegraphics[width=0.32\linewidth]{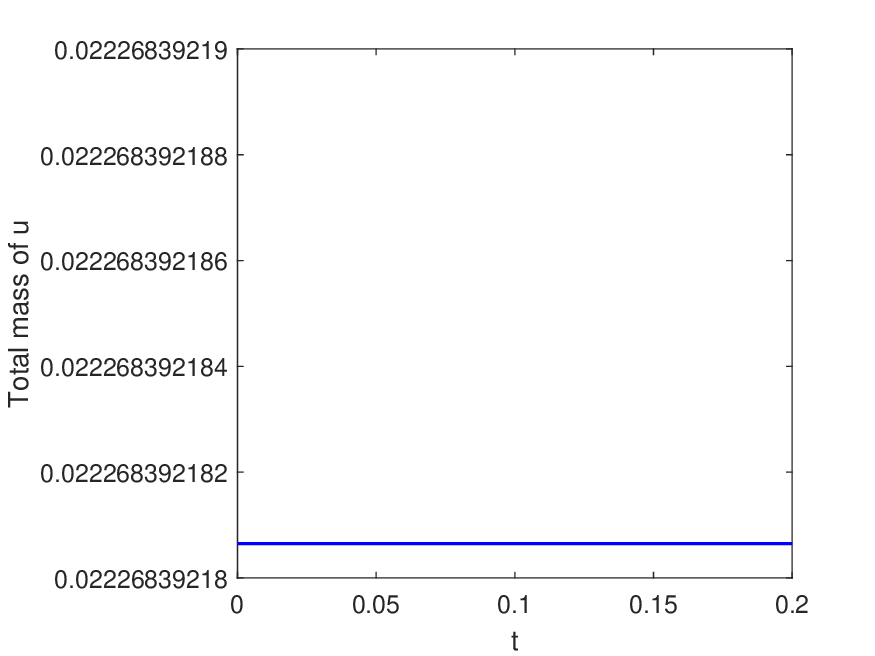}
 	\includegraphics[width=0.32\linewidth]{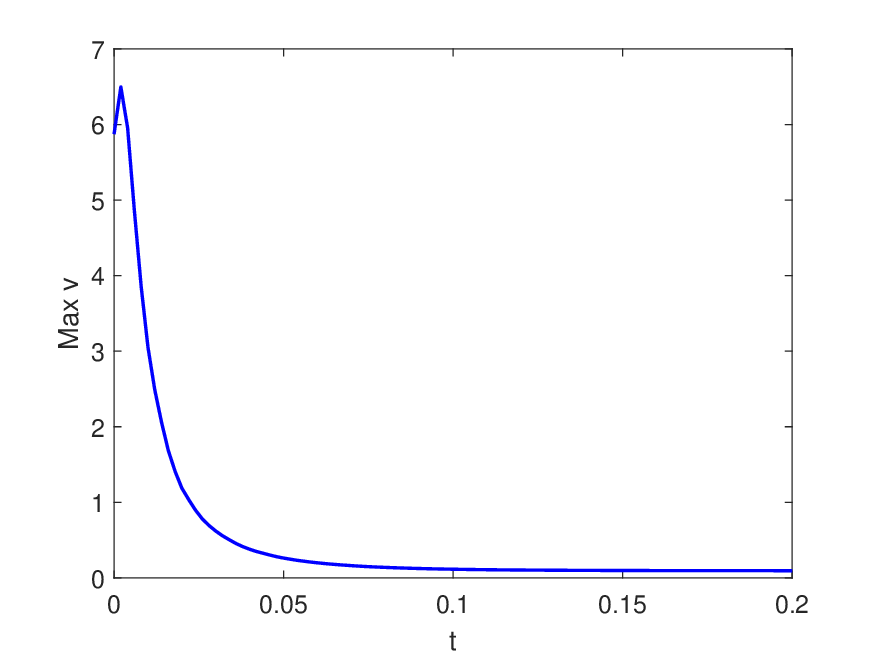}
 	\includegraphics[width=0.32\linewidth]{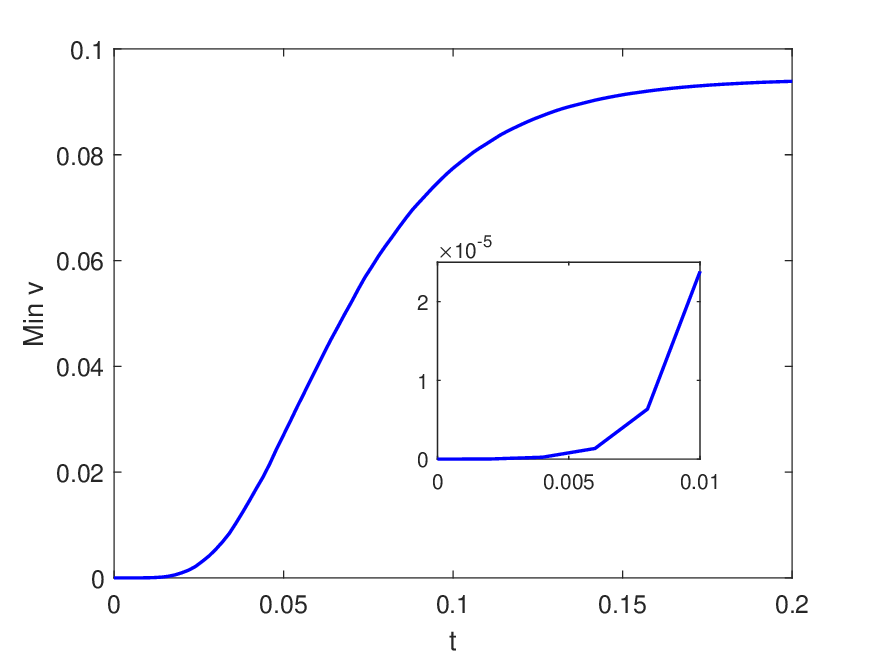}
 	\includegraphics[width=0.32\linewidth]{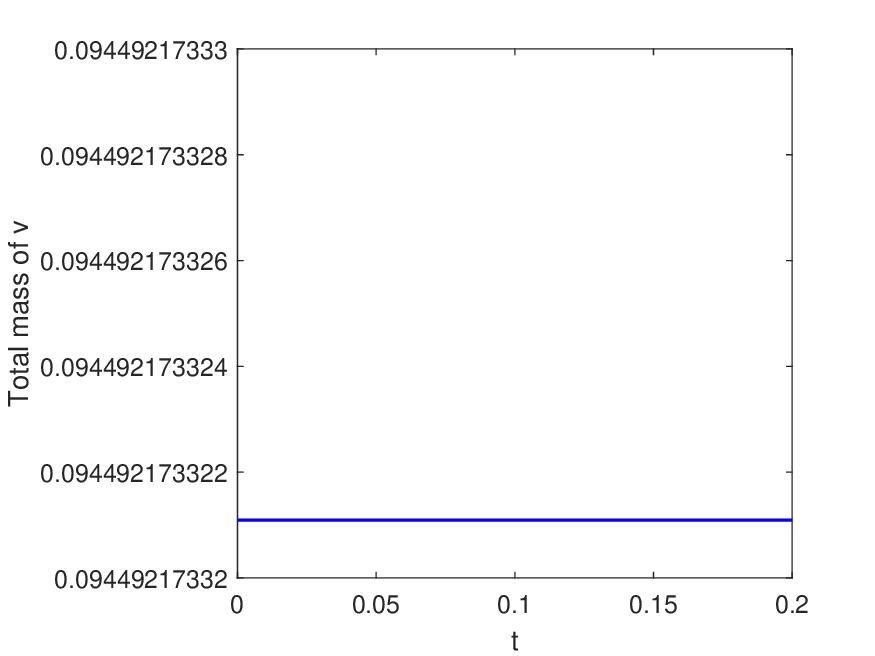}
 	\includegraphics[width=0.32\linewidth]{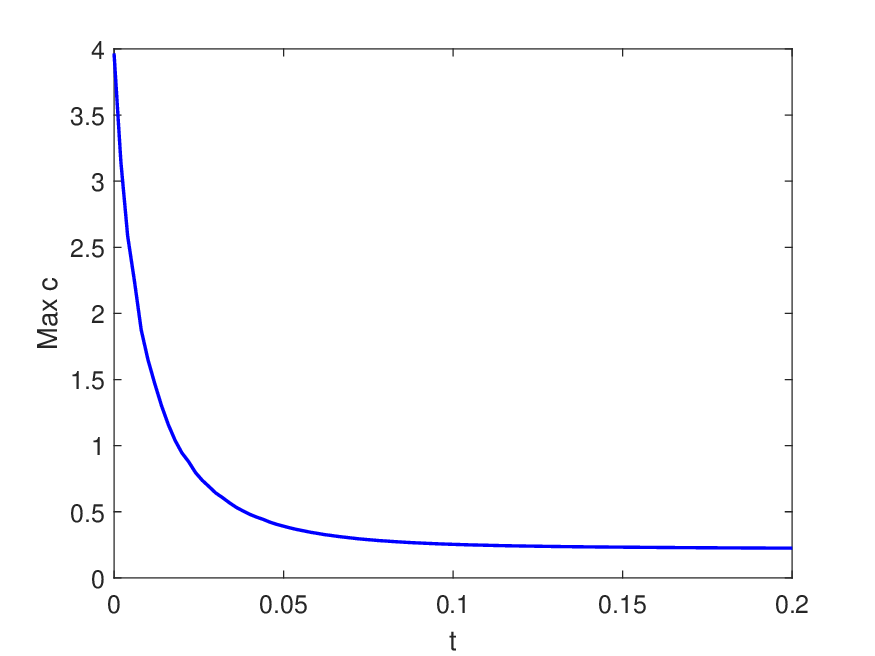}
 	\includegraphics[width=0.32\linewidth]{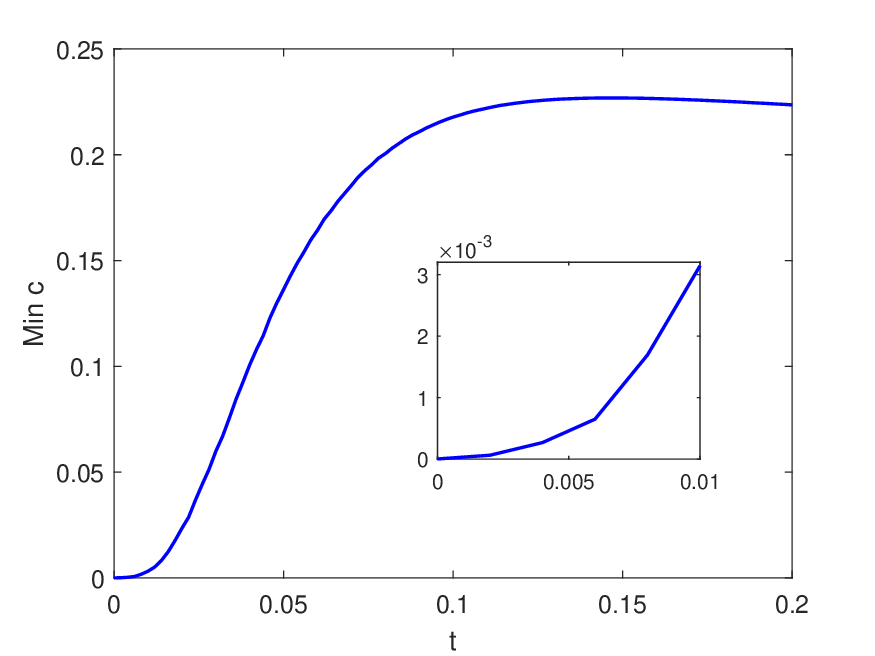}
 	\includegraphics[width=0.32\linewidth]{figure/iter_smooth_eg4.eps}
 	\caption{Evolution of the extrema of $u$, $v$, $c$, and the total mass and iteration numbers for $u$ and $v$ ($\mu=0.1$)  for Example \ref{exam:bu:3D}.}
 	\label{fig:eg4:mass:non}
 \end{figure}
 \begin{figure}[!ht]
	\centering
	\includegraphics[width=0.45\linewidth]{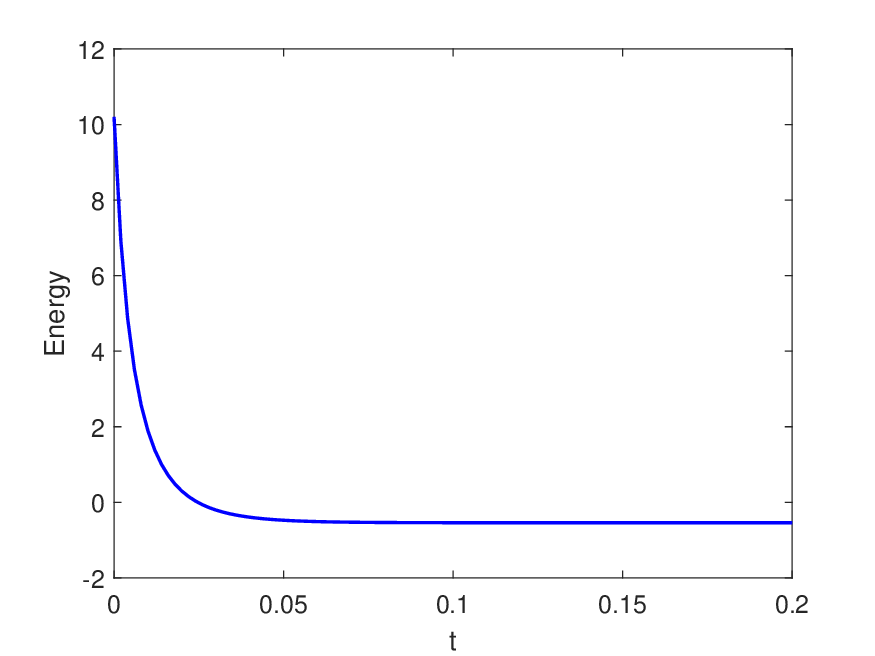}
    \includegraphics[width=0.45\linewidth]{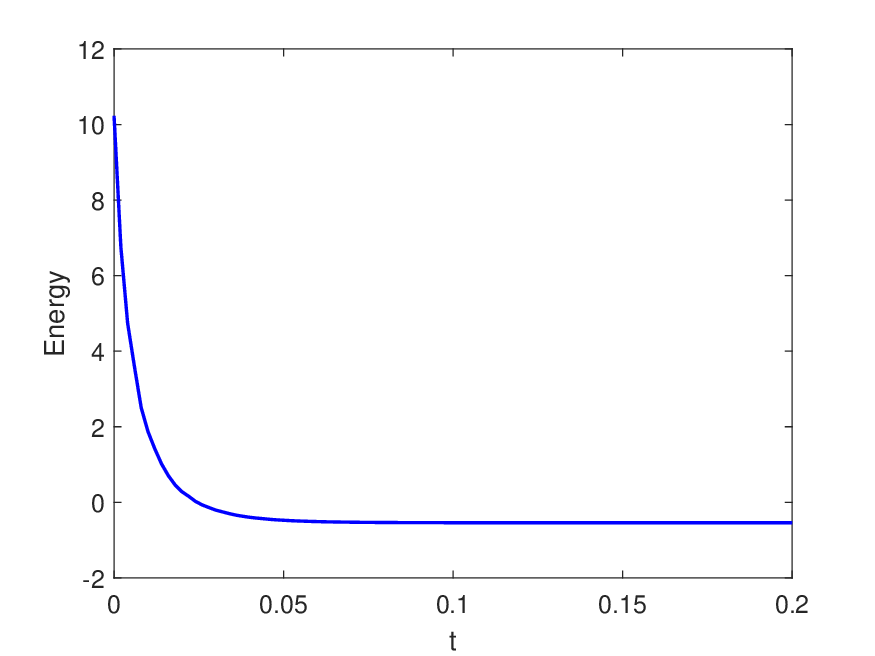}
	\caption{Evolution of the energy $\mu=0$ (left) and $\mu=0.1$ (right) for Example \ref{exam:bu:3D}.}	\label{fig:eg4:energy}
\end{figure}
 \begin{figure}[!htbp]
	\centering
	\includegraphics[width=0.24\linewidth]{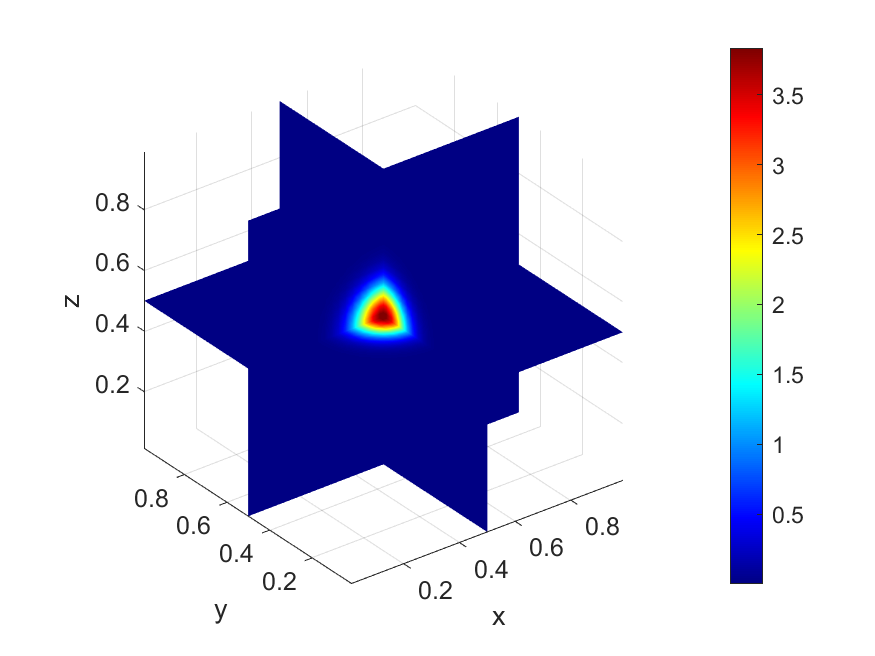}
	\includegraphics[width=0.24\linewidth]{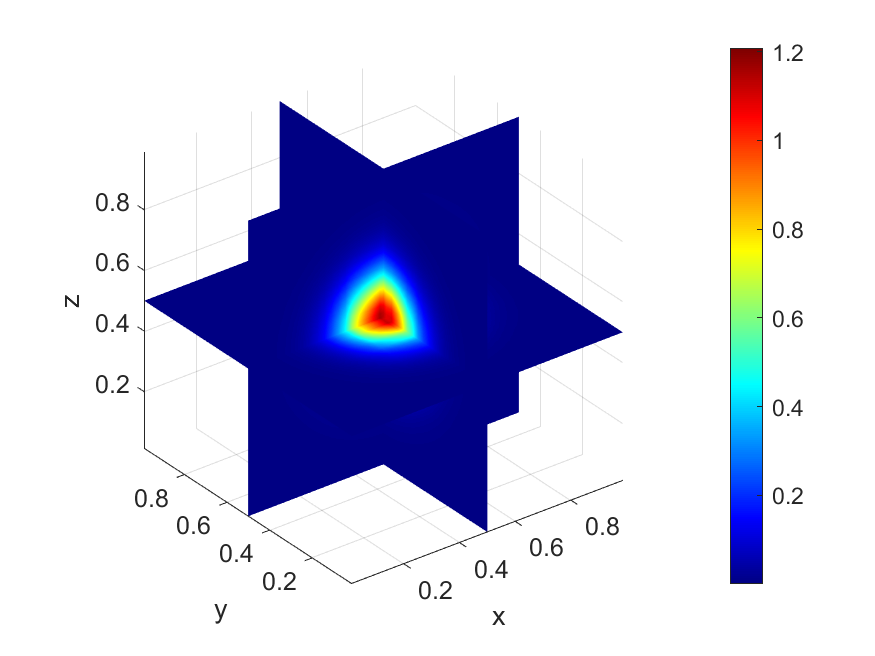}
	\includegraphics[width=0.24\linewidth]{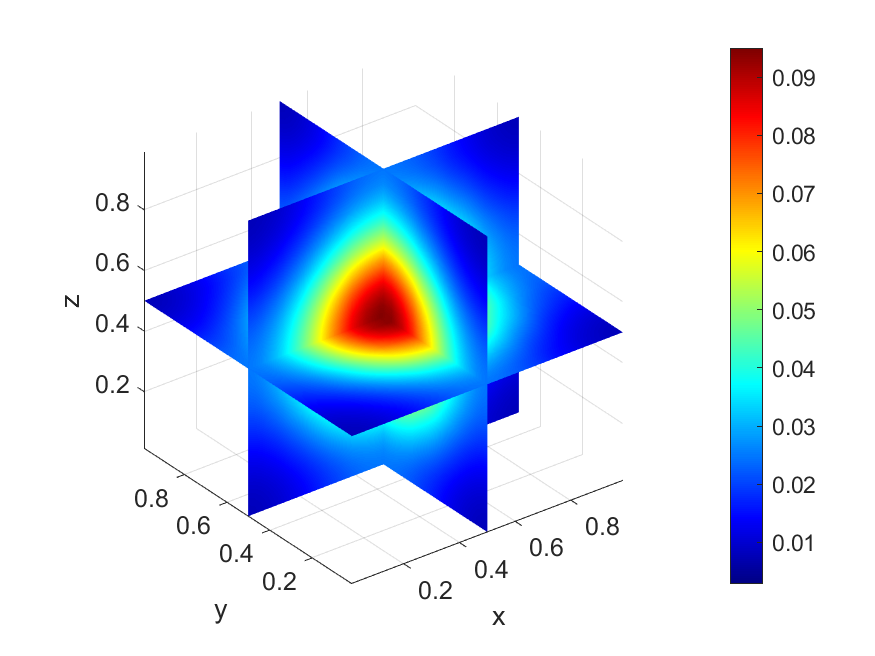}
	\includegraphics[width=0.24\linewidth]{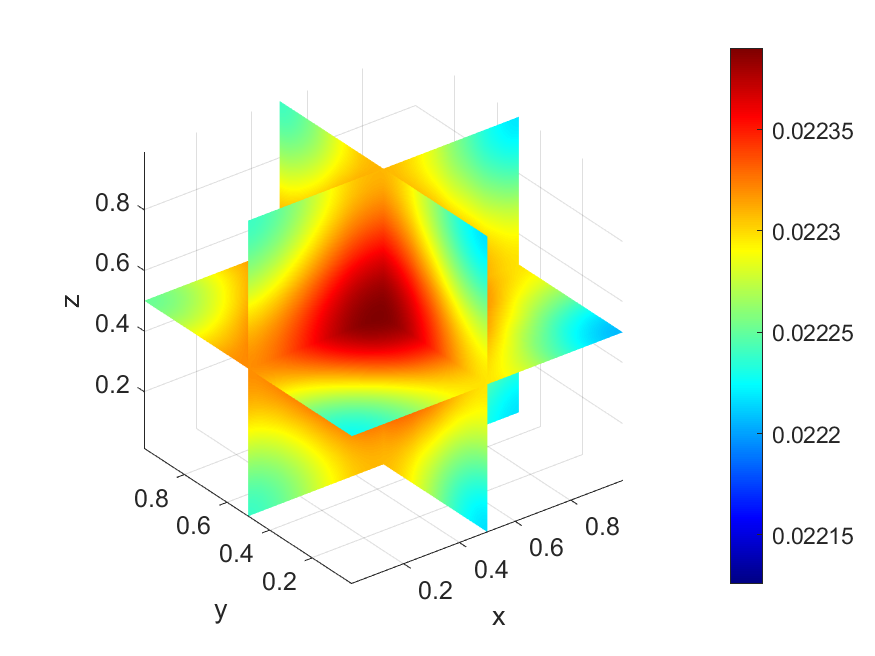}
	\includegraphics[width=0.24\linewidth]{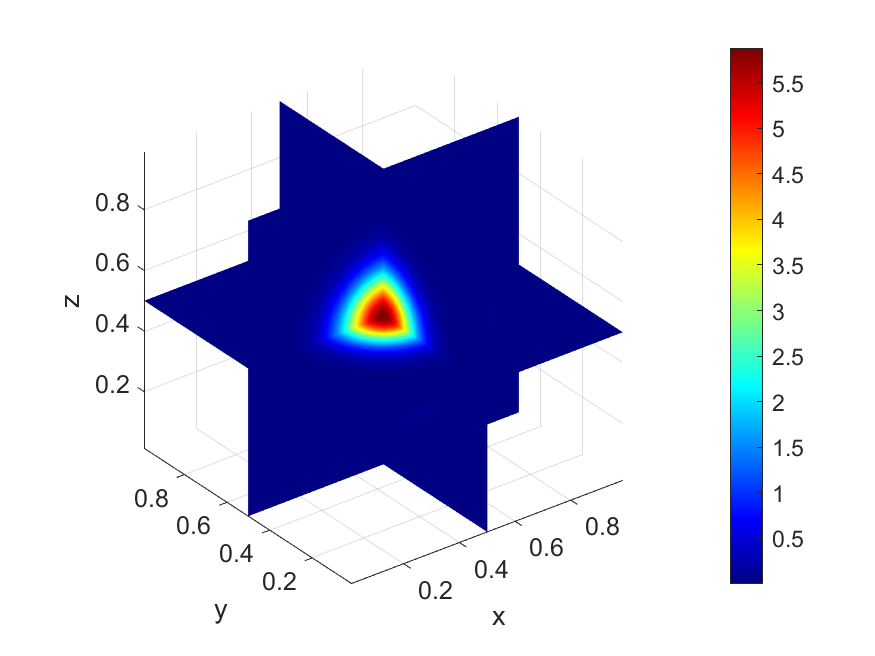}
	\includegraphics[width=0.24\linewidth]{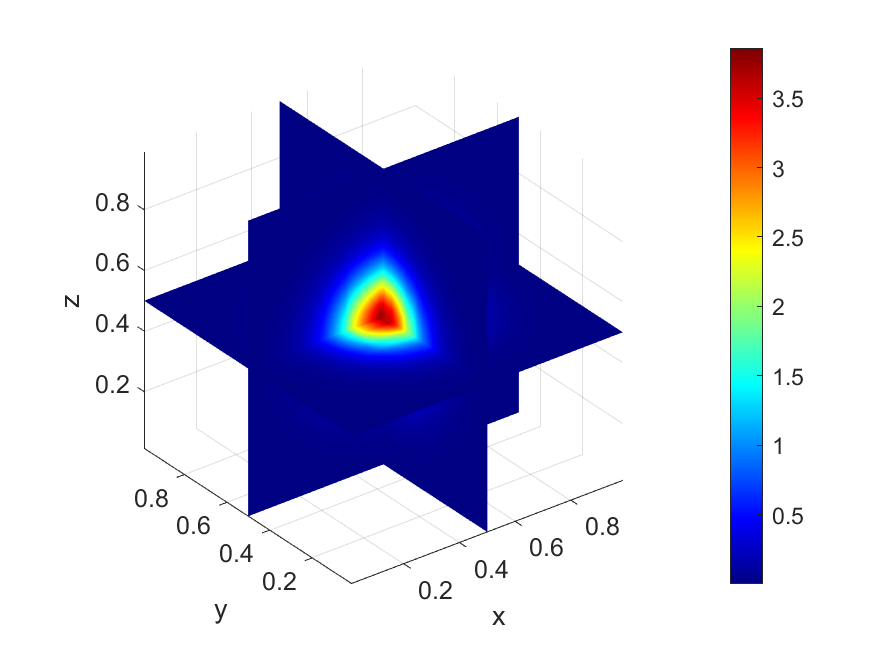}
	\includegraphics[width=0.24\linewidth]{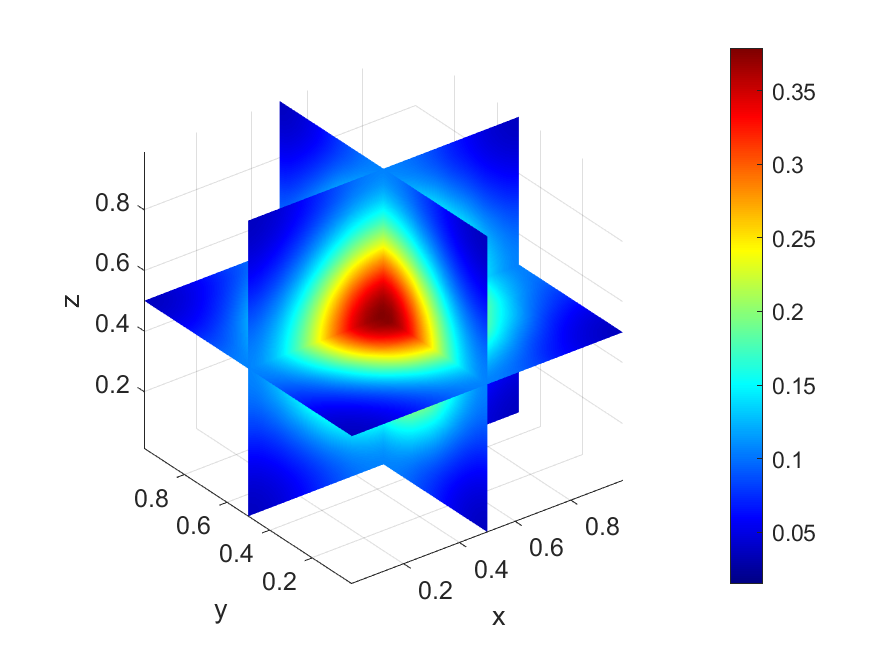}
	\includegraphics[width=0.24\linewidth]{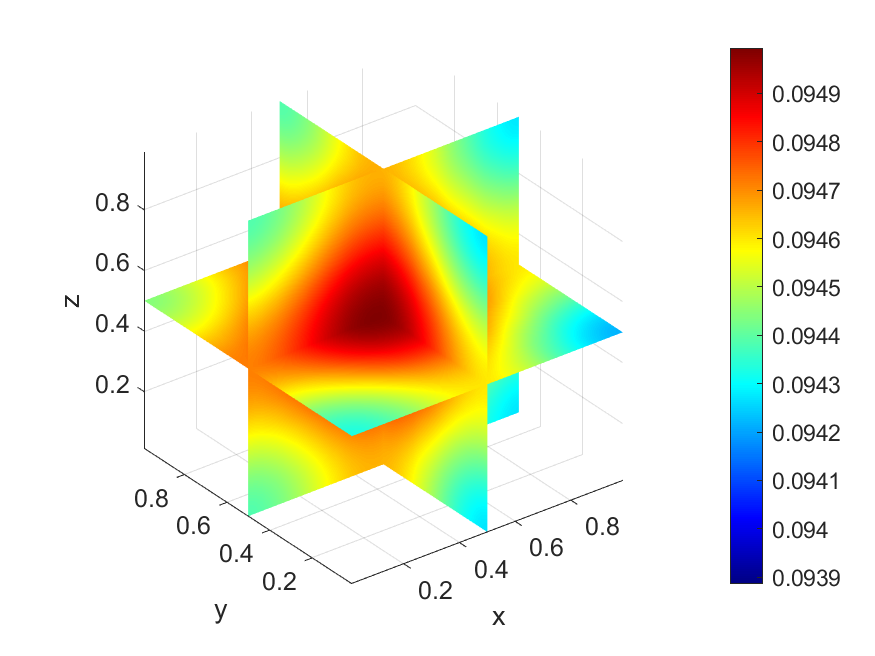}	\includegraphics[width=0.24\linewidth]{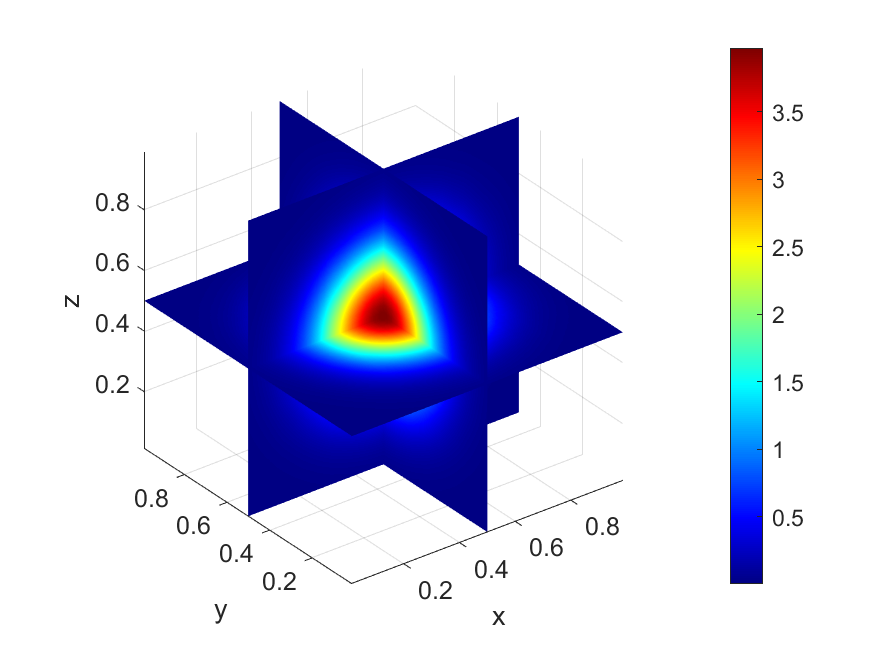}
	\includegraphics[width=0.24\linewidth]{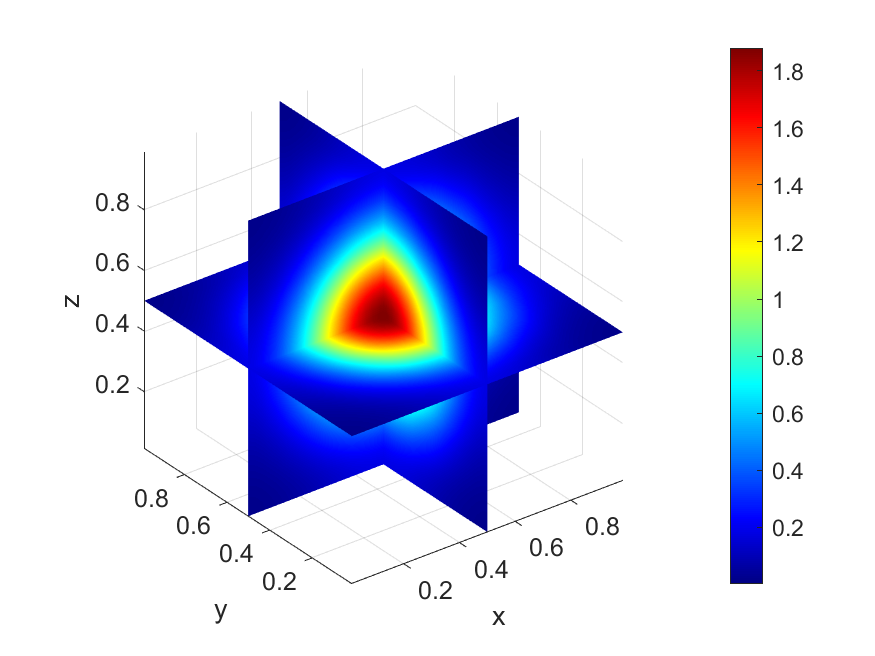}
	\includegraphics[width=0.24\linewidth]{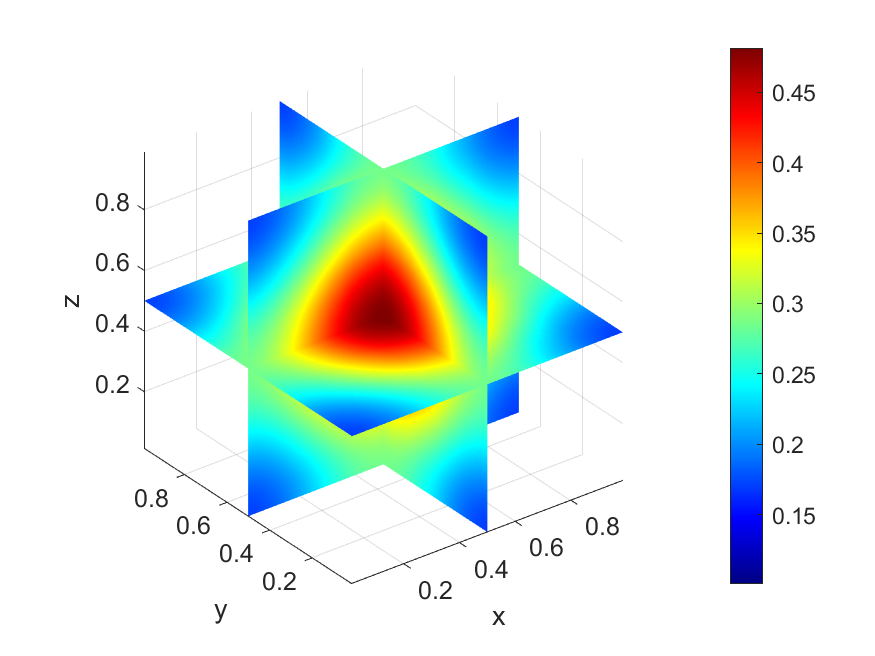}
	\includegraphics[width=0.24\linewidth]{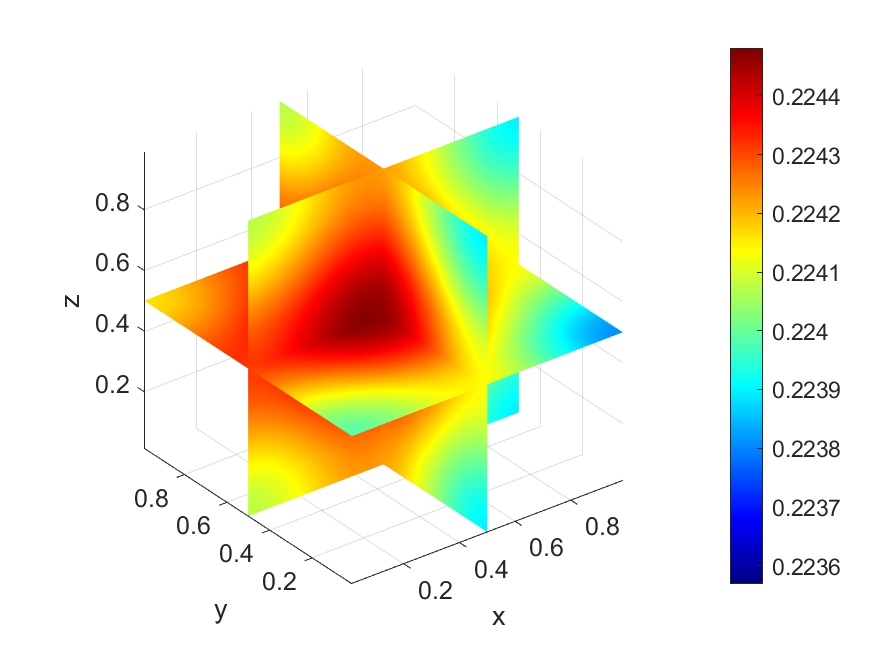}
  \caption{ Slices of $u,v,c$ (from top to bottom)  at time instants $t=0,\ 8.2\times 10^{-3},\ 4.0\times 10^{-2},\ 0.2$ (from left to right) for Example \ref{exam:bu:3D}. }
\label{fig:3D}
\end{figure}

\section{Conclusion}\label{sec:conclusion}
This paper has introduced a fully decoupled, linearly implicit, positivity-preserving, and time-staggered BCFD prediction-then-projection scheme for the multi-species Keller--Segel chemotaxis system. We demonstrate several key features of the proposed scheme:
\begin{itemize}
    \item[(i)] The proposed scheme ensures unconditional positivity and mass conservation for the cell densities at the discrete level. Moreover, under a suitable time-step condition,  the non-negativity of the chemoattractant concentration is also rigorously established.
    
    \item[(ii)] The unique solvability and optimal-order error estimates of the PP-MC-PBCFD scheme are rigorously established using the mathematical induction method and the discrete energy analysis approach, where the estimate established in Lemma \ref{lem:proj:e} for the discrete $L^2$ projection \eqref{model:KS:shemeI2} plays a crucial role in the error analysis. As proved in Theorem \ref{thm:coverg}, the cell densities $u$ and $v$ achieve second-order accuracy in both time  and space in the discrete $L^2$ norm, while the chemoattractant concentration $c$ attains the same order in the discrete $H^1$ norm.
    
    \item[(iii)] The use of variable time stepsize and time-staggered discretization fully decouples the solutions of the multi-species cell density variables and the   chemoattractant concentration variable, while also enabling linearization, thereby significantly enhancing computational efficiency.
    
    \item [(iv)] An adaptive time-stepping strategy \eqref{tau_adap}, driven by the numerical solution evolution behavior, together with the time-staggered BCFD method on non-uniform spatial grids, effectively and accurately captures the blow-up phenomenon.
\end{itemize}

Furthermore, extensive numerical experiments have validated the accuracy, positive-preserving and mass conservation properties of the proposed scheme for the multi-species Keller--Segel chemotaxis system,  while also demonstrating its reliability in simulating the blow-up phenomenon. As mentioned earlier, the model inherently satisfies the physical energy dissipation law \cite{HS'21,WLF'25,HZ'23,LWZ'18,SX'20,AGR'23}, which has been confirmed by our numerical results. However, a theoretical guarantee of this energy dissipation property is unavailable and requires further investigation.

\section*{CRediT authorship contribution statement}
	\textbf{Ao Zhang}: Methodology, Formal analysis, Software, Writing-Original draft.
	\textbf{Bingyin Zhang}: Methodology, Formal analysis, Writing-Original draft.
	\textbf{Hongfei Fu}: Conceptualization, Supervision, Writing-Reviewing and Editing, Methodology, Funding acquisition.
		
\section*{Declaration of competing interest} The authors declare that they have no competing interests.
			
\section*{Data availability}   
Data are available upon reasonable request.

\section*{Acknowledgements}
This work was supported in part by the Natural Science Foundation of Shandong Province (No. ZR2024MA023) and by the National Natural Science Foundation of China (No. 12131014).

\bibliographystyle{elsarticle-num}
\bibliography{reference.bib}

\end{document}